\documentclass[smallextended,draft]{svjour3}       
\smartqed  
\usepackage[final]{graphicx}
\usepackage{amsmath,amssymb}
\usepackage{amsfonts}
\usepackage{mathrsfs}
\usepackage{stmaryrd}
\usepackage[linesnumbered,algoruled,boxed]{algorithm2e}
\usepackage{psfrag}
\usepackage{subfig}
\usepackage[square,sort,comma,numbers]{natbib}
\usepackage{pgf,tikz,pgfplots}
\usepackage{enumitem}   

\journalname{Numerische Mathematik}
\begin{document}
\title{A posteriori error estimates in $\mathbf{W}^{1,p} \times \mathrm{L}^p$ spaces for the Stokes system with Dirac measures 
\thanks{FL is partially supported by CONICYT through  FONDECYT Postdoctoral project 3190204. EO is partially supported by CONICYT through FONDECYT project 11180193. DQ is partially supported by USM through Programa de Incentivos a la Iniciaci\'on Cient\'ifica (PIIC)}}

\titlerunning{Error estimates for the Stokes system with Dirac measures}        

\author{Francisco Fuica \and Felipe Lepe \and Enrique Ot\'arola \and Daniel Quero
}


\institute{F. Fuica \at
              Departamento de Matem\'atica, Universidad T\'ecnica Federico Santa Mar\'ia, Av. Espa\~{n}a 1680, Valpara\'iso, Chile. \\
              \email{francisco.fuica@sansano.usm.cl}          
           \and
           F. Lepe \at
 			 Departamento de Matem\'atica, Universidad T\'ecnica Federico Santa Mar\'ia, Av. Espa\~{n}a 1680, Valpara\'iso, Chile. \\
 			 \email{felipe.lepe@usm.cl}  
 		   \and
           E. Ot\'arola \at
 			 Departamento de Matem\'atica, Universidad T\'ecnica Federico Santa Mar\'ia, Av. Espa\~{n}a 1680, Valpara\'iso, Chile. \\
 			 \email{enrique.otarola@usm.cl} 
 		   \and
           D. Quero \at
 			 Departamento de Matem\'atica, Universidad T\'ecnica Federico Santa Mar\'ia, Av. Espa\~{n}a 1680, Valpara\'iso, Chile. \\
 			 \email{daniel.quero@alumnos.usm.cl} 
}

\date{Received: date / Accepted: date}

\maketitle

\begin{abstract}
We design and analyze a posteriori error estimators for the Stokes system with singular sources in suitable $\mathbf{W}^{1,p}\times \mathrm{L}^p$ spaces. We consider classical low-order inf-sup stable and stabilized finite element discretizations. We prove, in two and three dimensional Lipschitz, but not necessarily convex polytopal domains, that the devised error estimators are reliable and locally efficient. On the basis of the devised error estimators, we design a simple adaptive strategy that yields optimal experimental rates of convergence for the numerical examples that we perform.
\end{abstract}

\section{Introduction}\label{sec:intro}
For $d\in\{2,3\}$, we let $\Omega$ be an open and bounded polytopal domain in $\mathbb{R}^d$ with Lipschitz boundary $\partial\Omega$. The purpose of this work is the design and analysis of a posteriori error estimators for classical low-order inf-sup stable and stabilized finite element approximations of the Stokes problem
\begin{equation}\label{def:Stokes_singular_rhs}
\left\{
\begin{array}{rcll}
-\Delta\boldsymbol{u}+\nabla \pi & = & \boldsymbol{f}\delta_{x_0} & \text{ in } \quad \Omega, \\
\text{div }\boldsymbol{u} & = & 0 & \text{ in } \quad \Omega, \\
\boldsymbol{u} & = & \mathbf{0} & \text{ on } \quad \partial\Omega,
\end{array}
\right.
\end{equation}where $\delta_{x_0}$ corresponds to the Dirac delta supported at the interior point $x_0 \in \Omega$ and $\boldsymbol{f}\in\mathbb{R}^d$. As it is customary in fluid mechanics, $\boldsymbol{u}$ represents the velocity of the fluid, $\pi$ the pressure and $\boldsymbol{f}\delta_{x_0} $ is an externally applied force. Notice that, for simplicity, we have taken the viscosity to be equal to one. An instance of \eqref{def:Stokes_singular_rhs} appears in the modeling of active thin structures \cite{LACOUTURE2015187,FULFORD1986381}; there the right hand side is a linear combination of Dirac deltas supported at interior points of $\Omega$. We also mention other applications such as the use of flagella by sessile organisms to generate feeding currents \cite{higdon79}, modeling the flow of a fluid through structures with singular sources \cite{liron78,hasimoto59}, slender body theories \cite{stephen81}, improved models for the movement by cilia \cite{blake_1972}, and optimal control of fluid flows \cite{2019arXiv190711096F,2018arXiv181002415A}.


When the body force acting on the fluid and the mass production rate are smooth, the study of solution techniques for the Stokes and related models within a standard Hilbert space--based setting is well understood \cite{MR851383,MR2050138}. However, recent models have emerged where the motion of an incompressible fluid is described by problem \eqref{def:Stokes_singular_rhs} or a small variation of it.  Due to the singular nature of the body force $\boldsymbol{f}\delta_{x_0}$, the problem must be understood in a completely different setting where the analysis of approximation techniques is scarce. Since $\Omega$ is a Lipschitz polytope, the fact that $\delta_{x_0}\in \mathrm{W}^{-1,p}(\Omega)$, with $p \in (1,d/(d-1))$, yields the existence of a unique solution $(\boldsymbol{u},\pi)\in \mathbf{W}^{1,p}(\Omega)\times \mathrm{L}^p(\Omega)/\mathbb{R}$ with $p\in(2d/(d+1)-\varepsilon,d/(d-1))$ \cite{Larsson_Svensson,MR2987056,MR1386766}. Here, $\varepsilon$ denotes a positive constant that depends on $\Omega$. For a complete treatment of boundary value problems for the Stokes system on Lipschitz domains we refer the reader to \cite{MR2987056}, where the authors prove optimal well--posedness results in all space dimensions and for all major types of boundary conditions.

Regarding the design and analysis of solution techniques for problem \eqref{def:Stokes_singular_rhs}, and to the best of our knowledge, the first work that proposed an scheme is \cite{LACOUTURE2015187}. Later, the authors of \cite{MR3854357} derived quasi--optimal local convergence results in $\mathbf{H}^1 \times \mathrm{L}^2$.  The authors operated under the assumption that the underlying domain $\Omega \subset \mathbb{R}^2$ is an open and bounded $C^{\infty}$ domain, or a square, and considered finite element discretizations based on the mini element and Taylor--Hood approximations. The error is analyzed on a subdomain which does not contain the singularity of the involved solution. On the other hand, in view of the fact that there is a Muckenhoupt weight $\omega$ related  to the distance to $x_0$ such that $\delta_{x_0} \in \mathrm{H}^{-1}(\omega,\Omega)$, the authors of \cite{2019arXiv190500476D,MR3892359} have operated within a weighted Sobolev space setting and derived a priori and a posteriori error estimates for classical low--order inf--sup stable finite element approximations.

Since $\delta_{x_{0}}$ is very singular, it is not expected for the pair $(\boldsymbol{u},\pi)$, solution to \eqref{def:Stokes_singular_rhs}, to have any global regularity properties beyond those 
inherited from the well--posedness of the problem.
As a consequence, optimal error estimates for classical low--order inf--sup stable finite element approximations, such as the mini element and the lowest order Taylor--Hood element, cannot be expected. This motivates the design and analysis of adaptive finite element methods (AFEMs) for the efficient resolution of problem \eqref{def:Stokes_singular_rhs} since they are known to outperform classical FEM in practice and deliver optimal convergence rates when FEM cannot. AFEMs are a fundamental numerical tool in science and engineering that allow for the resolution of PDEs with relatively modest computational resources. An essential ingredient of an AFEM is an a posteriori error estimator. This is a computable quantity that depends on the discrete solution and data, and provides information about the local quality of the approximate solution. Therefore, it can be used for adaptive mesh refinement and coarsening, error control, and equidistribution of the computational effort. The a posteriori error analysis for linear second-order elliptic boundary value problems has attained a mature understanding \cite{MR1885308,MR2648380,MR3059294}.

In contrast to the well-established theory for linear elliptic PDEs with smooth data, the a posteriori error analysis for finite element approximations of problems with singular forcing has not yet been fully understood. The main source of difficulty is the reduced regularity properties exhibited by the underlying solution. Within this context, the first work that provides an a posteriori error analysis for a finite element approximation of a Poisson problem with a Dirac delta as a forcing term is \cite{MR2262756}. The authors of this work utilize suitable $\mathrm{W}^{1,p}$--norms and design, on a two dimensional setting, residual--type a posteriori error estimators. The devised estimators are proven to be reliable and locally efficient. We would like to also mention reference \cite{MR3237857} where the authors consider a posteriori error estimates for an electrostatics problem with a current dipole source and extend some of the results of \cite{MR2262756} to the three dimensional case. This is a singular problem, since the current dipole model involves first--order derivatives of a Dirac delta measure.

To the best of our knowledge, the only work that provides an advance concerning the a posteriori error analysis for the Stokes system \eqref{def:Stokes_singular_rhs} is \cite{MR3892359}. In such a work, the authors propose a posteriori error estimators for classical low--order inf--sup stable and stabilized finite element approximations of the Stokes problem \eqref{def:Stokes_singular_rhs} in two and three dimensional Lipschitz polytopal domains. The authors operate within the setting of Muckenhoupt weighted Sobolev spaces and prove that the devised error estimators are reliable and locally efficient. In contrast, in this work we operate under a complete different setting; we make use of the fact that, since $\Omega \subset \mathbb{R}^d$ is Lipschitz $(d \in \{2,3\})$, there exists $\varepsilon >0$ such that $(\boldsymbol{u},\pi)\in \mathbf{W}^{1,p}(\Omega)\times \mathrm{L}^p(\Omega)/\mathbb{R}$ with $p\in(2d/(d+1)-\varepsilon,d/(d-1))$ and devise a posteriori error estimator based on $\mathrm{L}^p$--norms. We consider the classical saddle point formulation of \eqref{def:Stokes_singular_rhs} and propose approximations based on popular low--order inf--sup stable and stabilized finite elements. For all these schemes, we devise a posteriori error estimators that are proven to be globally reliable and locally efficient when the approximation error is measured in suitable $\mathbf{W}^{1,p}(\Omega) \times \mathrm{L}^p(\Omega)$--norms. With the proposed estimators at hand, we also design simple adaptive strategies that yield optimal rates of convergence for the numerical examples that we perform.

The outline of this manuscript is as follows. In Section \ref{sec:notation_and_preliminaries} we introduce the notation and functional framework we shall work with. In Section \ref{sec:model_problem} we present a saddle point formulation for the Stokes system \eqref{def:Stokes_singular_rhs}. We also review the well--posedness of the system and state regularity properties of its solution. In Section \ref{sec:fem} we introduce classical low--order inf--sup stable finite element approximations of \eqref{def:Stokes_singular_rhs}. The core of our work is Section \ref{sec:a_posteriori_estimates}, where we design a posteriori error estimators and obtain global reliability and local efficiency results. We extend, in Section \ref{sec:stabilized}, the results obtained in Section \ref{sec:a_posteriori_estimates} to the case when stabilized finite element approximations are considered. Finally, in Section \ref{sec:numericos}, we report numerical tests, in two and three dimensions, that illustrate the theory and exhibit the performance of the devised estimators.


\section{Notation and preliminaries}
\label{sec:notation_and_preliminaries}
Let us set notation and describe the setting we shall operate with.

Throughout this work, $d\in\{2,3\}$ and $\Omega$ is an open and bounded polytopal domain  of $\mathbb{R}^{d}$ with Lipschitz boundary $\partial \Omega$. If $\mathscr{X}$ and $\mathscr{Y}$ are normed vector spaces, we write $\mathscr{X} \hookrightarrow \mathscr{Y}$ to denote that $\mathscr{X}$ is continuously embedded in $\mathscr{Y}$. We denote by $\mathscr{X}'$ and $\|\cdot\|_{\mathscr{X}}$ the dual and the norm of $\mathscr{X}$, respectively.

Given $p \in (1,\infty)$, we denote by $p\prime$ the real number such that $1/p + 1/p\prime = 1$, i.e., $p\prime = p/(p-1)$.

The relation $\texttt{a} \lesssim \texttt{b}$ indicates that $\texttt{a} \leq C \texttt{b}$, with a positive constant $C$ which is independent of $\texttt{a}$, $\texttt{b}$, and the size of the elements in the mesh. The value of $C$ might change at each occurrence.


\section{The model problem}\label{sec:model_problem}
We begin with a motivation for the use of the spaces $\mathbf{W}^{1,p}(\Omega)\times \mathrm{L}^p(\Omega)$ with $p< d/(d-1)$.

\subsection{Motivation}
\label{sec:motivation}
Let us assume that $\Omega = \mathbb{R}^d$. If this is the case, the results of \cite[Section IV.2]{MR2808162} yield the following asymptotic behavior, near the point $x_0\in\Omega$, for the solution $(\boldsymbol{u},\pi)$ to problem \eqref{def:Stokes_singular_rhs}:
\begin{equation}
\label{eq:asymptotic_behavior}
|\nabla\boldsymbol{u}(x)|\approx|x-x_0|^{1-d},
\qquad
|\pi(x)|\approx|x-x_0|^{1-d}.
\end{equation}
This immediately implies that  $(\boldsymbol{u},\pi) \notin \mathbf{H}_0^1(\Omega) \times \mathrm{L}^2(\Omega)$. More precisely, a simple computation based on \eqref{eq:asymptotic_behavior} suggests that $|\nabla\boldsymbol{u}| \in \mathrm{L}^p(\Omega)$ and $\pi \in \mathrm{L}^p(\Omega)$ provided $p<d/(d-1)$. 

\subsection{Saddle point formulation}
The motivation presented in Section \ref{sec:motivation} suggests to consider the following saddle point formulation for problem \eqref{def:Stokes_singular_rhs}: Find $(\boldsymbol{u},\pi)\in \mathbf{W}_0^{1,p}(\Omega)\times \mathrm{L}^p(\Omega)/\mathbb{R}$, with $p<d/(d-1)$, such that
\begin{equation}\label{eq:weak_Stokes_system}
\begin{array}{rcll}
 a(\boldsymbol{u},\boldsymbol{v})+b(\boldsymbol{v},\pi) & = &\langle\boldsymbol{f}\delta_{x_0},\boldsymbol{v}\rangle &\quad \forall\ \boldsymbol{v}\in\mathbf{W}^{1,p\prime}_0(\Omega),\\
b(\boldsymbol{u},q) & = & 0 &\quad \forall\ q\in\mathrm{L}^{p\prime}(\Omega)/\mathbb{R},
\end{array}
\end{equation}
where $\langle \cdot, \cdot \rangle$ denotes the duality pairing between the spaces $\mathbf{W}^{-1,p}(\Omega):= \mathbf{W}^{1,p\prime}_0(\Omega)'$ and $\mathbf{W}^{1,p\prime}_0(\Omega)$.
The bilinear forms  $a(\cdot,\cdot)$ and $b(\cdot,\cdot)$ are defined, respectively, by
\begin{equation*}
a(\boldsymbol{w},\boldsymbol{v}):=\int_{\Omega}\nabla\boldsymbol{w}:\nabla\boldsymbol{v},
\qquad 
b(\boldsymbol{v},q):=-\int_{\Omega}q\text{div}\boldsymbol{v}.
\end{equation*}

\subsection{Well--posedness}
Our heuristic argument suggests that the well--posedness of problem  \eqref{eq:weak_Stokes_system} is conditioned to $p < d/ ( d-1)$. To make matters precise, we introduce $\mathbf{G}_D$ as the Green operator for the inhomogeneous problem for the incompressible Stokes system with vanishing Dirichlet boundary conditions. That is, if $(\boldsymbol{\varphi},\xi)$ solves
\[
-\Delta \boldsymbol{\varphi} + \nabla \xi = \mathcal{F} \textrm{ in } \Omega,
\quad
\text{div} \boldsymbol{\varphi} = 0 \textrm{ in } \Omega,
\quad
\boldsymbol{\varphi} = \boldsymbol{0} \textrm{ on } \partial\Omega,
\]
then $\mathbf{G}_D \mathcal{F} :=  \boldsymbol{\varphi}$.  The regularity results of \cite[Corollary 1.7]{MR2987056} (with $\alpha = -1$ and $q=2$) guarantee that the operator $\mathbf{G}_D :\mathbf{W}^{-1,p}(\Omega)\rightarrow \mathbf{W}^{1,p}(\Omega)$ is bounded if
\begin{equation}\label{eq:rangep}
\frac{2d}{d+1}-\varepsilon<p<\frac{2d}{d-1}+\varepsilon
\end{equation}
for some $\varepsilon = \varepsilon(\Omega)>0$. As a consequence, for every bounded and Lipschitz domain $\Omega \subset \mathbb{R}^{3}$, there exist $p = p(\Omega) > 3$ such that $\mathbf{G}_D$ is well--defined and bounded. When $\Omega \subset \mathbb{R}^2$ is a bounded and Lipschitz domain, the same conclusion holds for some $p = p(\Omega)>4$.

We present the following result.

\begin{theorem}[well--posedness]\label{thm:well-posedness}
Let $d \in \{2,3\}$ and $\Omega \subset \mathbb{R}^d$ be an open and bounded Lipschitz polytope. There exists $\varepsilon= \varepsilon(\Omega) > 0$ such that, if $p \in (2d/(d+1) - \varepsilon, d/(d-1))$, then, problem \eqref{eq:weak_Stokes_system} is well--posed. In addition,
\[
\| \nabla \boldsymbol{u} \|_{\mathbf{L}^{p}(\Omega)} + \|  \pi \|_{\mathrm{L}^p(\Omega)} \lesssim | \boldsymbol{f} |  \| \delta_{x_0} \|_{\mathrm{W}^{-1,p}(\Omega)}, 
\]
where the hidden constant is independent of the solution and data.
\end{theorem}
\begin{proof} We proceed on the basis of two cases.
\begin{itemize}
\item[i)] $d=2$. Since $p\prime>2$, the Sobolev embedding $\mathbf{W}_0^{1,p\prime}(\Omega)\hookrightarrow\mathbf{C}(\overline{\Omega})$ guarantees that the forcing term $\langle \boldsymbol{f}\delta_{x_0},\boldsymbol{v}\rangle = \boldsymbol{f} \cdot \boldsymbol{v}(x_0)$ is well--defined and that $\delta_{x_0} \in \mathrm{W}^{-1,p}(\Omega)$. Since $\mathbf{G}_D$ is bounded when $p$ is restricted to \eqref{eq:rangep}, we conclude that problem \eqref{eq:weak_Stokes_system} is well--posed for $p\in\left(4/3-\varepsilon, 2\right)$.
\item[ii)] $d=3$. Notice that $p\prime>3$. Analogous arguments to the ones presented in the previous case allow us to conclude that \eqref{eq:weak_Stokes_system} is well--posed for $p\in\left(3/2-\varepsilon, 3/2 \right)$.
\end{itemize}
This concludes the proof. \qed
\end{proof}

\subsection{Inf--sup condition}
Let us introduce the product spaces 
\[
\mathcal{X}:=\mathbf{W}_0^{1,p}(\Omega)\times \mathrm{L}^{p}(\Omega)/\mathbb{R}, \quad \mathcal{Y}:=\mathbf{W}_0^{1,p\prime}(\Omega)\times \mathrm{L}^{p\prime}(\Omega)/\mathbb{R}.
\]
With these spaces at hand, we define the bilinear form $c:\mathcal{X}\times\mathcal{Y}\rightarrow \mathbb{R}$ by
\begin{equation}\label{eq:stokes_bilinear_c}
c((\boldsymbol{w},r),(\boldsymbol{v},q)):=a(\boldsymbol{w},\boldsymbol{v})+b(\boldsymbol{v},r)-b(\boldsymbol{w},q),
\end{equation}
with norm
\begin{equation}\label{eq:stokes_bilinear_c_norm}
\|c\|=\sup_{(\boldsymbol{0},0)\neq(\boldsymbol{w},r)\in\mathcal{X}}\sup_{(\boldsymbol{0},0)\neq(\boldsymbol{v},q)\in\mathcal{Y}}\frac{c((\boldsymbol{w},r),(\boldsymbol{v},q))}{\|(\boldsymbol{w},r)\|_{\mathcal{X}}\|(\boldsymbol{v},q)\|_{\mathcal{Y}}}.
\end{equation}

We introduce the following alternative weak formulation for problem \eqref{def:Stokes_singular_rhs}: Find $(\boldsymbol{u},\pi)\in\mathcal{X}$ such that
\begin{equation*}
c((\boldsymbol{u},\pi),(\boldsymbol{v},q))=\langle\boldsymbol{f}\delta_{x_0},\boldsymbol{v}\rangle\qquad \forall (\boldsymbol{v},q)\in\mathcal{Y}.
\end{equation*}
With the well--posedness of system \eqref{eq:weak_Stokes_system} for $p \in (2d/(d+1) - \varepsilon, d/(d-1))$ at hand, we conclude the existence of a constant $\beta>0$ such that bilinear form $c(\cdot,\cdot)$ satisfies the following inf--sup condition \cite[Theorem 2.1 and Remark 2.1]{MR972452}
\begin{multline}\label{eq:beta_infsup}
\inf_{(\boldsymbol{0},0)\neq(\boldsymbol{w},r)\in\mathcal{X}}\sup_{(\boldsymbol{0},0)\neq(\boldsymbol{v},q)\in\mathcal{Y}}\frac{c((\boldsymbol{w},r),(\boldsymbol{v},q))}{\|(\boldsymbol{w},r)\|_{\mathcal{X}}\|(\boldsymbol{v},q)\|_{\mathcal{Y}}}= \\ 
\inf_{(\boldsymbol{0},0)\neq(\boldsymbol{v},q)\in\mathcal{Y}}\sup_{(\boldsymbol{0},0)\neq(\boldsymbol{w},r)\in\mathcal{X}}\frac{c((\boldsymbol{w},r),(\boldsymbol{v},q))}{\|(\boldsymbol{w},r)\|_{\mathcal{X}}\|(\boldsymbol{v},q)\|_{\mathcal{Y}}}=\beta.
\end{multline}

\section{Finite element approximation}
\label{sec:fem}

We now introduce the discrete setting in which we will operate. We first introduce some terminology and a few basic ingredients and assumptions that will be common to all our methods.


\subsection{Triangulation and finite element spaces}
We consider $\mathscr{T}=\{T\}$ to be a conforming partition  of $\overline{\Omega}$ into closed simplices $T$ with size $h_T=\text{diam}(T)$. Define $h_{\mathscr{T}}:=\max_{T\in\mathscr{T}}h_T$. We denote by $\mathbb{T}$ the collection of conforming and shape regular meshes that are refinements of an initial mesh $\mathscr{T}_0$. 

Let $\mathscr{S}$ be the set of internal $(d-1)-$dimensional interelement boundaries $S$ of $\mathscr{T}$. For $S \in \mathscr{S}$, we denote by $h_S$ the diameter of $S$. For $T \in \mathscr{T}$, let $\mathscr{S}_T$ denote the subset of $\mathscr{S}$ which contains the sides in $\mathscr{S}$ which are sides of $T$. We also denote by $\mathcal{N}_S$ the subset of $\mathscr{T}$ that contains the two elements that have $S$ as a side, in other words, $\mathcal{N}_S=\{T^+,T^-\}$, where $T^+, T^- \in \mathscr{T}$ are such that $S = T^+ \cap T^-$. For $T \in \mathscr{T}$, we define the \emph{stars} or \emph{patches} associated with an element $T$ as
\begin{equation}\label{def:patch}
\mathcal{N}_T:=\bigcup_{T^{\prime}\in\mathscr{T}:T\cap T^\prime\neq\emptyset}T^\prime, \qquad \qquad 
\mathcal{N}_T^{*}:=\bigcup_{T^{\prime}\in\mathscr{T}:\mathscr{S}_{T}\cap \mathscr{S}_{T^\prime}\neq\emptyset}T^\prime.
\end{equation}

For a discrete tensor valued function $\mathbf{W}_{\mathscr{T}}$, we define the jump or interelement
residual on the internal side $S\in\mathscr{S}$, shared by the distinct elements $T^+, T^-\in\mathcal{N}_S$, by $\llbracket{\mathbf{W}_{\mathscr{T}}\cdot\boldsymbol{\nu}} \rrbracket=\mathbf{W}_{\mathscr{T}}|_{T^+}\cdot\boldsymbol{\nu}^+ +\mathbf{W}_{\mathscr{T}}|_{T^-}\cdot\boldsymbol{\nu}^-$. Here, $\boldsymbol{\nu}^+$ and $\boldsymbol{\nu}^-$ are unit normal on $S$ pointing towards $T^+$ and $T^-$, respectively.

\subsection{Inf--sup stable finite element spaces}
\label{sec:infsup_fem}

We now introduce the inf--sup stable finite element spaces that will be considered in our work. Given a mesh $\mathscr{T}\in\mathbb{T}$, we denote by $\mathbf{V}(\mathscr{T})$ and $\mathcal{P}(\mathscr{T})$ the finite element spaces that approximate the velocity field and the pressure, respectively. The following
elections are popular:
\begin{itemize}
\item[(a)] The mini element \cite[Section 4.2.4]{MR2050138}: Here,
\begin{align*}
&\mathbf{V}(\mathscr{T})=\{\boldsymbol{v}_{\mathscr{T}}\in\mathbf{C}(\overline{\Omega})\ :\ \boldsymbol{v}_{\mathscr{T}}|_T\in[\mathbb{P}_1(T)\oplus\mathbb{B}(T)]^{d} \ \forall \ T\in\mathscr{T}\}\cap\mathbf{W}_0^{1,p\prime}(\Omega),\\
&\mathcal{P}(\mathscr{T})=\{ q_{\mathscr{T}}\in C(\overline{\Omega})\ :\ q_{\mathscr{T}}|_T\in\mathbb{P}_1(T) \ \forall \ T\in\mathscr{T} \}\cap \mathrm{L}^{p\prime}(\Omega)/\mathbb{R},
\end{align*}
where $\mathbb{B}(T)$ denotes the space spanned by local bubble functions.
\item[(b)] The lowest order Taylor--Hood element \cite[Section 4.2.5]{MR2050138}: In this case,
\begin{align}
\label{eq:V_TH}
&\mathbf{V}(\mathscr{T})=\{\boldsymbol{v}_{\mathscr{T}}\in\mathbf{C}(\overline{\Omega})\ :\ \boldsymbol{v}_{\mathscr{T}}|_T\in[\mathbb{P}_2(T)]^{d} \ \forall \ T\in\mathscr{T}\}\cap\mathbf{W}_0^{1,p\prime}(\Omega),\\
\label{eq:P_TH}
&\mathcal{P}(\mathscr{T})=\{ q_{\mathscr{T}}\in C(\overline{\Omega})\ :\ q_{\mathscr{T}}|_T\in\mathbb{P}_1(T) \ \forall\ T\in\mathscr{T} \}\cap \mathrm{L}^{p\prime}(\Omega)/\mathbb{R}.
\end{align}
\end{itemize}
We observe that, for the values of $p$ provided in the statement of Theorem \ref{thm:well-posedness}, we have $\mathbf{V}(\mathscr{T})\subset\mathbf{W}^{1,p\prime}(\Omega)\subset\mathbf{W}^{1,p}(\Omega)$ and $\mathcal{P}(\mathscr{T})\subset \mathrm{L}^{p\prime}(\Omega) / \mathbb{R} \subset \mathrm{L}^{p}(\Omega) / \mathbb{R}$.

In the analysis that follows, the pair $(\mathbf{V}(\mathscr{T}),\mathcal{P}(\mathscr{T}))$ will represent indistinctly both the mini element and the lowest order Taylor--Hood element. An important property that these pairs of finite element spaces satisfy is the following compatibility condition: Let $1 < p < \infty$ and let $p\prime$ be the conjugate of $p$. Then, there exists $\gamma>0$, independent of $h_{\mathscr{T}}$, such that
\begin{equation}\label{eq:infsup_div}
\inf_{0\neq q_{\mathscr{T}}\in\mathcal{P}(\mathscr{T})}\sup_{\boldsymbol{0}\neq\boldsymbol{v}_{\mathscr{T}}\in\mathbf{V}(\mathscr{T})}\frac{b(\boldsymbol{v}_{\mathscr{T}},q_{\mathscr{T}})}{\|\nabla\boldsymbol{v}_{\mathscr{T}}\|_{\mathbf{L}^p(\Omega)}\|q_{\mathscr{T}}\|_{\mathrm{L}^{p\prime}(\Omega)}}\geq\gamma.
\end{equation}
We refer the reader to \cite[Lemma 4.20 and Lemma 4.24]{MR2050138} for a proof.

We consider the following finite element approximation of problem \eqref{eq:weak_Stokes_system}:
Find $(\boldsymbol{u}_{\mathscr{T}},\pi_{\mathscr{T}})\in\mathbf{V}(\mathscr{T})\times\mathcal{P}(\mathscr{T})$ such that 
\begin{equation}\label{eq:disc_Stokes_system}
\begin{array}{rcll}
 a(\boldsymbol{u}_{\mathscr{T}},\boldsymbol{v}_{\mathscr{T}})+b(\boldsymbol{v}_{\mathscr{T}},\pi_{\mathscr{T}}) & = &\langle\boldsymbol{f}\delta_{x_0},\boldsymbol{v}_{\mathscr{T}}\rangle &\quad \forall\ \boldsymbol{v}_{\mathscr{T}}\in\mathbf{V}(\mathscr{T}),\\
b(\boldsymbol{u}_{\mathscr{T}},q_{\mathscr{T}}) & = & 0 &\quad \forall\ q_{\mathscr{T}}\in\mathcal{P}(\mathscr{T}).
\end{array}
\end{equation}
Notice that, since $\mathbf{V}(\mathscr{T}) \hookrightarrow \mathbf{C}(\bar \Omega)$, the term  $\langle\boldsymbol{f}\delta_{x_0},\boldsymbol{v}_{\mathscr{T}}\rangle$ is well--defined. In fact $\langle\boldsymbol{f}\delta_{x_0},\boldsymbol{v}_{\mathscr{T}}\rangle = \boldsymbol{f} \cdot \boldsymbol{v}_{\mathscr{T}}(x_0)$. We thus conclude,  in view of the compatibility condition \eqref{eq:infsup_div}, the existence and uniqueness of a discrete solution; see \cite[Corollary 2.2]{MR972452}.

\subsection{Interpolation error estimates}
\label{sec:interpolation}
For $\mathscr{T} \in \mathbb{T}$ and $v \in \mathrm{W}_0^{1,p\prime}(\Omega)$, with $p\prime>d$, we define $\mathcal{I}_{\mathscr{T}} v$ as the Lagrange interpolation operator onto continuous piecewise polynomials of degree $k \in \{1,2\}$ over $\mathscr{T}$, that vanish on $\partial \Omega$. We will consider $k=1$ for approximation based on mini element and $k=2$ for Taylor--Hood approximation. For $\boldsymbol{v} \in \mathbf{W}_0^{1,p\prime}(\Omega)$, we set $\mathcal{I}_{\mathscr{T}}\boldsymbol{v}$ to be the Lagrange interpolation operator applied componentwise.

The following result provides interpolation error estimates.

\begin{lemma}[interpolation error estimates]
\label{lemma:interp_proper}
Let $T \in \mathscr{T}$. If $\boldsymbol{v}\in\mathbf{W}^{1,{p\prime}}(T)$, with $p\prime>d$, then
\begin{equation}\label{eq:estimate_element}
\|\boldsymbol{v}-\mathcal{I}_{\mathscr{T}}\boldsymbol{v}\|_{\mathbf{L}^{p\prime}(T)}\lesssim h_T\|\nabla\boldsymbol{v}\|_{\mathbf{L}^{p\prime}(T)}.
\end{equation} 
Let $T\in\mathscr{T}$ and $S \subset\mathscr{S}_T$. If $\boldsymbol{v}\in\mathbf{W}^{1,{p\prime}}(\mathcal{N}_S)$, with $p\prime>d$, then
\begin{equation}\label{eq:estimate_side}
\|\boldsymbol{v}-\mathcal{I}_{\mathscr{T}}\boldsymbol{v}\|_{\mathbf{L}^{p\prime}(S)}\lesssim h_T^{1-1/{p\prime}}\|\nabla\boldsymbol{v}\|_{\mathbf{L}^{p\prime}(\mathcal{N}_S)}.
\end{equation}
\end{lemma}
\begin{proof}
The estimate \eqref{eq:estimate_element} is standard; see, for instance, \cite[Theorem 1.103]{MR2050138}. The estimate \eqref{eq:estimate_side} follows from the scaled--trace inequality
\[
\| \boldsymbol{w} \|_{\mathbf{L}^{p\prime}(S)} \lesssim h_T^{-1/p\prime}\| \boldsymbol{w} \|_{\mathbf{L}^{p\prime}(T)} + h_T^{1-1/p\prime}\| \nabla \boldsymbol{w} \|_{\mathbf{L}^{p\prime}(T)} \quad \forall\ \boldsymbol{w} \in\mathrm{W}^{1,{p\prime}}(T),
\]
which follows, for instance, from the trace identity in \cite[Lemma 6.2]{MR2648380} and standard interpolation error estimates for the Lagrange interpolation operator \cite[Theorem 1.103]{MR2050138}. This concludes the proof.\qed
\end{proof} 

\section{A posteriori error estimates}\label{sec:a_posteriori_estimates}

We begin our analysis by introducing the so--called residual. Let $(\boldsymbol{u},\pi)\in\mathcal{X}$ and $(\boldsymbol{u}_{\mathscr{T}},\pi_{\mathscr{T}})\in\mathbf{V}(\mathscr{T})\times\mathcal{P}(\mathscr{T})$ be the solutions to problems \eqref{eq:weak_Stokes_system} and \eqref{eq:disc_Stokes_system}, respectively. We define the residual $\mathcal{R}:=\mathcal{R}(\boldsymbol{u}_{\mathscr{T}},\pi_{\mathscr{T}},\boldsymbol{f}\delta_{x_0})\in\mathcal{Y}^\prime$ as follows:
\begin{equation}\label{def:residual}
\langle\mathcal{R},(\boldsymbol{v},q)\rangle_{\mathcal{Y}^\prime\times\mathcal{Y}}:=\langle\boldsymbol{f}\delta_{x_0},\boldsymbol{v}\rangle-c((\boldsymbol{u}_{\mathscr{T}},\pi_{\mathscr{T}}),(\boldsymbol{v},q)) \quad \forall\ (\boldsymbol{v},q)\in\mathcal{Y}.
\end{equation}
Notice that the residual depends on the approximated solution $(\boldsymbol{u}_{\mathscr{T}},\pi_{\mathscr{T}}) \in \mathbf{V}(\mathscr{T})\times\mathcal{P}(\mathscr{T})$ and the data $\boldsymbol{f}$ and $\delta_{x_0}$.

\subsection{Error and residual}
Let us define the error $(\mathbf{e}_{\boldsymbol{u}}, e_\pi):=(\boldsymbol{u}-\boldsymbol{u}_{\mathscr{T}},\pi-\pi_{\mathscr{T}})$. The residual and the error are related by the following identity:
\begin{equation}\label{eq:identity_residual_error}
\langle\mathcal{R},(\boldsymbol{v},q)\rangle_{\mathcal{Y}^\prime\times\mathcal{Y}}=c((\mathbf{e}_{\boldsymbol{u}},e_\pi),(\boldsymbol{v},q)) \quad \forall\ (\boldsymbol{v},q)\in\mathcal{Y}.
\end{equation}

The next result guarantees that the residual and the error are equivalent.
\begin{lemma}[equivalence result]\label{lemma:eqErrorRes}
Let $(\boldsymbol{u},\pi)$ and $(\boldsymbol{u}_{\mathscr{T}},\pi_{\mathscr{T}})$ be the solutions to \eqref{eq:weak_Stokes_system} and \eqref{eq:disc_Stokes_system}, respectively. If $p \in (2d/(d+1) - \varepsilon, d/(d-1))$, then
\begin{equation*}
\beta\|(\mathbf{e}_{\boldsymbol{u}},e_\pi)\|_{\mathcal{X}}\leq\|\mathcal{R}\|_{\mathcal{Y}\prime}\leq\|c\|\|(\mathbf{e}_{\boldsymbol{u}},e_\pi)\|_{\mathcal{X}},
\end{equation*}
where $\beta>0$ is the inf--sup constant associated to the bilinear form $c(\cdot,\cdot)$, given in \eqref{eq:beta_infsup}, and $\|c\| \geq \beta$ corresponds to the norm of $c(\cdot,\cdot)$, which is defined in \eqref{eq:stokes_bilinear_c_norm}.
\end{lemma}
\begin{proof}
In view of the inf--sup condition \eqref{eq:beta_infsup} and the relation \eqref{eq:identity_residual_error}, we immediately arrive at
\begin{equation*}
\beta\|(\mathbf{e}_{\boldsymbol{u}},e_\pi)\|_{\mathcal{X}} \leq \sup_{(\boldsymbol{0},0)\neq(\boldsymbol{v},q)\in\mathcal{Y}}\frac{c((\mathbf{e}_{\boldsymbol{u}},e_\pi),(\boldsymbol{v},q))}{\|(\boldsymbol{v},q)\|_{\mathcal{Y}}}=\|\mathcal{R}\|_{\mathcal{Y}^\prime}.
\end{equation*}
On the other hand, invoking again \eqref{eq:identity_residual_error}, and then \eqref{eq:stokes_bilinear_c_norm}, we conclude that
\begin{equation*}
\|\mathcal{R}\|_{\mathcal{Y}^\prime}=\sup_{(\boldsymbol{0},0)\neq(\boldsymbol{v},q)\in\mathcal{Y}}\frac{c((\mathbf{e}_{\boldsymbol{u}},e_\pi),(\boldsymbol{v},q))}{\|(\boldsymbol{v},q)\|_{\mathcal{Y}}}
\leq \|c\|\|(\mathbf{e}_{\boldsymbol{u}},e_\pi)\|_{\mathcal{X}}.
\end{equation*}
This concludes the proof.\qed
\end{proof}

\subsection{A posteriori error estimators}\label{sec:error_estimators}
We now introduce a posteriori error estimators for the finite element approximation \eqref{eq:disc_Stokes_system} on the basis of the low--order inf--sup stable finite element pairs introduced in Section \ref{sec:infsup_fem}. 

Let $T \in \mathscr{T}$. If $x_0 \in T$ is such that
\begin{enumerate}[label=(\roman*)]
\item $x_0$ is not a vertex of $T$ or a midpoint of a side of $T$, when Taylor--Hood approximation is considered, or
\label{i}
\item $x_0$ is not a vertex of $T$, when the approximation based on the mini element is considered, then
\label{ii}
\end{enumerate}
we define the element error indicators
\begin{multline}
\eta_{p,T}:=\Big( 
h_T^p \|\Delta\boldsymbol{u}_{\mathscr{T}}-\nabla\pi_{\mathscr{T}}\|_{\mathbf{L}^p(T)}^p 
+
h_T \|\llbracket{(\nabla\boldsymbol{u}_{\mathscr{T}}-\mathbb{I}_{d}\pi_{\mathscr{T}})\cdot\boldsymbol{\nu}}\rrbracket \|_{\mathbf{L}^p(\partial T \setminus \partial \Omega)}^p 
\\
+
\|\text{div}\boldsymbol{u}_{\mathscr{T}}\|_{\mathrm{L}^p(T)}^p+h_T^{d-p(d-1)} |\boldsymbol{f}|^p \Big)^{\frac{1}{p}}.
\label{eq:local_indicator}
\end{multline}
If $x_0 \in T$ and \ref{i} or \ref{ii} do not hold, then
\begin{multline}
\eta_{p,T}:=\Big(
h_T^p \|\Delta\boldsymbol{u}_{\mathscr{T}}-\nabla\pi_{\mathscr{T}}\|_{\mathbf{L}^p(T)}^p 
+
h_T \|\llbracket{(\nabla\boldsymbol{u}_{\mathscr{T}}-\mathbb{I}_{d}\pi_{\mathscr{T}})\cdot\boldsymbol{\nu}}\rrbracket \|_{\mathbf{L}^p(\partial T \setminus \partial \Omega)}^p 
\\
+
\|\text{div}\boldsymbol{u}_{\mathscr{T}}\|_{\mathrm{L}^p(T)}^p
\Big)^{\frac{1}{p}}.
\label{eq:local_indicator2}
\end{multline}
If $x_0 \notin T$, then the indicator $\eta_{p,T}$ is defined as in \eqref{eq:local_indicator2}. Here, $(\boldsymbol{u}_{\mathscr{T}},\pi_{\mathscr{T}})$ denotes the solution to the discrete problem \eqref{eq:disc_Stokes_system} and $\mathbb{I}_{d}$ denotes the identity matrix in $\mathbb{R}^{d\times d}$. 
We recall that we consider our elements $T$ to be closed sets. Notice that, when Taylor--Hood approximation is considered, for functions $\boldsymbol{v} \in \mathbf{W}^{1,p\prime}(\Omega)$, with $p\prime>d$, we have that $(\boldsymbol{v} - \mathcal{I}_{\mathscr{T}} \boldsymbol{v})(x_0)$ vanishes when $x_0$ is a vertex of $T$ or a midpoint of a side of $T$. This motivates \ref{i}. Similar arguments motivate \ref{ii}; see also the proof of Theorem \ref{thm:reliability_estimate} below.

The a posteriori error estimators are thus defined by
\begin{equation}\label{eq:error_estimator}
\eta_p:=\left( \sum_{T\in\mathscr{T}} \eta_{p,T}^p\right)^{\frac{1}{p}}.
\end{equation}

\subsection{Reliability}
The main objective of this section is to obtain a global reliability property for the a posteriori error estimators $\eta_{p}$.

\begin{theorem}[global reliability]\label{global_reliability}
Let $p \in (2d/(d+1) - \varepsilon, d/(d-1))$. Let $(\boldsymbol{u},\pi)\in \mathbf{W}^{1,p}(\Omega)\times \mathrm{L}^{p}(\Omega)/\mathbb{R}$ be the solution to \eqref{eq:weak_Stokes_system} and $(\boldsymbol{u}_\mathscr{T},\pi_\mathscr{T})\in \mathbf{V}(\mathscr{T})\times \mathcal{P}(\mathscr{T})$ its finite element approximation obtained as the solution to \eqref{eq:disc_Stokes_system}. Then
\begin{equation*}\label{eq:reliability_estimate}
\|(\mathbf{e}_{\boldsymbol{u}},e_\pi)\|_{\mathcal{X}}
\lesssim
\eta_p,
\end{equation*}
where $\eta_p$ is defined as in \eqref{eq:error_estimator}. The hidden constant is independent of the solution $(\boldsymbol{u},\pi)$, its finite element approximation $(\boldsymbol{u}_\mathscr{T},\pi_\mathscr{T})$, the size of the elements in the mesh $\mathscr{T}$, and $\#\mathscr{T}$.
\label{thm:reliability_estimate}
\end{theorem}
\begin{proof}
We begin the proof by invoking the basic estimate $\beta \|(\mathbf{e}_{\boldsymbol{u}},e_\pi)\|_{\mathcal{X}} \leq \| \mathcal{R}\|_{\mathcal{Y}'}$, which follows immediately from Lemma \ref{lemma:eqErrorRes}. It thus suffices to bound $\| \mathcal{R}\|_{\mathcal{Y}'}$. To accomplish this task, we notice that, in view of definitions \eqref{def:residual} and \eqref{eq:stokes_bilinear_c}, we have for $(\boldsymbol{v},q)\in\mathcal{Y}$ arbitrary,
\begin{equation*}\label{eq:identity_1}
\langle\mathcal{R},(\boldsymbol{v},q)\rangle_{\mathcal{Y}^\prime\times\mathcal{Y}}=\langle\boldsymbol{f}\delta_{x_0},\boldsymbol{v}\rangle-a(\boldsymbol{u}_{\mathscr{T}},\boldsymbol{v})-b(\boldsymbol{v},\pi_\mathscr{T})+b(\boldsymbol{u}_{\mathscr{T}},q).
\end{equation*}
Next, we write $a(\boldsymbol{u}_{\mathscr{T}},\boldsymbol{v})$, $b(\boldsymbol{v},\pi_\mathscr{T})$, and $b(\boldsymbol{u}_\mathscr{T},q)$ as integrals over elements $T \in \mathscr{T}$ and utilize elementwise integration by parts to arrive at
\begin{multline}
\label{eq:identity_2aux}
\langle\mathcal{R},(\boldsymbol{v},q)\rangle_{\mathcal{Y}^\prime\times\mathcal{Y}}
=
\langle\boldsymbol{f}\delta_{x_0},\boldsymbol{v} \rangle
-
\sum_{T\in\mathscr{T}}\int_T(-\Delta \boldsymbol{u}_\mathscr{T} + \nabla \pi_\mathscr{T})\cdot \boldsymbol{v} 
\\
-\sum_{T\in\mathscr{T}}\int_T \text{div}\boldsymbol{u}_{\mathscr{T}} q 
+\sum_{S \in \mathscr{S}}\int_S \llbracket{(\nabla \boldsymbol{u}_\mathscr{T}-\pi_\mathscr{T}\mathbb{I}_{d})\cdot\boldsymbol\nu}\rrbracket\cdot \boldsymbol{v}.
\end{multline}
We invoke the relation $c((\mathbf{e}_{\boldsymbol{u}},e_{\pi}),(\mathcal{I}_\mathscr{T}\boldsymbol{v},0)) = 0$, which follows from Galerkin orthogonality, to arrive at
\begin{multline}\label{eq:identity_2}
\langle\mathcal{R},(\boldsymbol{v},q)\rangle_{\mathcal{Y}^\prime\times\mathcal{Y}}=\langle\boldsymbol{f}\delta_{x_0},\boldsymbol{v}-\mathcal{I}_{\mathscr{T}}\boldsymbol{v}\rangle-\sum_{T\in\mathscr{T}}\int_T(-\Delta \boldsymbol{u}_\mathscr{T}+\nabla \pi_\mathscr{T})\cdot(\boldsymbol{v}-\mathcal{I}_{\mathscr{T}}\boldsymbol{v}) \\-\sum_{T\in\mathscr{T}}\int_T \text{div}\boldsymbol{u}_\mathscr{T} q 
+\sum_{S \in \mathscr{S}}\int_S \llbracket{(\nabla \boldsymbol{u}_\mathscr{T}-\pi_\mathscr{T}\mathbb{I}_{d})\cdot\boldsymbol\nu}\rrbracket\cdot (\boldsymbol{v}-\mathcal{I}_{\mathscr{T}}\boldsymbol{v})
\\
=: \mathbf{I}- \sum_{T \in \mathscr{T}} \mathbf{II}_T - \sum_{T \in \mathscr{T}} \mathbf{III}_T + \sum_{S \in \mathscr{S}} \mathbf{IV}_S.
\end{multline}
Here, $\mathcal{I}_\mathscr{T}$ denotes the Lagrange interpolation operator; see Section \ref{sec:interpolation} for details. We must immediately mention that, since $\boldsymbol{v} \in  \mathbf{W}^{1,p\prime}(\Omega)$, with $p\prime> d$, we have that $\mathbf{W}^{1,p\prime}(\Omega) \hookrightarrow \mathbf{C}(\bar \Omega)$. Consequently, $\mathcal{I}_\mathscr{T} \boldsymbol{v}$ is well--defined. We now bound each term on the right--hand side of \eqref{eq:identity_2} separately.

We estimate $\mathbf{I}$. Let $T \in \mathscr{T}$ such that $x_0 \in T$. Notice that, if conditions \ref{i} or \ref{ii} do not hold, then $\mathbf{I} = \langle\boldsymbol{f}\delta_{x_0},\boldsymbol{v}-\mathcal{I}_{\mathscr{T}}\boldsymbol{v}\rangle = \boldsymbol{f} \cdot (\boldsymbol{v}-\mathcal{I}_{\mathscr{T}}\boldsymbol{v}) (x_0)$ vanishes. Assume that $x_0 \in T$ and conditions  \ref{i} or \ref{ii} hold. If this is the case, standard interpolation error estimates for the Lagrange operator $\mathcal{I}_{\mathscr{T}}$ yields
\begin{equation*}
\mathbf{I}
\lesssim 
|\boldsymbol f|\|\boldsymbol{v}-\mathcal{I}_{\mathscr{T}}\boldsymbol{v}\|_{\mathbf{L}^{\infty}(T)} 
\lesssim 
h_{T}^{1-d/p\prime}|\boldsymbol f|\|\nabla\boldsymbol{v}\|_{\mathbf{L}^{p\prime}(T)}.
\end{equation*}

The control of the term $\mathbf{II}_T$ follows from H\"older's inequality and \eqref{eq:estimate_element}. In fact, for $T \in \mathscr{T}$, we have
\begin{equation*}
\mathbf{II}_T \lesssim  h_T \|\Delta\boldsymbol{u}_{\mathscr{T}}-\nabla\pi_{\mathscr{T}}\|_{\mathbf{L}^p(T)}\|\nabla\boldsymbol{v}\|_{\mathbf{L}^{p\prime}(T)}.
\end{equation*}

The control of the term $\mathbf{III}_T$ follows from a basic application of H\"older's inequality: If $T \in \mathscr{T}$, then
\begin{equation*}
\mathbf{III}_T
\leq \|\text{div}\boldsymbol{u}_{\mathscr{T}}\|_{\mathrm{L}^{p}(T)}  \|q\|_{\mathrm{L}^{p\prime}(T)}.
\end{equation*}

We now estimate the term $\mathbf{IV}_S$. To accomplish this task, we first apply H\"older's inequality and then the estimate \eqref{eq:estimate_side}. These arguments yield, for $S \in \mathscr{S}$,
\begin{align*}
\mathbf{IV}_S
& \leq 
\| \llbracket{(\nabla \boldsymbol{u}_\mathscr{T}-\mathbb{I}_{d}\pi_\mathscr{T})\cdot\boldsymbol\nu}\rrbracket\|_{\mathbf{L}^p(S)}\| \boldsymbol{v}-\mathcal{I}_{\mathscr{T}}\boldsymbol{v}\|_{\mathbf{L}^{p\prime}(S)} 
\\ 
& \lesssim
h_T^{\frac{1}{p}} \| \llbracket{(\nabla \boldsymbol{u}_\mathscr{T}-\mathbb{I}_{d}\pi_\mathscr{T})\cdot\boldsymbol\nu}\rrbracket\|_{\mathbf{L}^p(S)} 
\|\nabla\boldsymbol{v}\|_{\mathbf{L}^{p\prime}(\mathcal{N}_S)}.
\end{align*}

Consequently, replacing the estimates  obtained for $\mathbf{I}$, $\mathbf{II}_T$, $\mathbf{III}_T$, and $\mathbf{IV}_S$ into  \eqref{eq:identity_2}, we obtain
\begin{multline*}
\label{eq:reli_estimate_aux}
\langle\mathcal{R},(\boldsymbol{v},q)\rangle_{\mathcal{Y}^\prime\times\mathcal{Y}}  \lesssim 
\sum_{T \ni x_0} h_{T}^{1-d/p\prime}|\boldsymbol f|\|\nabla\boldsymbol{v}\|_{\mathbf{L}^{p\prime}(T)}
+\sum_{T \in \mathscr{T}} \|\text{div}\boldsymbol{u}_{\mathscr{T}}\|_{\mathrm{L}^{p}(T)}  \| q\|_{\mathrm{L}^{p\prime}(T)}
\\
+ \sum_{T \in \mathscr{T} } h_T^{\frac{1}{p}} \| \llbracket{(\nabla \boldsymbol{u}_\mathscr{T}-\mathbb{I}_{d}\pi_\mathscr{T})\cdot\boldsymbol\nu}\rrbracket\|_{\mathbf{L}^p(\partial T \setminus \partial \Omega)}  \| \nabla \boldsymbol{v}\|_{\mathbf{L}^{p\prime}(\mathcal{N}_T^{*})} 
\\
+ \sum_{T \in \mathscr{T} } h_T \|\Delta\boldsymbol{u}_{\mathscr{T}}-\nabla\pi_{\mathscr{T}}\|_{\mathbf{L}^p(T)}  \| \nabla \boldsymbol{v}\|_{\mathbf{L}^{p\prime}(\mathcal{N}_T^{*})}
\\
\lesssim \eta_{p} \left( \sum_{T \in \mathscr{T}} ( \| q\|_{\mathrm{L}^{p\prime}(T)} + \| \nabla \boldsymbol{v}\|_{\mathbf{L}^{p\prime}(\mathcal{N}_T)} )^{p\prime} \right)^{\frac{1}{p\prime}}
\lesssim \eta_{p} \| (\boldsymbol{v},q) \|_{\mathcal{Y}}
\end{multline*}
where we have used H\"older's inequality, the definition of the local error indicators $\eta_{p,T}$, given in \eqref{eq:local_indicator} and \eqref{eq:local_indicator2} and the finite overlapping property of stars. This concludes the proof.\qed
\end{proof}

\subsection{Efficiency}
In this section we study the efficiency properties of the a posteriori error estimators $\eta_p$, defined in \eqref{eq:error_estimator}, by examining each of their contributions separately. To accomplish this task, we will invoke standard residual estimation techniques based in suitable bubble functions. Before proceeding with such analysis, we introduce the following notation: for an edge/face or triangle/tetrahedron $G$, let $\mathcal{V}(G)$ be the set of vertices of $G$. With this notation at hand, we define, for  $T\in\mathscr{T}$ and $S\in\mathscr{S}$, the standard element and edge bubble functions \cite{MR993474,MR3059294,MR1885308}
\begin{equation*}\label{def:standard_bubbles}
\varphi^{}_{T}=
(d+1)^{(d+1)}\prod_{\textsc{v} \in \mathcal{V}(T)} \lambda^{}_{\textsc{v}},\quad
\varphi^{}_{S}=
d^{d} \prod_{\textsc{v} \in \mathcal{V}(S)}\lambda^{}_{\textsc{v}}|^{}_{T'}\quad \text{with } T' \subset \mathcal{N}_{S},
\end{equation*}
respectively, where $\lambda_{\textsc{v}}$ are the barycentric coordinates of $T$. We recall that $\mathcal{N}_{S}$ corresponds to the patch composed of the two elements of $\mathscr{T}$ sharing $S$.

We also introduce, inspired in \cite[Section 3]{MR2262756} and \cite[Section 3]{MR3237857}, the following bubble functions. Given $T\in\mathscr{T}$, we define $\phi_T$ as
\begin{equation}\label{def:bubble_element}
\phi_T(x):=
\begin{cases}
\varphi_T(x) \frac{|x-x_0|^{2}}{h_T^2} &\text{ if } x_0\in T,\\
\varphi_T(x) &\text{ if } x_0\not\in T.
\end{cases}
\end{equation}
Given $S\in \mathscr{S}$, we define $\phi_S$ as
\begin{equation}\label{def:bubble_side}
\phi_S(x):=
\begin{cases}
\varphi_S(x) \frac{|x-x_0|^{2}}{h_S^2} &\text{ if } x_0\in \mathring{\mathcal{N}}_S,\\
\varphi_S(x) &\text{ if } x_0\not\in \mathring{\mathcal{N}}_S,
\end{cases}
\end{equation}
where $\mathring{\mathcal{N}}_S$ denotes the interior of $\mathcal{N}_S$. We recall that the Dirac measure $\delta_{x_0}$ is supported at $x_0 \in \Omega$: it can thus be supported on the interior, an edge, or a vertex of an element $T$ of the triangulation $\mathscr{T}$.

Given $S\in\mathscr{S}$, we introduce the continuation operator $\Pi:\mathrm{L}^\infty(S)\rightarrow \mathrm{L}^\infty(\mathcal{N}_S)$ as defined in \cite[Section 3]{MR1650051}. This operator maps polynomials onto piecewise polynomials of the same degree. With this operator at hand, we provide the following result. 

\begin{lemma}[bubble function properties]\label{lemma:bubble_properties}
Let $T\in\mathscr{T}$, $S\in\mathscr{S}$, $m \in \mathbb{N}$, and $r\in(1,\infty)$. Then, the bubble functions $\phi_T$ and $\phi_S$ introduced in \eqref{def:bubble_element} and \eqref{def:bubble_side}, respectively, satisfy
\begin{equation}\label{eq:bubble_property_1}
\|\phi_T\|_{\mathrm{W}^{m,r}(T)}  \lesssim h_T^{d/r-m}.
\end{equation}

In addition, if $\boldsymbol{v}_{\mathscr{T}}|^{}_T\in[\mathbb{P}_2(T)]^d$ and $\boldsymbol{w}_\mathscr{T}|^{}_S\in[\mathbb{P}_3(S)]^d$, then
\begin{align}\label{eq:bubble_property_3}
\|\boldsymbol{v}_{\mathscr{T}}\|_{\mathbf{L}^r(T)} &\lesssim \|\boldsymbol{v}_{\mathscr{T}} \phi_T^{\frac{1}{r}}\|_{\mathbf{L}^r(T)} \lesssim \|\boldsymbol{v}_{\mathscr{T}}\|_{\mathbf{L}^r(T)}, 
\\
\|\boldsymbol{w}_\mathscr{T}\|_{\mathbf{L}^r(S)} &\lesssim \|\boldsymbol{w}_\mathscr{T} \phi_S^{\frac{1}{r}}\|_{\mathbf{L}^r(S)} \lesssim \|\boldsymbol{w}_\mathscr{T}\|_{\mathbf{L}^r(S)},\label{eq:bubble_property_4}
\\
\|\phi_S \Pi \boldsymbol{w}_\mathscr{T}\|_{\mathbf{L}^r(T)} & \lesssim h_T^{\frac{1}{r}}\|\boldsymbol{w}_\mathscr{T}\|_{\mathbf{L}^r(S)}.
\label{eq:bubble_property_2}
\end{align}
\end{lemma}
\begin{proof}
We derive \eqref{eq:bubble_property_1}. If $x_0 \notin T$, then $\phi_T=\varphi_T$. As a consequence, \eqref{eq:bubble_property_1} follows from standard arguments. If $x_0\in T$, then it follows from \cite[Lemma 4.5.3]{MR2373954} that
\begin{equation*}
\|\phi_T\|_{\mathrm{W}^{m,r}(T)}\lesssim h_T^{-m} \|\phi_T \|_{\mathrm{L}^r(T)}.
\end{equation*}
In view of the definition of $\phi_T$, we invoke properties of the standard bubble function $\varphi_T$ to conclude that
\begin{align*}
\|\phi_T\|_{\mathrm{W}^{m,r}(T)} 
\lesssim h_T^{-m} \left\{\int_T\left(\varphi_T\frac{|x-x_0|^{2}}{h_T^2}\right)^r \right\}^{\frac{1}{r}}
\lesssim h_T^{-m}|T|^{1/r} \lesssim h_T^{d/r-m},
\end{align*}
where we have also used that $|T|\approx h_T^d$. This yields \eqref{eq:bubble_property_1}. 

The estimates \eqref{eq:bubble_property_3}--\eqref{eq:bubble_property_2} follow standard arguments. For brevity, we skip the details.\qed
\end{proof}

We now provide local efficiency estimates for the indicator $\eta_{p,T}$ defined in \eqref{eq:local_indicator}--\eqref{eq:local_indicator2}.

\begin{theorem}[local efficiency]\label{thm:efficiency}
Let $p \in (2d/(d+1) - \varepsilon, d/(d-1))$. Let $(\boldsymbol{u},\pi)\in \mathbf{W}^{1,p}(\Omega)\times \mathrm{L}^{p}(\Omega)/\mathbb{R}$ be the solution to \eqref{eq:weak_Stokes_system} and $(\boldsymbol{u}_\mathscr{T},\pi_\mathscr{T})\in \mathbf{V}(\mathscr{T})\times \mathcal{P}(\mathscr{T})$ its finite element approximation obtained as the solution to \eqref{eq:disc_Stokes_system}. Then, for $T\in \mathscr{T}$, the local error indicator $\eta_{p,T}$, defined in \eqref{eq:local_indicator}--\eqref{eq:local_indicator2}, satisfies that
\begin{equation}\label{eq:local_efficiency}
\eta_{p,T}^{p}
\lesssim
\|\nabla\mathbf{e}_{\boldsymbol{u}}\|_{\mathbf{L}^p(\mathcal{N}_T)}^p+\|e_\pi\|_{\mathrm{L}^p(\mathcal{N}_T)}^p,
\end{equation}
where $\mathcal{N}_T$ is defined in \eqref{def:patch}. The hidden constant is independent of the solution $(\boldsymbol{u},\pi)$, its approximation $(\boldsymbol{u}_\mathscr{T},\pi_\mathscr{T})$, the size of the elements in the mesh $\mathscr{T}$, and $\#\mathscr{T}$.
\end{theorem}
\begin{proof}
We proceed in five steps.

\underline{Step 1.} Let $T\in\mathscr{T}$. In this step we bound the term $h_T^p\|\Delta  \boldsymbol{u}_\mathscr{T}-\nabla \pi_\mathscr{T}\|_{\mathbf{L}^p(T)}$ in \eqref{eq:local_indicator}--\eqref{eq:local_indicator2}. To accomplish this task and also to simplify the presentation of the material, we define $\mathbf{R}_T:=(\Delta  \boldsymbol{u}_\mathscr{T}-\nabla \pi_\mathscr{T})|^{}_T$ and $\boldsymbol\Phi_T:=\phi_T \mathbf{R}_T$. We recall that $\phi_T$ is given as in \eqref{def:bubble_element}. Now, set $\boldsymbol{v}=\boldsymbol\Phi_T$ and $q=0$ in \eqref{eq:identity_2aux}. This yields
\begin{equation*}\label{eq:eff_1}
|\langle\mathcal{R},(\boldsymbol\Phi_T,0)\rangle_{\mathcal{Y}^\prime\times\mathcal{Y}}| = 
\left|\int_T(\Delta \boldsymbol{u}_\mathscr{T}-\nabla \pi_\mathscr{T})\cdot \boldsymbol\Phi_T\right|=\|\mathbf{R}_T\phi_T^{\frac{1}{2}}\|_{\mathbf{L}^2(T)}^2.
\end{equation*}
Observe that $\langle\boldsymbol{f}\delta_{x_0},\boldsymbol\Phi_T \rangle = 0$. Now, set $\boldsymbol{v}=\boldsymbol\Phi_T$ and $q=0$ in \eqref{eq:identity_residual_error} and conclude that $\langle\mathcal{R},(\boldsymbol\Phi_T,0)\rangle_{\mathcal{Y}^\prime\times\mathcal{Y}} =c((\mathbf{e}_{\boldsymbol{u}},e_\pi),(\boldsymbol\Phi_T,0))$. We thus use \eqref{eq:bubble_property_3} to derive
\begin{multline}\label{eq:eff_2}
\|\mathbf{R}_T\|_{\mathbf{L}^2(T)}^2 \lesssim \|\mathbf{R}_T\phi_T^{\frac{1}{2}}\|_{\mathbf{L}^2(T)}^2
\lesssim 
|a(\mathbf{e}_{\boldsymbol{u}},\boldsymbol\Phi_T)+b(\boldsymbol\Phi_T,e_\pi)| \\
\leq \|\nabla\mathbf{e}_{\boldsymbol{u}}\|_{\mathbf{L}^p(T)}\|\nabla \boldsymbol\Phi_T\|_{\mathbf{L}^{p\prime}(T)}+ \|\text{div }\boldsymbol\Phi_T\|_{\mathrm{L}^{p\prime}(T)}\|e_\pi\|_{\mathrm{L}^p(T)},
\end{multline}
where, to obtain the last inequality, we have used H\"older's inequality.

On the other hand, notice that
\begin{equation*}\label{eq:grad_bubble_element}
\nabla \boldsymbol \Phi_T
=
\begin{bmatrix}
\nabla \phi_T \mathbf{R}_{1,T}+\phi_T \nabla\mathbf{R}_{1,T}, \ldots, \nabla \phi_T \mathbf{R}_{d,T}+\phi_T \nabla\mathbf{R}_{d,T}
\end{bmatrix}^{\intercal},
\end{equation*}
where $\intercal$ denotes the transpose operator. We invoke the properties that $\phi_T$ satisfies, which are stated in Lemma \ref{lemma:bubble_properties}, and standard inverse inequalities \cite[Lemma 4.5.3]{MR2373954} to arrive at
\begin{align*}
\|\nabla \boldsymbol\Phi_T\|_{\mathbf{L}^{p\prime}(T)}
&\lesssim
\|\nabla \phi_T \mathbf{R}_{T}\|_{\mathbf{L}^{p\prime}(T)}+\|\phi_T \nabla\mathbf{R}_{T}\|_{\mathbf{L}^{p\prime}(T)}\\
&\lesssim
\|\nabla \phi_T\|_{\mathbf{L}^\infty(T)}\|\mathbf{R}_{T}\|_{\mathbf{L}^{p\prime}(T)}+\|\nabla\mathbf{R}_{T}\|_{\mathbf{L}^{p\prime}(T)}\\
&\lesssim h_T^{-1}\|\mathbf{R}_{T}\|_{\mathbf{L}^{p\prime}(T)}\lesssim h_T^{-1}h_T^{d/p\prime-d/2}\|\mathbf{R}_{T}\|_{\mathbf{L}^{2}(T)}.
\end{align*}
Replacing this estimate into \eqref{eq:eff_2} yields
\begin{equation}\label{eq:eff_2_1}
\|\mathbf{R}_T\|_{\mathbf{L}^2(T)}^2 
\lesssim
\left(\|\nabla\mathbf{e}_{\boldsymbol{u}}\|_{\mathbf{L}^p(T)}+ \|e_\pi\|_{\mathrm{L}^p(T)}\right)h_T^{-1}h_T^{d/p\prime-d/2}\|\mathbf{R}_{T}\|_{\mathbf{L}^{2}(T)}.
\end{equation}
Now, on the basis of the inverse estimate $\|\mathbf{R}_T\|_{\mathbf{L}^{p}(T)}\lesssim h_T^{d/p-d/2} \|\mathbf{R}_T\|_{\mathbf{L}^{2}(T)}$, \eqref{eq:eff_2_1} reveals that
\begin{equation}\label{eq:residual_estimate_aux}
\|\mathbf{R}_T\|_{\mathbf{L}^{p}(T)}
\lesssim
h_T^{-1}\left(\|\nabla\mathbf{e}_{\boldsymbol{u}}\|_{\mathbf{L}^p(T)}+ \|e_\pi\|_{\mathrm{L}^p(T)}\right).
\end{equation}
This allows us to conclude that
\begin{equation}\label{eq:residual_estimate}
h_T^p\|\mathbf{R}_T\|_{\mathbf{L}^p(T)}^p 
\lesssim
\|\nabla\mathbf{e}_{\boldsymbol{u}}\|_{\mathbf{L}^p(T)}^p + \|e_\pi\|_{\mathrm{L}^p(T)}^p.
\end{equation}

\underline{Step 2.} Let $T\in\mathscr{T}$. The control of the term $\|\text{div }\boldsymbol{u}_\mathscr{T}\|_{\mathrm{L}^p(T)}$ in \eqref{eq:local_indicator}--\eqref{eq:local_indicator2} follows from the incompressibility condition
$\text{div }\boldsymbol{u}=0$. In fact,
\begin{equation}
\label{eq:divergence_estimate}
\|\text{div }\boldsymbol{u}_\mathscr{T}\|_{\mathrm{L}^p(T)}^p
=
\|\text{div }\mathbf{e}_{\boldsymbol{u}}\|_{\mathrm{L}^p(T)}^p
\lesssim
\|\nabla\mathbf{e}_{\boldsymbol{u}}\|_{\mathbf{L}^p(T)}^p.
\end{equation} 

\underline{Step 3.} Let $T\in\mathscr{T}$ and $S\in \mathscr{S}_T$. We bound $h_T\|\llbracket{(\nabla\boldsymbol{u}_{\mathscr{T}}-\mathbb{I}_{d}\pi_{\mathscr{T}})\cdot\boldsymbol{\nu}}\rrbracket \|_{\mathbf{L}^p(S)}^p$ in \eqref{eq:local_indicator}--\eqref{eq:local_indicator2}. To simplify the presentation of the material, we define
\[
 \mathbf{J}_S=\llbracket{(\nabla\boldsymbol{u}_{\mathscr{T}}-\mathbb{I}_{d}\pi_{\mathscr{T}})\cdot\boldsymbol{\nu}}\rrbracket, \quad 
 \boldsymbol\Phi_S:=\phi_S  \mathbf{J}_S.
\]
Set $\boldsymbol{v}=\boldsymbol\Phi_S$ and $q=0$ in \eqref{eq:identity_residual_error} and invoke \eqref{eq:identity_2aux}. This yields
\begin{multline*}
c((\mathbf{e}_{\boldsymbol{u}},e_\pi),(\boldsymbol\Phi_S,0)) 
= \langle\mathcal{R},(\boldsymbol\Phi_S,0)\rangle_{\mathcal{Y}^\prime\times\mathcal{Y}} =  
\sum_{{T^\prime}\in\mathcal{N}_S}\int_{T^\prime}(\Delta \boldsymbol{u}_\mathscr{T}-\nabla \pi_\mathscr{T})\cdot\boldsymbol\Phi_S
\\
+\int_S \llbracket{(\nabla \boldsymbol{u}_\mathscr{T}-\pi_\mathscr{T}\mathbb{I}_{d})\cdot\boldsymbol\nu}\rrbracket\cdot \boldsymbol\Phi_S
=
\sum_{{T^\prime}\in\mathcal{N}_S}\int_{T^\prime}\mathbf{R}_{T^\prime}\cdot\boldsymbol\Phi_S 
+\int_S \mathbf{J}_S\cdot \boldsymbol\Phi_S.
\end{multline*}
We recall that $\mathbf{R}_T:=(\Delta  \boldsymbol{u}_\mathscr{T}-\nabla \pi_\mathscr{T})|^{}_T$. We now use that $\boldsymbol\Phi_S:=\phi_S  \mathbf{J}_S$ and invoke estimate \eqref{eq:bubble_property_4} and H\"older's inequality, to arrive at
\begin{multline}\label{eq:eff_3}
\|\mathbf{J}_S\|_{\mathbf{L}^2(S)}^2
\lesssim 
\|\mathbf{J}_S \phi_S^{\frac{1}{2}}\|_{\mathbf{L}^2(S)}^2
\lesssim 
\sum_{{T^\prime}\in\mathcal{N}_S}\left(\|\nabla\mathbf{e}_{\boldsymbol{u}}\|_{\mathbf{L}^p({T^\prime})}\|\nabla\boldsymbol\Phi_S\|_{\mathbf{L}^{p\prime}({T^\prime})} \right.\\
\left.+ \|e_\pi\|_{\mathrm{L}^p({T^\prime})} \|\text{div }\boldsymbol\Phi_S\|_{\mathrm{L}^{p\prime}({T^\prime})} + \|\mathbf{R}_{T^\prime}\|_{\mathbf{L}^p({T^\prime})}\|\boldsymbol \Phi_S\|_{\mathbf{L}^{p\prime}({T^\prime})}  \right).
\end{multline}
This estimate in conjunction with \eqref{eq:residual_estimate_aux} and an inverse inequality yield
\begin{equation*}\label{eq:eff_4aux}
\|\mathbf{J}_S\|_{\mathbf{L}^2(S)}^2
\lesssim 
\sum_{{T^\prime}\in\mathcal{N}_S} h_{T'}^{-1}\left(\|\nabla\mathbf{e}_{\boldsymbol{u}}\|_{\mathbf{L}^p({T^\prime})}+\|e_\pi\|_{\mathrm{L}^p({T^\prime})} \right) \|\boldsymbol \Phi_S\|_{\mathbf{L}^{p\prime}({T^\prime})}.
\end{equation*}
We now notice that $ \|\boldsymbol \Phi_S\|_{\mathbf{L}^{p\prime}({T^\prime})} \approx h_{T'}^{1/p'} \|\boldsymbol \Phi_S\|_{\mathbf{L}^{p\prime}({S})}$. This implies that
\begin{equation*}\label{eq:eff_4}
\|\mathbf{J}_S\|_{\mathbf{L}^2(S)}^2
\lesssim 
\sum_{{T^\prime}\in\mathcal{N}_S} h_{T'}^{-1/p} \left( \|\nabla\mathbf{e}_{\boldsymbol{u}}\|_{\mathbf{L}^p({T^\prime})} +\|e_\pi\|_{\mathrm{L}^p({T^\prime})} \right) \|\mathbf{J}_S\|_{\mathbf{L}^{p\prime}(S)}.
\end{equation*}
We thus invoke the estimate $\|\mathbf{J}_S\|_{\mathbf{L}^{p\prime}(S)} \lesssim h_{T'}^{(d-1)(1/{p\prime}-1/2)} \|\mathbf{J}_S\|_{\mathbf{L}^{2}(S)}$, which follows from a scaled--trace inequality and an inverse estimate, to arrive at
\begin{equation*}
\|\mathbf{J}_S\|_{\mathbf{L}^2(S)}^2
\lesssim \hspace{-0.2cm}
\sum_{{T^\prime}\in\mathcal{N}_S}\left(\|\nabla\mathbf{e}_{\boldsymbol{u}}\|_{\mathbf{L}^p({T^\prime})}+\|e_\pi\|_{\mathrm{L}^p({T^\prime})}
\right)h_{T'}^{-1/p  + (d-1)(1/{p\prime}-1/2)} \|\mathbf{J}_S\|_{\mathbf{L}^{2}(S)}.
\end{equation*}
This and
$
\|\mathbf{J}_S\|_{\mathbf{L}^{p}(S)} \lesssim h_{T'}^{(d-1)(1/p-1/2)} \|\mathbf{J}_S\|_{\mathbf{L}^{2}(S)},
$
allow us to arrive at
\begin{equation}\label{eq:jump_estimate}
h_T\|\mathbf{J}_S\|_{\mathbf{L}^p(S)}^p
\lesssim 
\sum_{{T^\prime}\in\mathcal{N}_S}\left(\|\nabla\mathbf{e}_{\boldsymbol{u}}\|_{\mathbf{L}^p({T^\prime})}^p+\|e_\pi\|_{\mathrm{L}^p({T^\prime})}^p\right).
\end{equation}

\underline{Step 4.} Let $T\in\mathscr{T}$. In this step we bound the remaining term $h_T^{d-p(d-1)} |\boldsymbol{f}|^p$ in \eqref{eq:local_indicator}. If $T\cap \{x_0\}=\emptyset$, then the desired estimate \eqref{eq:local_efficiency} follows directly from \eqref{eq:residual_estimate}, \eqref{eq:divergence_estimate}, and \eqref{eq:jump_estimate}. If $T \cap \{x_0\} \neq \emptyset$, and \ref{i} and \ref{ii} hold, the indicator $\eta_{p,T}$ contains the term $h_{_T}^{d-p(d-1)} |\boldsymbol{f}|^p$.  To control this term, we invoke the smooth function $\mu$, whose construction we owe to \cite[Section 3]{MR2262756} and is such that
\begin{equation*}
\mathcal{S}_\mu:=\text{supp}(\mu)\subset \mathcal{N}_T, \quad
\mu(x_0)=1,\quad \|\mu\|_{\mathrm{L}^\infty(\mathcal{S}_\mu)}=1,\quad \|\nabla \mu \|_{\mathrm{L}^\infty(\mathcal{S}_\mu)}\lesssim h_T^{-1}.
\end{equation*}
We also have the properties
\begin{equation}\label{eq:properties_bubble_mu}
\|\mu\|_{\mathrm{L}^{p\prime}(\mathcal{N}_T)}\lesssim h_T^{d/{p\prime}}, \qquad 
\|\mu\|_{\mathrm{L}^{p\prime}(S)}\lesssim h_T^{(d-1)/{p\prime}}.
\end{equation}

With this smooth function at hand, we define $\mathbf{F}=\mu| \boldsymbol{f}|^{p-1}\text{sign}( \boldsymbol{f})$. Here, the involved operations, i.e., the power, the absolute value and the sign function, must be understood componentwise. Define 
\begin{equation*}
 \mathscr{S}(\mathcal{N}_T):=\{S\in\mathscr{S}: S\in \partial T^\prime, S\not\in \partial\mathcal{N}_T, T^\prime\in \mathcal{N}_T\}.
\end{equation*}

Set $\boldsymbol{v}=\mathbf{F}$ and $q=0$ in \eqref{eq:identity_residual_error}. Invoke the identity \eqref{eq:identity_2aux} to arrive at
\begin{equation*}
c((\mathbf{e}_{\boldsymbol{u}},e_\pi),(\mathbf{F},0))
=
|\boldsymbol{f}|^p+\sum_{{T^\prime}\in\mathcal{N}_T}\int_{T^\prime}\mathbf{R}_{T^\prime}\cdot\mathbf{F}
+\sum_{S\in\mathscr{S}(\mathcal{N}_T)}\int_S \mathbf{J}_S\cdot \mathbf{F}.
\end{equation*}
Invoking H\"older’s inequality and suitable estimates for the function $\mu$, which are stated in \eqref{eq:properties_bubble_mu}, we obtain that
\begin{multline*}
|\boldsymbol{f}|^p \lesssim 
\sum_{{T^\prime}\in\mathcal{N}_T}\left(\|\mathbf{R}_{T^\prime}\|_{\mathbf{L}^{p}(T^\prime)}\|\mathbf{F}\|_{\mathbf{L}^{p\prime}(T^\prime)}+\|\nabla\mathbf{e}_{\boldsymbol{u}}\|_{\mathbf{L}^{p}(T^\prime)}\|\nabla\mathbf{F}\|_{\mathbf{L}^{p\prime}(T^\prime)}\right.\\
\left.+\|e_\pi\|_{\mathrm{L}^{p}(T^\prime)}\|\text{div}\mathbf{F}\|_{\mathrm{L}^{p\prime}(T^\prime)}\right) + \sum_{S\in\mathscr{S}(\mathcal{N}_T)}\|\mathbf{J}_S\|_{\mathbf{L}^{p}(S)}\|\mathbf{F}\|_{\mathbf{L}^{p\prime}(S)}\\
\lesssim
 \left[
 \sum_{{T^\prime}\in\mathcal{N}_T} 
 \left( h_{T^\prime}^{d/p\prime}\|\mathbf{R}_{T^\prime}\|_{\mathbf{L}^{p}(T^\prime)}+ h_{T^\prime}^{d/p\prime-1}\|\nabla\mathbf{e}_{\boldsymbol{u}}\|_{\mathbf{L}^{p}(T^\prime)} +h_{T^\prime}^{d/p\prime-1}\|e_\pi\|_{\mathrm{L}^{p}(T^\prime)} \right)
\right.\\
\left. +  \sum_{S\in\mathscr{S}(\mathcal{N}_T)}h_T^{(d-1)/{p\prime}}\|\mathbf{J}_S\|_{\mathbf{L}^{p}(S)}
\right]
|\boldsymbol{f}|^{p-1}.
\end{multline*}
In view of \eqref{eq:residual_estimate_aux} and \eqref{eq:jump_estimate}, we arrive at the desired estimate
\begin{equation}\label{eq:vector_estimate}
h_T^{d-p(d-1)}|\boldsymbol{f}|^p
\lesssim
\sum_{{T^\prime}\in\mathcal{N}_T}\left(\|\nabla\mathbf{e}_{\boldsymbol{u}}\|_{\mathbf{L}^p({T^\prime})}^p+\|e_\pi\|_{\mathrm{L}^p({T^\prime})}^p\right).
\end{equation}
\underline{Step 5.} Finally, by gathering the estimates \eqref{eq:residual_estimate}, \eqref{eq:divergence_estimate}, \eqref{eq:jump_estimate}, and \eqref{eq:vector_estimate} we arrive at the desired local lower bound \eqref{eq:local_efficiency}.\qed
\end{proof}


\section{A stabilized scheme}\label{sec:stabilized}
In Section \ref{sec:a_posteriori_estimates} we have designed and analyzed a posteriori error estimators for classical low--order inf--sup stable finite element approximations of problem \eqref{eq:weak_Stokes_system}; the involved pairs of finite elements satisfy the compatibility condition \eqref{eq:infsup_div}, which comes with a cost. This condition requires to increase the polynomial degree of the discrete spaces beyond what is required for conformity. If lowest order possible is desired, it is thus necessary to modify the discrete problem to circumvent the need of satisfying condition \eqref{eq:infsup_div}. This gives rise to the so--called stabilized finite element methods. Several stabilized techniques are available in the literature. For an extensive review of different stabilized finite element methods we refer the reader to \cite[Part IV, Section 3]{MR2454024}, \cite[Chapter 7]{MR2490235} and \cite[Chapter 4]{MR3561143}.

We now describe the low--order stabilized schemes that we will consider in our work. To present them, we introduce the finite element spaces
\begin{equation}\label{eq:V_estab}
\mathbf{V}_{\text{stab}}(\mathscr{T})=\{\boldsymbol{v}_{\mathscr{T}}\in\mathbf{C}(\overline{\Omega}) : \boldsymbol{v}_{\mathscr{T}}|_T\in\mathbb{P}_1(T)^{d} \ \forall\ T\in\mathscr{T}\}\cap\mathbf{W}_0^{1,p\prime}(\Omega),
\end{equation}
and
\begin{equation}\label{eq:P_estab}
\mathcal{P}_{\ell,\text{stab}}(\mathscr{T})=\{q_{\mathscr{T}}\in \mathrm{L}^{p\prime}(\Omega)/\mathbb{R} : q_{\mathscr{T}}|_T\in\mathbb{P}_{\ell}(T) \ \forall \ T\in\mathscr{T}\},
\end{equation}
where $\ell\in\{0,1\}$. With these spaces at hand, we propose the following stabilized finite element method: Find $(\boldsymbol{u}_{\mathscr{T}},\pi_{\mathscr{T}})\in\mathbf{V}_{\text{stab}}(\mathscr{T})\times\mathcal{P}_{\ell,\text{stab}}(\mathscr{T})$ such  that
\begin{equation}\label{eq:stab_Stokes_system}
\begin{array}{rcll}
 a(\boldsymbol{u}_{\mathscr{T}},\boldsymbol{v}_{\mathscr{T}})+b(\boldsymbol{v}_{\mathscr{T}},\pi_{\mathscr{T}}) +s(\boldsymbol{u}_{\mathscr{T}},\boldsymbol{v}_{\mathscr{T}})& = &\langle\boldsymbol{f}\delta_{x_0},\boldsymbol{v}_{\mathscr{T}}\rangle &\forall\ \boldsymbol{v}_{\mathscr{T}}\in\mathbf{V}_{\text{stab}}(\mathscr{T}), \\
-b(\boldsymbol{u}_{\mathscr{T}},q_{\mathscr{T}})+m(\pi_{\mathscr{T}},q_{\mathscr{T}}) & = & 0 &\forall\ q_{\mathscr{T}}\in\mathcal{P}_{\ell,\text{stab}}(\mathscr{T}),
\end{array}\hspace{-0.4cm}
\end{equation}
where $s:\mathbf{V}_{\text{stab}}(\mathscr{T})\times \mathbf{V}_{\text{stab}}(\mathscr{T})\rightarrow\mathbb{R}$ and  $m:\mathcal{P}_{\ell,\text{stab}}(\mathscr{T}) \times \mathcal{P}_{\ell,\text{stab}}(\mathscr{T})\rightarrow\mathbb{R}$ are defined by
\begin{equation*}\label{eq:form_s}
\displaystyle s(\boldsymbol{u}_{\mathscr{T}},\boldsymbol{v}_{\mathscr{T}}):=\sum_{T\in\mathscr{T}}\tau_{\text{div}}\int_T\text{div}\boldsymbol{u}_{\mathscr{T}}\text{div}\boldsymbol{v}_{\mathscr{T}},
\end{equation*}
and
\begin{equation*}\label{eq:form_m}
m(\pi_{\mathscr{T}},q_{\mathscr{T}}):=\sum_{T\in\mathscr{T}}\tau_T\int_T\nabla\pi_{\mathscr{T}}\cdot\nabla q_{\mathscr{T}}+\sum_{S\in\mathscr{S}}\tau_S h_S\int_S\llbracket{\pi_{\mathscr{T}}}\rrbracket\llbracket{q_{\mathscr{T}}}\rrbracket,
\end{equation*}
respectively. Here, $\tau_{\text{div}}\geq 0$, $\tau_T\geq 0$, and $\tau_S>0$ correspond to stabilization parameters. The well--posedness of problem \eqref{eq:stab_Stokes_system} follows from \cite[Lemma 3.4, Section 3.1]{MR2454024}, when $\tau_{T}>0$, and \cite[Section 2.1]{MR1740398}, when $\tau_T = 0$ and $\ell = 0$, in conjunction with the equivalence of norms on discrete spaces.

We must immediately notice that, in view of the stabilization terms $s(\cdot,\cdot)$ and $m(\cdot,\cdot)$ in problem \eqref{eq:stab_Stokes_system}, the Galerkin orthogonality is no longer valid. Instead, we have the following relation for $(\boldsymbol{v}_{\mathscr{T}},q_{\mathscr{T}})\in\mathbf{V}_{\text{stab}}(\mathscr{T})\times\mathcal{P}_{\ell,\text{stab}}(\mathscr{T})$:
\begin{equation*}
\langle\mathcal{R},(\boldsymbol{v}_{\mathscr{T}},q_{\mathscr{T}})\rangle_{\mathcal{Y}\prime\times\mathcal{Y}}=s(\boldsymbol{u}_{\mathscr{T}},\boldsymbol{v}_{\mathscr{T}})+m(\pi_{\mathscr{T}},q_{\mathscr{T}}).
\end{equation*}
We recall that $\mathcal{R}$ denotes the residual and is defined in \eqref{eq:identity_residual_error}. For $(\boldsymbol{v}_{\mathscr{T}},q_{\mathscr{T}})\in\mathbf{V}_{\text{stab}}(\mathscr{T})\times\mathcal{P}_{\ell,\text{stab}}(\mathscr{T})$, the previous relation can be rewritten as
\begin{equation}\label{eq:rel_1}
0=\langle\boldsymbol{f}\delta_{x_0},\boldsymbol{v}_{\mathscr{T}}\rangle-c((\boldsymbol{u}_\mathscr{T},\pi_\mathscr{T}),(\boldsymbol{v}_\mathscr{T},q_\mathscr{T}))-s(\boldsymbol{u}_\mathscr{T},\boldsymbol{v}_\mathscr{T})-m(\pi_\mathscr{T},q_\mathscr{T}).
\end{equation}

We now introduce local error indicators and a posteriori error estimators. Let $T \in \mathscr{T}$. If $x_0 \in T$ is such that $x_0$ is not a vertex of $T$, we define the element error indicators
\begin{multline}
\eta_{\text{stab},p,T}:=\Big( 
h_T^p \|\Delta\boldsymbol{u}_{\mathscr{T}}-\nabla\pi_{\mathscr{T}}\|_{\mathbf{L}^p(T)}^p 
+
h_T \|\llbracket{(\nabla\boldsymbol{u}_{\mathscr{T}}-\mathbb{I}_{d}\pi_{\mathscr{T}})\cdot\boldsymbol{\nu}}\rrbracket \|_{\mathbf{L}^p(\partial T \setminus \partial \Omega)}^p 
\\
+
(1+\tau_{\text{div}}^p)\|\text{div}\boldsymbol{u}_{\mathscr{T}}\|_{\mathrm{L}^p(T)}^p+h_T^{d-p(d-1)} |\boldsymbol{f}|^p \Big)^{\frac{1}{p}}.
\label{eq:local_indicator_stab}
\end{multline}
If $x_0\in T$ is a vertex of $T$, then
\begin{multline}
\eta_{\text{stab},p,T}:=\Big( 
h_T^p \|\Delta\boldsymbol{u}_{\mathscr{T}}-\nabla\pi_{\mathscr{T}}\|_{\mathbf{L}^p(T)}^p 
+
h_T \|\llbracket{(\nabla\boldsymbol{u}_{\mathscr{T}}-\mathbb{I}_{d}\pi_{\mathscr{T}})\cdot\boldsymbol{\nu}}\rrbracket \|_{\mathbf{L}^p(\partial T \setminus \partial \Omega)}^p 
\\
+
(1+\tau_{\text{div}}^p)\|\text{div}\boldsymbol{u}_{\mathscr{T}}\|_{\mathrm{L}^p(T)}^p\Big)^{\frac{1}{p}}.
\label{eq:local_indicator_stab_not_vertex}
\end{multline}
If $x_0 \notin T$, then the indicator $\eta_{\text{stab},p,T}$ is defined as in \eqref{eq:local_indicator_stab_not_vertex}. Here, $(\boldsymbol{u}_{\mathscr{T}},\pi_{\mathscr{T}})$ denotes the solution to the stabilized discrete problem \eqref{eq:stab_Stokes_system} and $\mathbb{I}_{d}$ denotes the identity matrix in $\mathbb{R}^{d\times d}$. We recall that we consider our elements $T$ to be closed sets. 

The a posteriori error estimators are thus defined by
\begin{equation}\label{eq:error_estimator_global}
\eta_{\text{stab},p}:=\left( \sum_{T\in\mathscr{T}} \eta_{\text{stab},p,T}^p\right)^{\frac{1}{p}}.
\end{equation}

We now derive global reliability and local efficiency properties for the error estimators $\eta_{\text{stab},p}$.

\begin{theorem}[reliability and local efficiency]\label{thm:rel_eff_stab}
Let $p \in (2d/(d+1) - \varepsilon, d/(d-1))$. Let $(\boldsymbol{u},\pi)\in\mathbf{W}_0^{1,p}(\Omega)\times\mathrm{L}^p(\Omega)/\mathbb{R}$ be the solution to \eqref{eq:weak_Stokes_system} and $(\boldsymbol{u}_{\mathscr{T}},\pi_{\mathscr{T}})\in\mathbf{V}_{\text{stab}}(\mathscr{T})\times\mathcal{P}_{\ell,\text{stab}}(\mathscr{T})$ its stabilized finite element approximation obtained as the solution to \eqref{eq:stab_Stokes_system}. Then
\begin{equation}\label{eq:rel_stab}
\|(\mathbf{e}_{\boldsymbol{u}},e_\pi)\|_{\mathcal{X}}
\lesssim
\eta_{\textnormal{stab},p},
\end{equation} 
and 
\begin{equation}\label{eq:local_efficiency_stab}
\eta_{\textnormal{stab},p,T}^{p}
\lesssim
\|\nabla\mathbf{e}_{\boldsymbol{u}}\|_{\mathbf{L}^p(\mathcal{N}_T)}^p+\|e_\pi\|_{\mathrm{L}^p(\mathcal{N}_T)}^p.
\end{equation}
The hidden constants are independent of the continuous and discrete solutions, the size of the elements of the mesh $\mathscr{T}$, and $\#\mathscr{T}$.
\end{theorem}
\begin{proof}
We first derive the reliability estimate \eqref{eq:rel_stab}. To accomplish this task, we first observe, from \eqref{eq:rel_1}, that
\begin{multline*}\label{eq:identity_stab}
\langle\mathcal{R},(\boldsymbol{v},q)\rangle_{\mathcal{Y}^\prime\times\mathcal{Y}}=\langle\boldsymbol{f}\delta_{x_0},\boldsymbol{v}-\mathcal{I}_{\mathscr{T}}\boldsymbol{v}\rangle-\sum_{T\in\mathscr{T}}\int_T(-\Delta \boldsymbol{u}_\mathscr{T}+\nabla \pi_\mathscr{T})\cdot(\boldsymbol{v}-\mathcal{I}_{\mathscr{T}}\boldsymbol{v}) \\-\sum_{T\in\mathscr{T}}\int_T \text{div}\boldsymbol{u}_\mathscr{T} q 
+\sum_{S \in \mathscr{S}}\int_S \llbracket{(\nabla \boldsymbol{u}_\mathscr{T}-\pi_\mathscr{T}\mathbb{I}_{d})\cdot\boldsymbol\nu}\rrbracket\cdot (\boldsymbol{v}-\mathcal{I}_{\mathscr{T}}\boldsymbol{v})+s(\boldsymbol{u}_\mathscr{T},\mathcal{I}_{\mathscr{T}}\boldsymbol{v}),
\end{multline*}
where $\mathcal{I}_{\mathscr{T}}$ denote the Lagrange interpolation operator of Section \ref{sec:interpolation}. 

A simple inspection of the right--hand side of the previous expression reveals that, with the exception $s(\boldsymbol{u}_\mathscr{T},\mathcal{I}_{\mathscr{T}}\boldsymbol{v})$, all the involved terms have been estimated in the proof of Theorem \ref{global_reliability}. To control $s(\boldsymbol{u}_\mathscr{T},\mathcal{I}_{\mathscr{T}}\boldsymbol{v})$ we proceed as follows:
\begin{align*}
|s(\boldsymbol{u}_\mathscr{T},\mathcal{I}_{\mathscr{T}}\boldsymbol{v})|&\leq\sum_{T\in\mathscr{T}}\int_{T}\tau_{\text{div}}|\text{div}\boldsymbol{u}_{\mathscr{T}}\text{div}\mathcal{I}_{\mathscr{T}}\boldsymbol{v}|\\
                                  &\leq\sum_{T\in\mathscr{T}}\tau_{\text{div}}\|\text{div}\boldsymbol{u}_{\mathscr{T}}\|_{\mathrm{L}^p(T)}\|\text{div}\mathcal{I}_{\mathscr{T}}\boldsymbol{v}\|_{\mathrm{L}^{p\prime}(T)}.
\end{align*}
We thus invoke the stability of the Lagrange interpolation operator \cite[Theorem 1.103]{MR2050138} to arrive at
\begin{equation*}
\displaystyle |s(\boldsymbol{u}_\mathscr{T},\mathcal{I}_{\mathscr{T}}\boldsymbol{v})| \lesssim \| \nabla\boldsymbol{v}\|_{\mathbf{L}^{p\prime}(\Omega)}\sum_{T\in\mathscr{T}}\left(\tau^p_{\text{div}} \|\text{div}\boldsymbol{u}_\mathscr{T}\|_{\mathrm{L}^p(T)}^p\right)^{\frac{1}{p}}.
\end{equation*}
This estimate combined with the estimates obtained in the proof of Theorem \ref{thm:reliability_estimate} yield \eqref{eq:rel_stab}.

The local efficiency \eqref{eq:local_efficiency_stab} is a direct consequence of the results of Theorem \ref{thm:efficiency} since the lower bound does not contain any consistency terms.\qed
\end{proof}

\section{Numerical Experiments}\label{sec:numericos}
We conduct a series of numerical examples that illustrate the performance of the devised a posteriori error estimators. To explore the performance of the estimators $\eta_p$, defined in \eqref{eq:error_estimator}, we consider the discrete system \eqref{eq:disc_Stokes_system} with the discrete spaces \eqref{eq:V_TH} and \eqref{eq:P_TH}. This setting will be referred to as \emph{Taylor--Hood approximation}. The performance of the error estimators $\eta_{\text{stab},p}$, defined in \eqref{eq:error_estimator_global}, is explored by solving the stabilized discrete system \eqref{eq:stab_Stokes_system} with the following finite element setting: the discrete spaces are \eqref{eq:V_estab} and \eqref{eq:P_estab}, with $\ell=0$, and the stabilization parameters are $\tau_{\text{div}}=0$, $\tau_T=0$, and $\tau_S=1/12$. This setting will be referred to as \emph{low--order stabilized approximation}.


\subsection{Implementation}\label{sec:implementation}

All the experiments have been carried out with the help of a code that we implemented using \texttt{C++}. All matrices have been assembled exactly. The right hand sides of the assembled systems, the local indicators, and the approximation errors, are computed by a quadrature formula which is exact for polynomials of degree $19$ for two dimensional domains and degree $14$ for three dimensional domains. The linear systems were solved using the multifrontal massively parallel sparse direct solver (MUMPS) \cite{MUMPS1,MUMPS2}. 

For a given partition $\mathscr{T}$ we seek $(\boldsymbol{u}_{\mathscr{T}},\pi_{\mathscr{T}})$ that solves the discrete system \eqref{eq:disc_Stokes_system} or the stabilized discrete scheme \eqref{eq:stab_Stokes_system}. We thus compute the local error indicators $\eta_{p,T}$ or $\eta_{\text{stab},p,T}$ to drive the adaptive procedure described in \textbf{Algorithm} \ref{Algorithm} and compute the global error estimators $\eta_p$ or $\eta_{\text{stab},p}$ in order to assess the accuracy of the approximation. A sequence of adaptively refined meshes is thus generated from the initial meshes shown in Figure \ref{fig:initial_meshes}. For \emph{Taylor--Hood approximation}, the total number of degrees of freedom is $\mathsf{Ndof} := \dim(\mathbf{V}(\mathscr{T}))+\dim(\mathcal{P}(\mathscr{T}))$,
where $(\mathbf{V}(\mathscr{T}),\mathcal{P}(\mathscr{T}))$ is given by \eqref{eq:V_TH}--\eqref{eq:P_TH}. For \emph{low--order stabilized approximation}, $\mathsf{Ndof} := \dim(\mathbf{V}_{\text{stab}}(\mathscr{T}))+\dim(\mathcal{P}_{\ell,\text{stab}}(\mathscr{T}))$, where $(\mathbf{V}_{\text{stab}}(\mathscr{T}),\mathcal{P}_{\ell,\text{stab}}(\mathscr{T}))$ is given by \eqref{eq:V_estab}--\eqref{eq:P_estab} with $\ell = 0$. The error is measured in the norm $\|(\mathbf{e}_{\boldsymbol{u}},e_\pi)\|_\mathcal{X}$.

\begin{figure}[!ht]
\centering
\begin{minipage}{0.3\textwidth}\centering
\includegraphics[trim={0 0 0 0},clip,width=3.5cm,height=3.0cm,scale=0.2]{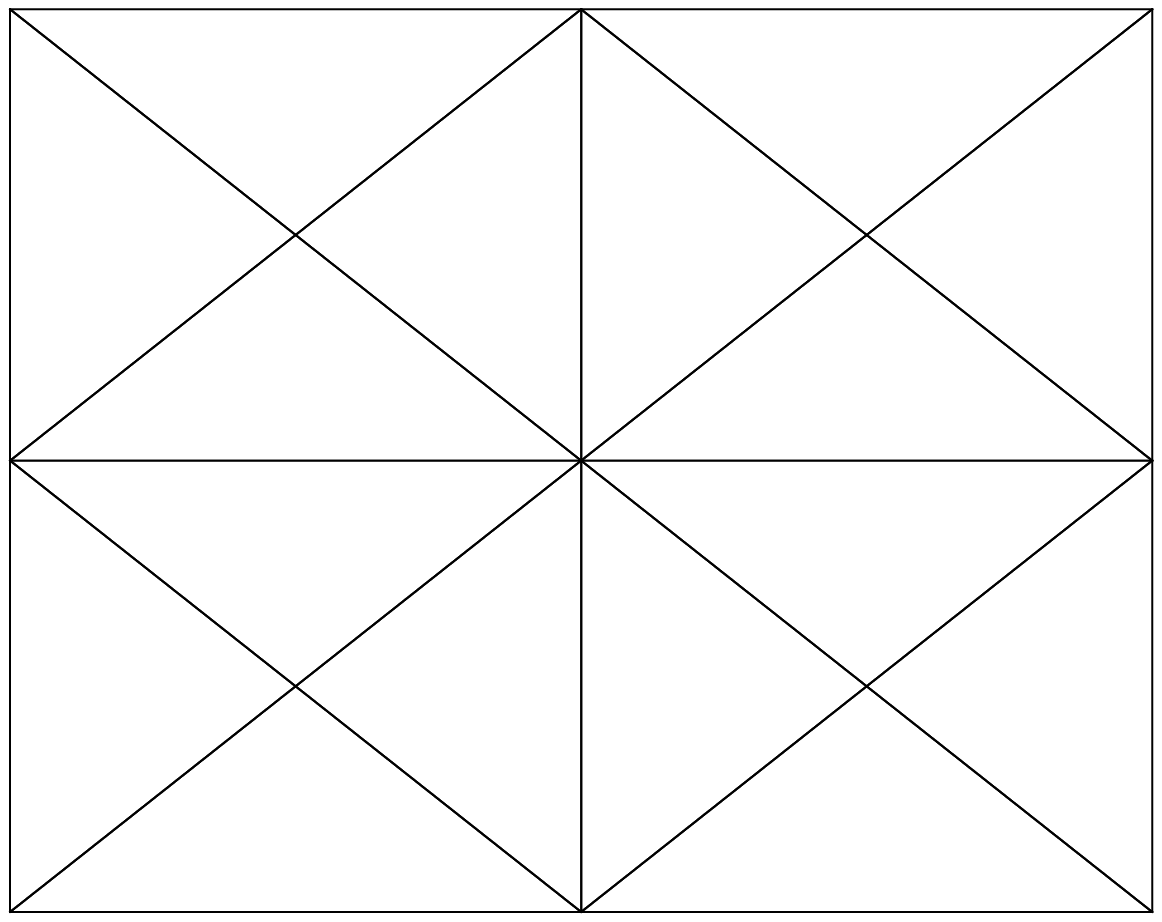}
\end{minipage}
\begin{minipage}{0.32\textwidth}\centering
\includegraphics[trim={0 0 0 0},clip,width=3.6cm,height=3.0cm,scale=0.2]{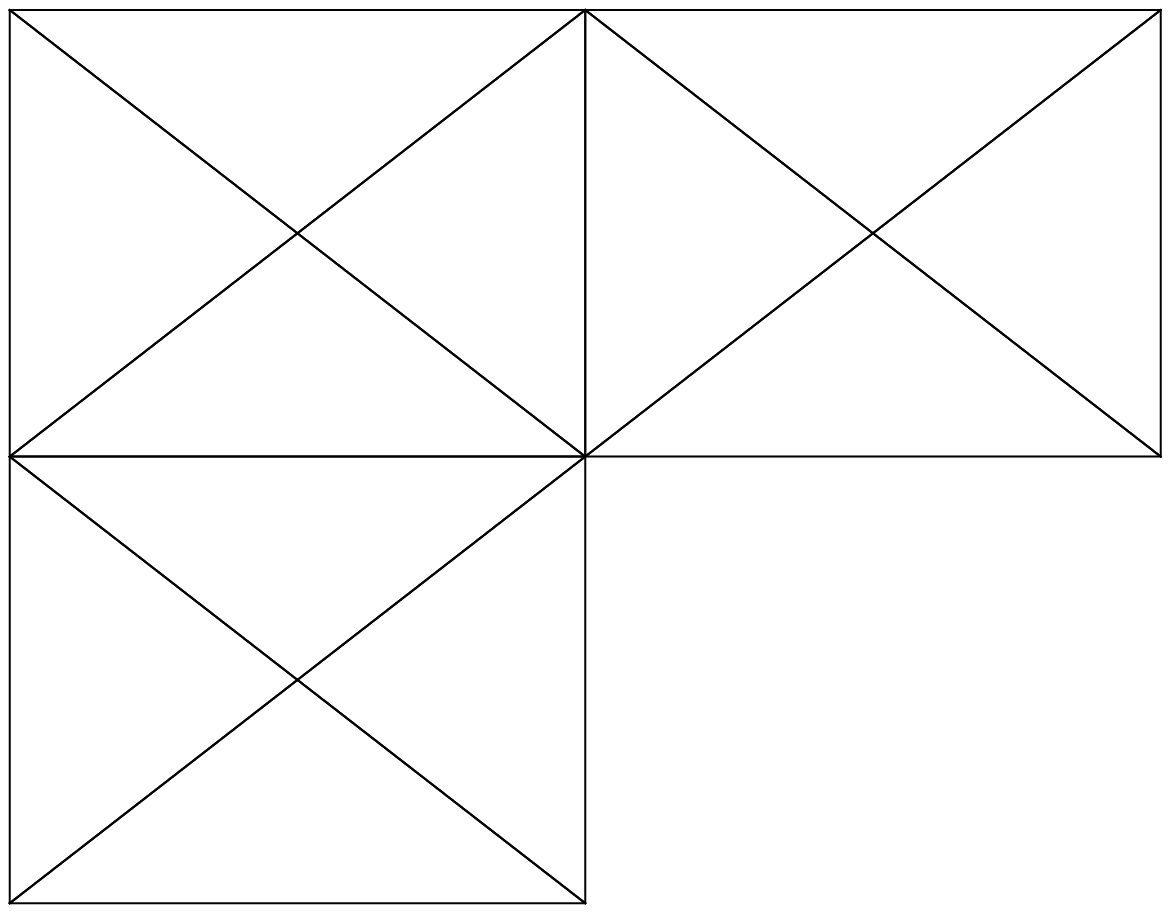}
\end{minipage}
\begin{minipage}{0.3\textwidth}\centering
\includegraphics[trim={0 0 0 0},clip,width=3.2cm,height=3.2cm,scale=0.2]{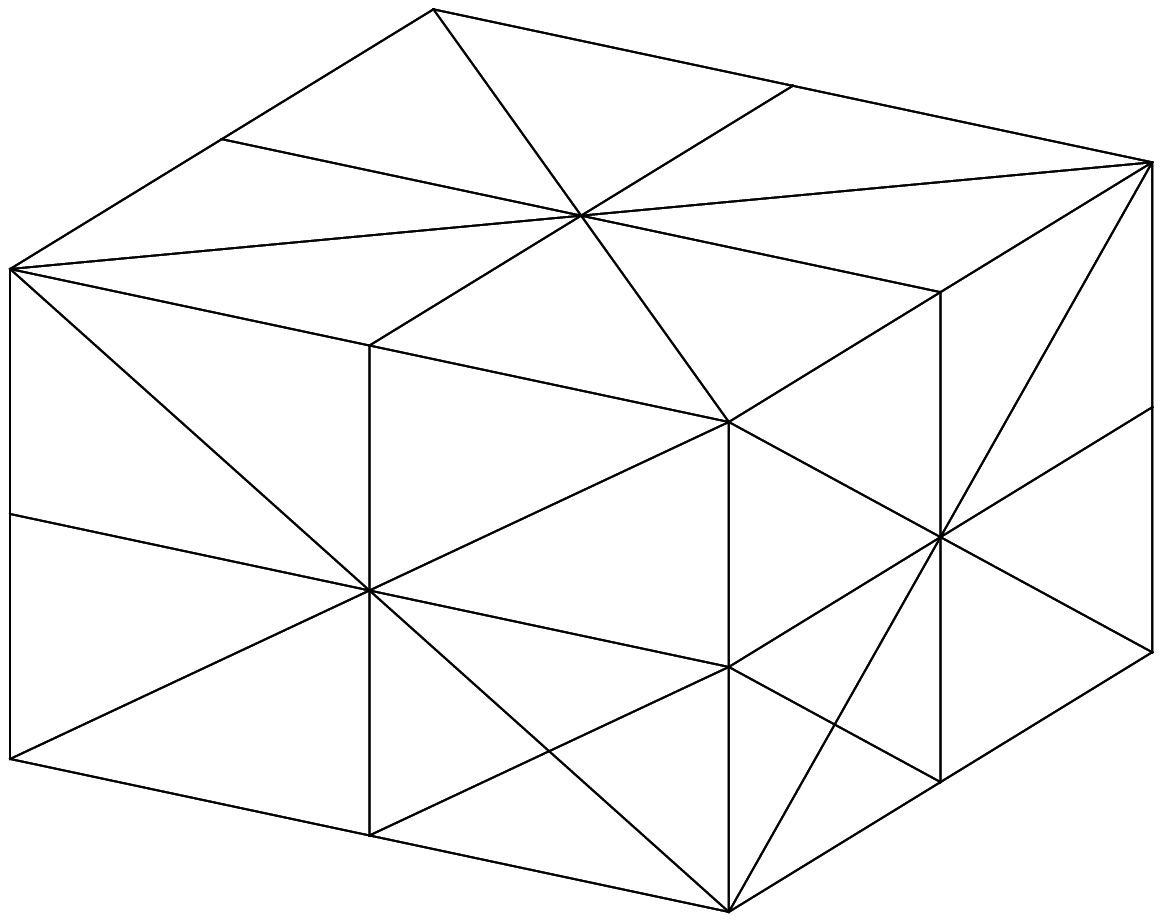}
\end{minipage}
\caption{The initial meshes used in the adaptive \textbf{Algorithm} \ref{Algorithm} when the domain $\Omega$ is a square (Example 1), a two dimensional L--shape (Examples 2 and 4), and a cube (Examples 3 and 5).}
\label{fig:initial_meshes}
\end{figure}

\begin{algorithm}[ht]
\caption{\textbf{Adaptive algorithm.}}
\label{Algorithm}
\SetKwInput{set}{Set}
\SetKwInput{ase}{Problem solution}
\SetKwInput{al}{Adaptive loop}
\SetKwInput{Input}{Input}
\Input{Initial mesh $\mathscr{T}_{0}$, subset $\mathcal{D}$, vectors $\{\boldsymbol{f}_t\}_{t\in\mathcal{D}}$, and stabilization parameters.}
\set{$i=0$.}
Solve the discrete system \eqref{eq:disc_Stokes_system} (\eqref{eq:stab_Stokes_system});
\\
For each $T\in\mathscr{T}_i$ compute the local error indicator $\eta_{p,T}$ ($\eta_{\text{stab},p,T}$) given as in \eqref{eq:local_indicator}--\eqref{eq:local_indicator2} (\eqref{eq:local_indicator_stab}--\eqref{eq:local_indicator_stab_not_vertex});
\\
Mark an element $T$ for refinement if 
\[
\eta_{p,T}^p > \displaystyle\frac{1}{2}\max_{T'\in \mathscr{T}}\eta_{p,T'}^p \quad
\left(\eta^p_{\text{stab},p,T}> \displaystyle\frac{1}{2}\max_{T'\in \mathscr{T}}\eta^p_{\text{stab},p,T'}\right);
\]
\\
From step $\boldsymbol{3}$, construct a new mesh, using a longest edge bisection algorithm. Set $i \leftarrow i + 1$, and go to step $\boldsymbol{1}$.
\\
\end{algorithm}
\normalsize

In the experiments that we perform we go beyond the presented theory and include a series of Dirac delta sources on the right--hand side of the momentum equation. To make matters precise, we will replace the momentum equation in \eqref{def:Stokes_singular_rhs} by
\begin{equation*}\label{eq:equation_combination_deltas}
-\Delta \boldsymbol{u} + \nabla \pi = \sum_{t \in \mathcal{D}}\boldsymbol{f}_{t} \delta_{t},
\end{equation*}
where $\mathcal{D}$ corresponds to a finite ordered subset of $\Omega$ with cardinality $\# \mathcal{D}$ and $\{\boldsymbol{f}_{t}\}_{t \in \mathcal{D}} \subset \mathbb{R}^{d}$. We thus propose the following a posteriori error estimator when Taylor--Hood
approximation is considered:
\begin{equation*}\label{eq:error_estimator_zeta}
\zeta_{p} := \left(\sum_{T \in \mathscr{T} }\zeta_{p,T}^p\right)^{\frac{1}{p}}.
\end{equation*}
For each $T\in\mathscr{T}$, the local error indicators are given by: If $t \in \mathcal{D} \cap T$ and \ref{i} or \ref{ii} hold, then
\begin{multline}\label{eq:local_indicator_zeta}
\zeta_{p,T}:=\Big(h_T^p \|\Delta\boldsymbol{u}_{\mathscr{T}}-\nabla\pi_{\mathscr{T}}\|_{\mathbf{L}^p(T)}^p + h_T\|\llbracket{(\nabla\boldsymbol{u}_{\mathscr{T}}-\mathbb{I}_{d}\pi_{\mathscr{T}})\cdot\boldsymbol{\nu}}\rrbracket \|_{\mathbf{L}^p(\partial T \setminus \partial \Omega)}^p  \\
+\|\text{div}\boldsymbol{u}_{\mathscr{T}}\|_{\mathrm{L}^p(T)}^p + \sum_{t \in \mathcal{D} \cap T}h_T^{d-p(d-1)} |\boldsymbol{f}_{t}|^p\Big)^{\frac{1}{p}}.
\end{multline}
If $t \in \mathcal{D} \cap T$ and \ref{i} or \ref{ii} do not hold, then
\begin{multline}\label{eq:local_indicator_zeta_2}
\zeta_{p,T}:=\Big(h_T^p \|\Delta\boldsymbol{u}_{\mathscr{T}}-\nabla\pi_{\mathscr{T}}\|_{\mathbf{L}^p(T)}^p + h_T\|\llbracket{(\nabla\boldsymbol{u}_{\mathscr{T}}-\mathbb{I}_{d}\pi_{\mathscr{T}})\cdot\boldsymbol{\nu}}\rrbracket \|_{\mathbf{L}^p(\partial T \setminus \partial \Omega)}^p \\
+\|\text{div}\boldsymbol{u}_{\mathscr{T}}\|_{\mathrm{L}^p(T)}^p\Big)^{\frac{1}{p}}.
\end{multline}
If $T \cap \mathcal{D} = \emptyset$, then the indicator is defined as in \eqref{eq:local_indicator_zeta_2}. Notice that, when $\#\mathcal{D}=1$, the total error estimator $\zeta_p$ coincides with $\eta_p$, which is defined in \eqref{eq:error_estimator}. 

Similarly, when the \emph{low--order stabilized approximation} scheme is considered, we propose the error estimator
\begin{equation*}\label{eq:estimator_stab_zeta}
\zeta_{\text{stab},p}:=\left( \sum_{T\in\mathscr{T}} \zeta_{\text{stab},p,T}^p\right)^{\frac{1}{p}}.
\end{equation*}
For each $T\in \mathscr{T}$, the local error indicators are given by: If $t \in \mathcal{D}\cap T$ is such that $t$ is not a vertex of $T$, then
\begin{multline}
\zeta_{\text{stab},p,T}:=\Big( 
h_T^p \|\Delta\boldsymbol{u}_{\mathscr{T}}-\nabla\pi_{\mathscr{T}}\|_{\mathbf{L}^p(T)}^p 
+
h_T \|\llbracket{(\nabla\boldsymbol{u}_{\mathscr{T}}-\mathbb{I}_{d}\pi_{\mathscr{T}})\cdot\boldsymbol{\nu}}\rrbracket \|_{\mathbf{L}^p(\partial T \setminus \partial \Omega)}^p 
\\
+
(1+\tau_{\text{div}}^p)\|\text{div}\boldsymbol{u}_{\mathscr{T}}\|_{\mathrm{L}^p(T)}^p+\sum_{t \in \mathcal{D} \cap T}h_T^{d-p(d-1)} |\boldsymbol{f}_t|^p \Big)^{\frac{1}{p}}.
\label{eq:local_indicator_stab_zeta}
\end{multline}
If $t \in \mathcal{D}\cap T$ is a vertex of $T$, then
\begin{multline}
\zeta_{\text{stab},p,T}:=\Big( 
h_T^p \|\Delta\boldsymbol{u}_{\mathscr{T}}-\nabla\pi_{\mathscr{T}}\|_{\mathbf{L}^p(T)}^p 
+
h_T \|\llbracket{(\nabla\boldsymbol{u}_{\mathscr{T}}-\mathbb{I}_{d}\pi_{\mathscr{T}})\cdot\boldsymbol{\nu}}\rrbracket \|_{\mathbf{L}^p(\partial T \setminus \partial \Omega)}^p 
\\
+
(1+\tau_{\text{div}}^p)\|\text{div}\boldsymbol{u}_{\mathscr{T}}\|_{\mathrm{L}^p(T)}^p\Big)^{\frac{1}{p}}.
\label{eq:local_indicator_stab_not_vertex_zeta}
\end{multline}
If $T \cap \mathcal{D} = \emptyset$, then the indicator is defined as in \eqref{eq:local_indicator_stab_not_vertex_zeta}.

In the following numerical examples, when it corresponds, we replace $\eta_{p,T}$ by $\zeta_{p,T}$ and $\eta_{\text{stab},T}$ by $\zeta_{\text{stab},T}$ in \textbf{Algorithm} \ref{Algorithm}.

We consider problems with homogeneous boundary conditions whose exact solutions are not known. We also consider problems with inhomogeneous Dirichlet boundary conditions whose exact solutions are known. Notice that this violates the assumption of homogeneous Dirichlet boundary conditions which is needed for the analysis that we have performed. In this case, we write the solution $(\boldsymbol{u},\pi)$ in terms of fundamental solutions of the Stokes equations \cite[Section IV.2]{MR2808162}:
\begin{equation}\label{def:fund_sol}
\boldsymbol{u}(x):= 
\sum_{t\in\mathcal{D}}
\sum_{i=1}^{d}
\widetilde{\mathbf{T}}_{t}(x)\cdot \mathbf{e}_{i} ,\qquad 
\pi(x):= 
\sum_{t\in\mathcal{D}}
\sum_{i=1}^{d}
\mathbf{T}_{t}(x)\cdot \mathbf{e}_{i},
\end{equation}
where, if $\mathbf{r}_{t} = x - t$ and $\mathbb{I}_d$ is the identity matrix in $\mathbb{R}^{d\times d}$, then 
\begin{align*}
\begin{array}{c}\displaystyle
\widetilde{\mathbf{T}}_{t}(x)
=
\left\{\begin{array}{ll}
-\dfrac{1}{4\pi}\bigg(\log|\mathbf{r}_{t}|\mathbb{I}^{}_2
-\dfrac{\mathbf{r}_{t}\mathbf{r}_{t}^{\intercal}}{|\mathbf{r}_{t}|^2}
\bigg), & \text{if }  d = 2, \\
\dfrac{1}{8\pi}\bigg(\dfrac{1}{|\mathbf{r}_{t}|}\mathbb{I}^{}_3
+\dfrac{\mathbf{r}_{t}\mathbf{r}_{t}^{\intercal}}{|\mathbf{r}_{t}|^3}
\bigg), & \text{if } d = 3;
\end{array}
\right. \displaystyle
\mathbf{T}_{t}(x)=\left\{\begin{array}{ll}
-\dfrac{\mathbf{r}_{t}}{2\pi|\mathbf{r}_{t}|^{2}}, & \text{if } d = 2, \\
-\dfrac{\mathbf{r}_{t}}{4\pi|\mathbf{r}_{t}|^{3}}, & \text{if } d = 3;
\end{array}
\right.
\end{array}
\end{align*} 
$\{ \mathbf{e}_{i} \}_{i=1}^d$ denotes the canonical basis of $\mathbb{R}^{d}$.

\subsection{Taylor--Hood approximation}\label{sec:T-H_approx}

We perform two and three dimensional examples on convex and nonconvex domains and with different number of source points. 
\\~\\
\textbf{Example 1 (Convex domain).} We consider $\Omega=(0,1)^2$ and 
\[
 \mathcal{D}=\{(0.25,0.25),(0.25,0.75),(0.75,0.25),(0.75,0.75).
\]
The solution $(\boldsymbol{u},\pi)$ is given as in \eqref{def:fund_sol}.

In this example we investigate the effect of varying the integrability index $p$. Notice that, since problem \eqref{eq:weak_Stokes_system} is well--posed for $p \in (4/3-\varepsilon,2)$, the solution $(\boldsymbol{u},\pi)$ belongs to $\mathbf{W}^{1,p}(\Omega) \times \mathrm{L}^p(\Omega)/\mathbb{R}$ for every $p<2$. In the particular setting of Example 1, we will thus consider $p \in \{1.2,1.4,1.6,1.8\}$. Finally, for each integrability index $p$, we compute the effectivity index $\mathscr{I}_{p} := \zeta_p / \|(\mathbf{e}_{\boldsymbol{u}},e_\pi)\|_\mathcal{X}$.

In Figures \ref{fig:ex-1.1} and \ref{fig:ex-1.2} we present the results obtained for Example 1. In particular, Figure \ref{fig:ex-1.1} presents, for different values of the integrability index $p \in \{1.2,1.4,1.6,1.8\}$, experimental rates of convergence for $\zeta_p$ and $\|(\mathbf{e}_{\boldsymbol{u}},e_\pi)\|_\mathcal{X}$, effectivity indices $\mathscr{I}_p$, and adaptively refined meshes. We observe, in subfigures (A.1)--(D.1), optimal experimental rates of convergence for the error estimators $\zeta_p$ and the total error $\|(\mathbf{e}_{\boldsymbol{u}},e_\pi)\|_\mathcal{X}$. In subfigures (A.3)--(D.3), we appreciate the effect of varying the integrability index $p$ on the adaptively refined meshes. In particular, we observe that the adaptive refinement is mostly concentrated on the points $t\in\mathcal{D}$ where the Dirac measures are supported. We also observe that the effectivity indices $\mathscr{I}_p$ decrease as the index $p$ increases; see subfigures (A.2)--(D.2). Finally, all the effectivity indices are stabilized around values between 6 and 13. This shows the accuracy of the proposed a posteriori error estimators $\zeta_p$ when used in the adaptive loop described in Algorithm \ref{Algorithm}. 
In Figure \ref{fig:ex-1.2} we present experimental rates of convergence for $\|(\mathbf{e}_{\boldsymbol{u}},e_\pi)\|_\mathcal{X}$ and $\zeta_p$ for uniform and adaptive refinement when $p = 1.05$. From subfigures (A)--(B) we observe that the devised adaptive loop outperforms uniform refinement. Moreover, adaptive refinement exhibits an optimal experimental rate of convergence. 

~\\
\textbf{Example 2 (L-shaped domain).} We let $\Omega=(0,1)^2 \setminus [0.5,1)\times (0,0.5]$, $p \in \{1.05,1.2,1.4,1.6,1.8\}$, $
\mathcal{D}=\{(0.25,0.25),(0.25,0.75),(0.75,0.75)\}$, and
\[
 \boldsymbol{f}_{(0.25,0.25)}=(4,4), \quad \boldsymbol{f}_{(0.25,0.75)}=(6,6), \quad \boldsymbol{f}_{(0.75,0.75)}=(-4,-4).
\]

In Figure \ref{fig:ex_2} we report the results obtained for Example 2. We present the finite element approximations of $|{\boldsymbol{u}}_\mathscr{T}|$ and $\pi_\mathscr{T}$, experimental rates of convergence for the error estimators $\zeta_p$, and adaptively refined meshes. We observe, in subfigure (A), that, for all the values of $p \in \{1.05,1.2,1.4,1.6,1.8\}$, optimal experimental rates of convergence for the total error estimators $\zeta_p$ are attained. We also observe, in the adaptively refined meshes (D)--(F), that the refinement is being concentrated around the re--entrant corner and the source points ($p=1.4$). 

\begin{figure}[!ht]
\centering
\psfrag{error}{{\huge $\|(\mathbf{e}_{\boldsymbol{u}},e_{\pi})\|_{\mathcal{X}}$}}
\psfrag{estimador}{{\Huge $\zeta_p$}}
\psfrag{O(Ndofs-1)}{\huge$\mathsf{Ndof}^{-1}$}
\psfrag{O(Ndofs-11)}{\huge$\mathsf{Ndof}^{-1}$}
\psfrag{O(Ndofs-111)}{\huge$\mathsf{Ndof}^{-1}$}
\psfrag{Ndofs}{{\huge $\mathsf{Ndof}$}}
\psfrag{IEE}{{\huge $\mathscr{I}_{p}$}}
\begin{minipage}[c]{0.35\textwidth}
\centering
\psfrag{Example 1 - p12}{\hspace{-0.7cm}\huge{$\|(\mathbf{e}_{\boldsymbol{u}},e_{\pi})\|_{\mathcal{X}}$ and $\zeta_p$ for $p = 1.2$}}
\includegraphics[trim={0 0 0 0},clip,width=3.7cm,height=3.5cm,scale=0.65]{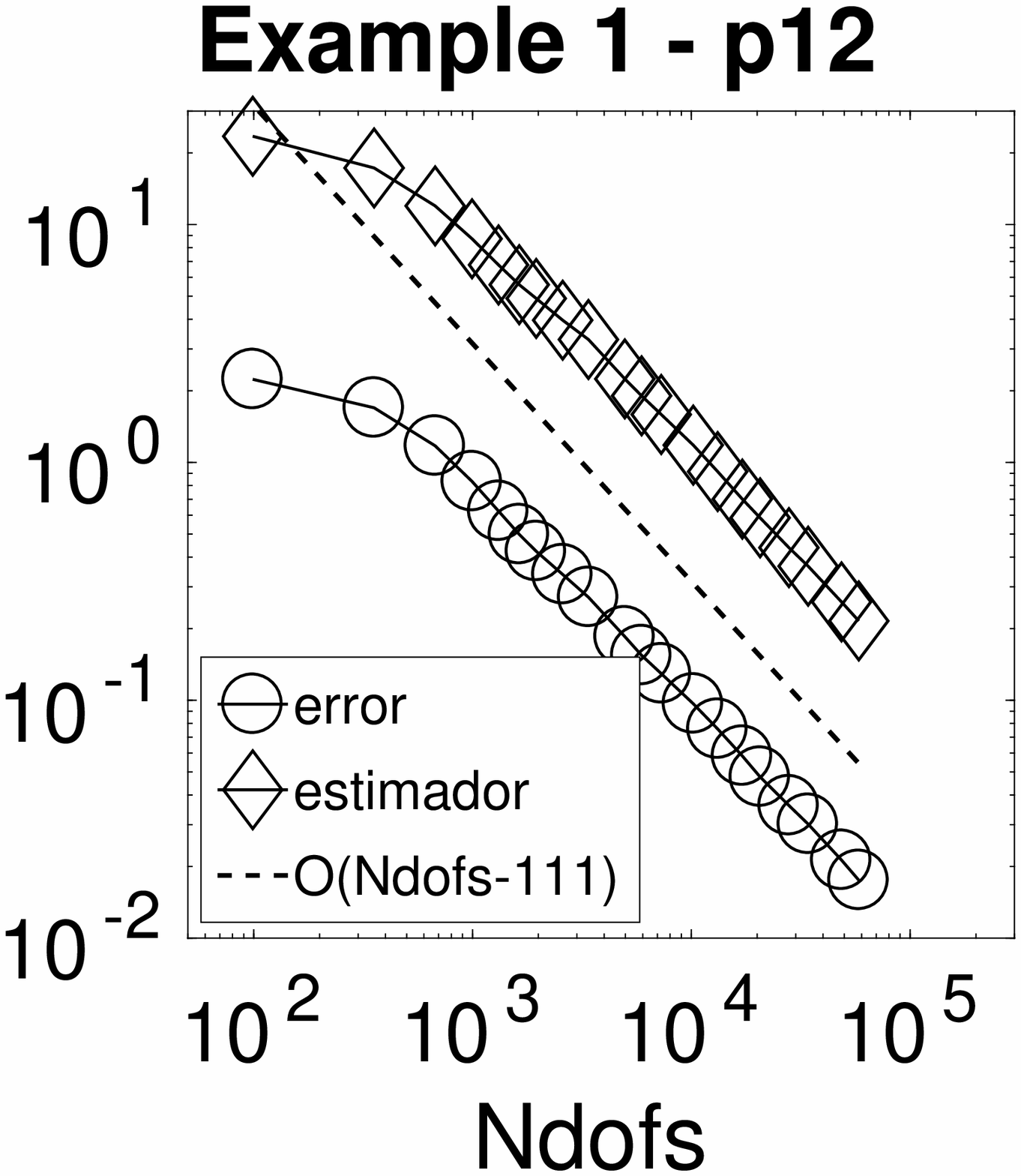} \\
\tiny{(A.1)}
\\~\\
\psfrag{Example 1 - p14}{\hspace{-0.7cm}\huge{$\|(\mathbf{e}_{\boldsymbol{u}},e_{\pi})\|_{\mathcal{X}}$ and $\zeta_p$ for $p = 1.4$}}
\includegraphics[trim={0 0 0 0},clip,width=3.7cm,height=3.5cm,scale=0.65]{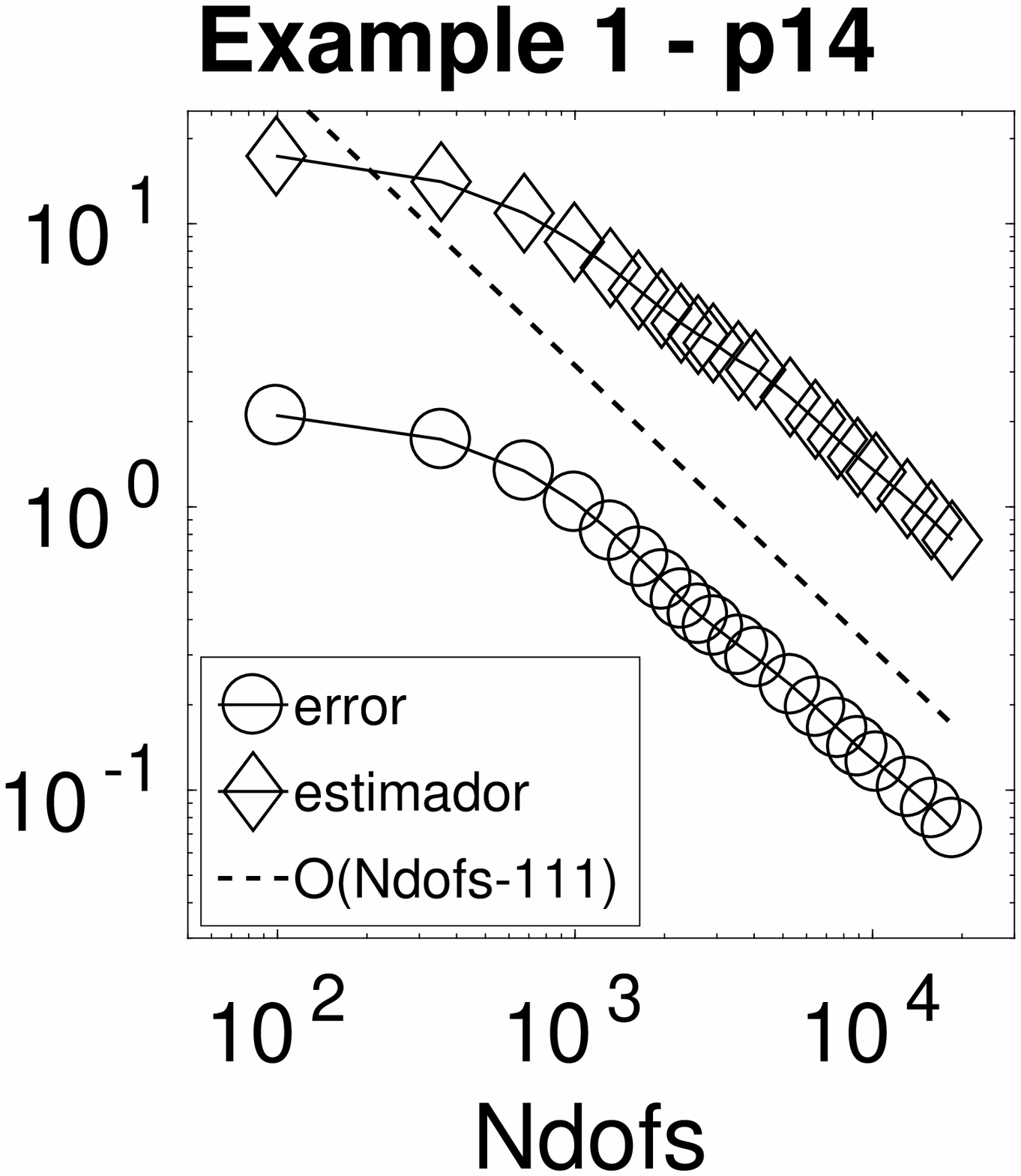} \\
\tiny{(B.1)}
\\~\\
\psfrag{Example 1 - p16}{\hspace{-0.7cm}\huge{$\|(\mathbf{e}_{\boldsymbol{u}},e_{\pi})\|_{\mathcal{X}}$ and $\zeta_p$ for $p = 1.6$}}
\includegraphics[trim={0 0 0 0},clip,width=3.7cm,height=3.5cm,scale=0.65]{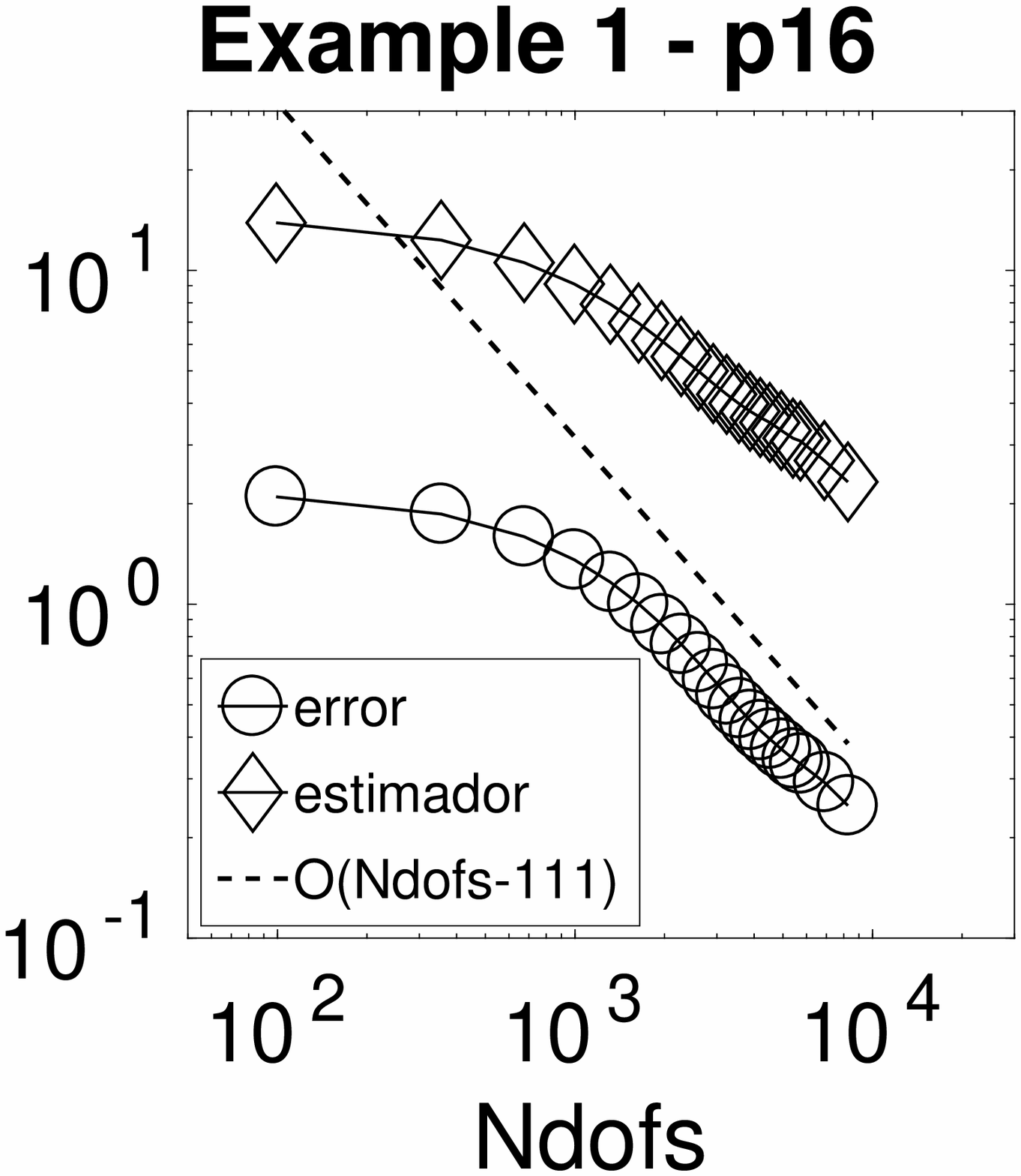} \\
\tiny{(C.1)}
\\~\\
\psfrag{Example 1 - p18}{\hspace{-0.7cm}\huge{$\|(\mathbf{e}_{\boldsymbol{u}},e_{\pi})\|_{\mathcal{X}}$ and $\zeta_p$ for $p = 1.8$}}
\includegraphics[trim={0 0 0 0},clip,width=3.5cm,height=3.5cm,scale=0.65]{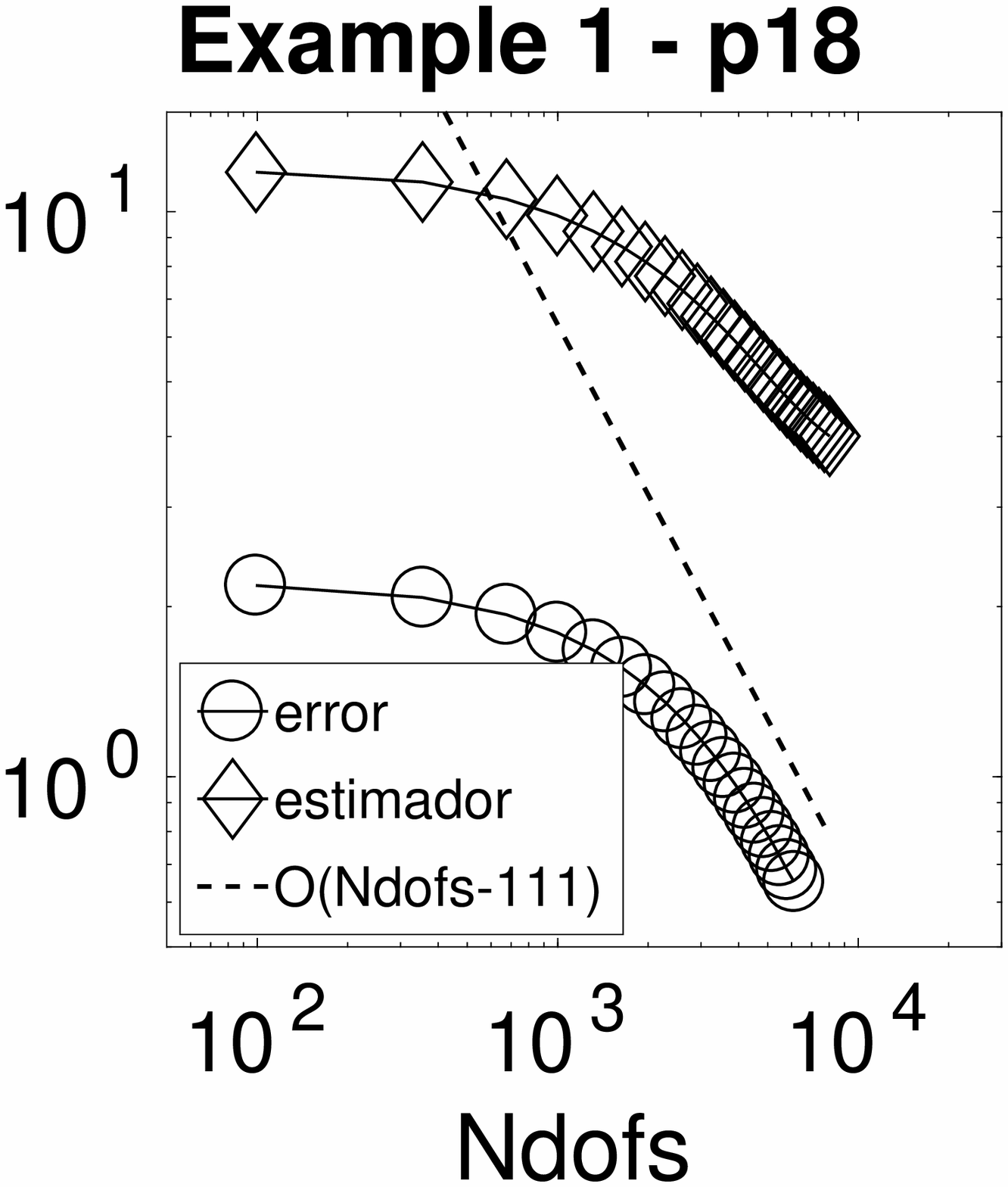} \\
\tiny{(D.1)}
\end{minipage}
\begin{minipage}[c]{0.35\textwidth}
\centering
\psfrag{indice ef}{\hspace{-1.8cm}\huge{Effectivity index for $p = 1.2$}}
\includegraphics[trim={0 0 0 0},clip,width=4.0cm,height=3.5cm,scale=0.65]{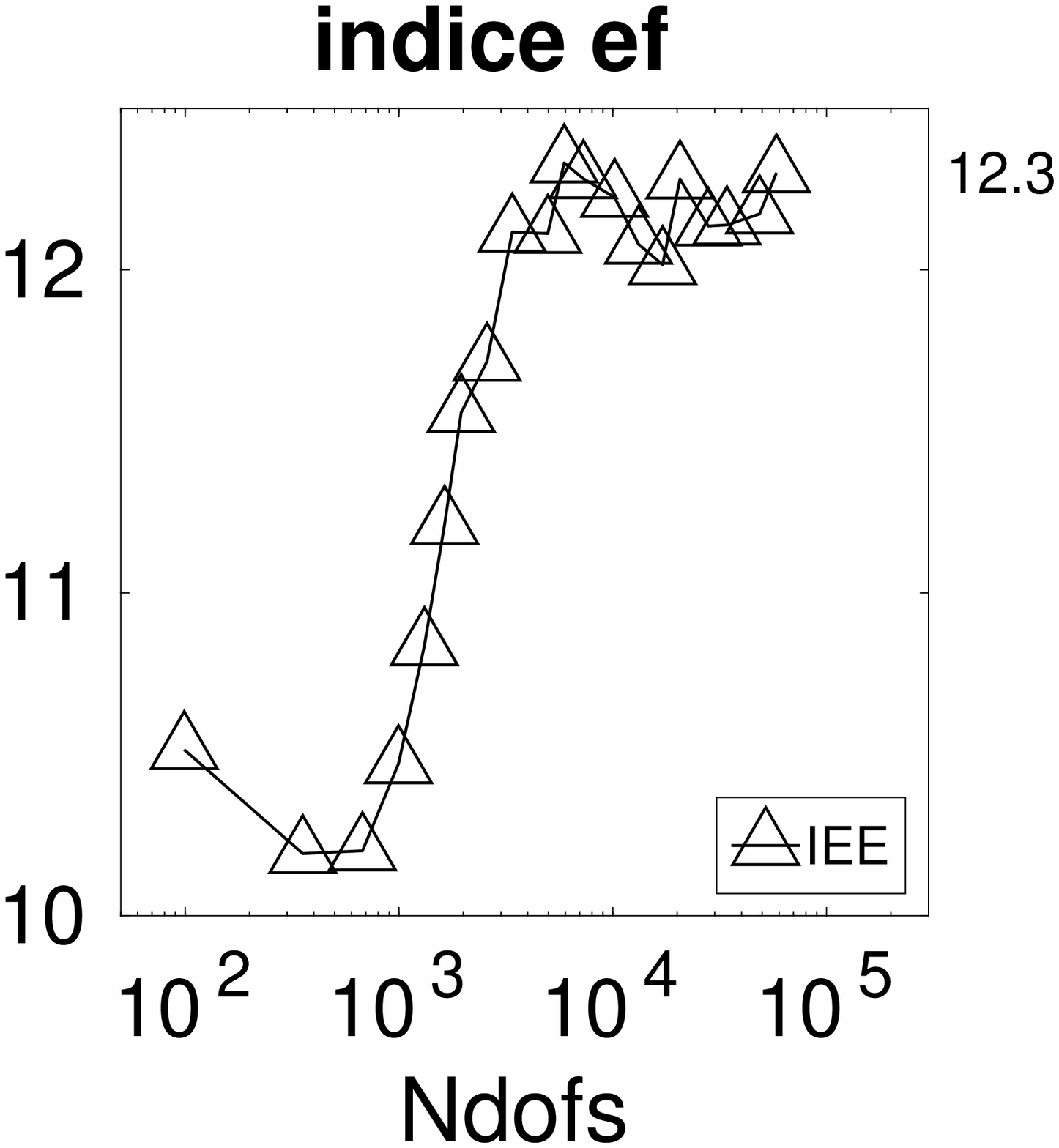} \\
\hspace{-0.8cm}\tiny{(A.2)}\\~\\
\psfrag{indice ef}{\hspace{-1.8cm}\huge{Effectivity index for $p = 1.4$}}
\includegraphics[trim={0 0 0 0},clip,width=4.0cm,height=3.5cm,scale=0.65]{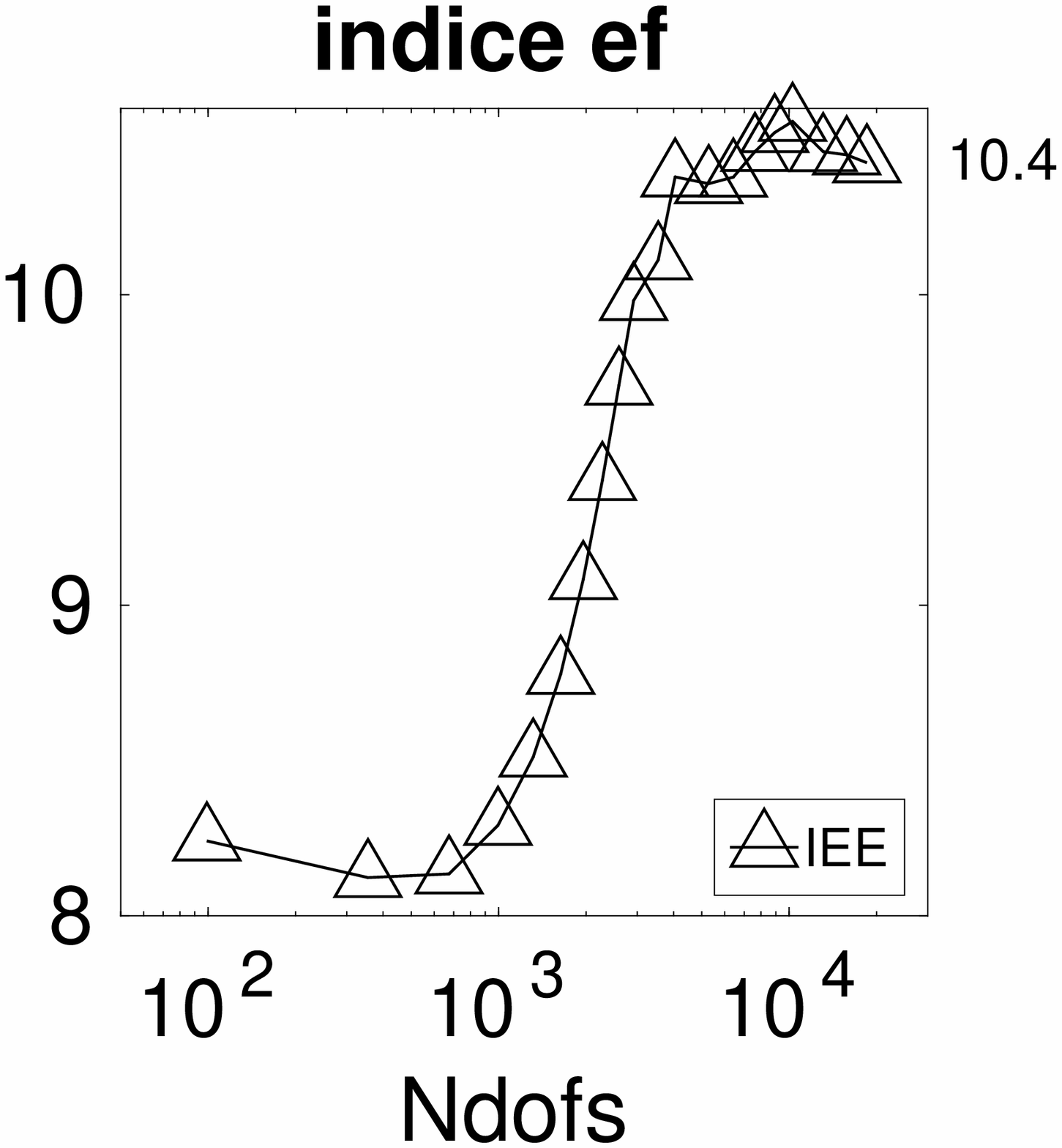} \\
\hspace{-0.8cm}\tiny{(B.2)}\\~\\
\psfrag{indice ef}{\hspace{-1.8cm}\huge{Effectivity index for $p = 1.6$}}
\includegraphics[trim={0 0 0 0},clip,width=3.8cm,height=3.5cm,scale=0.65]{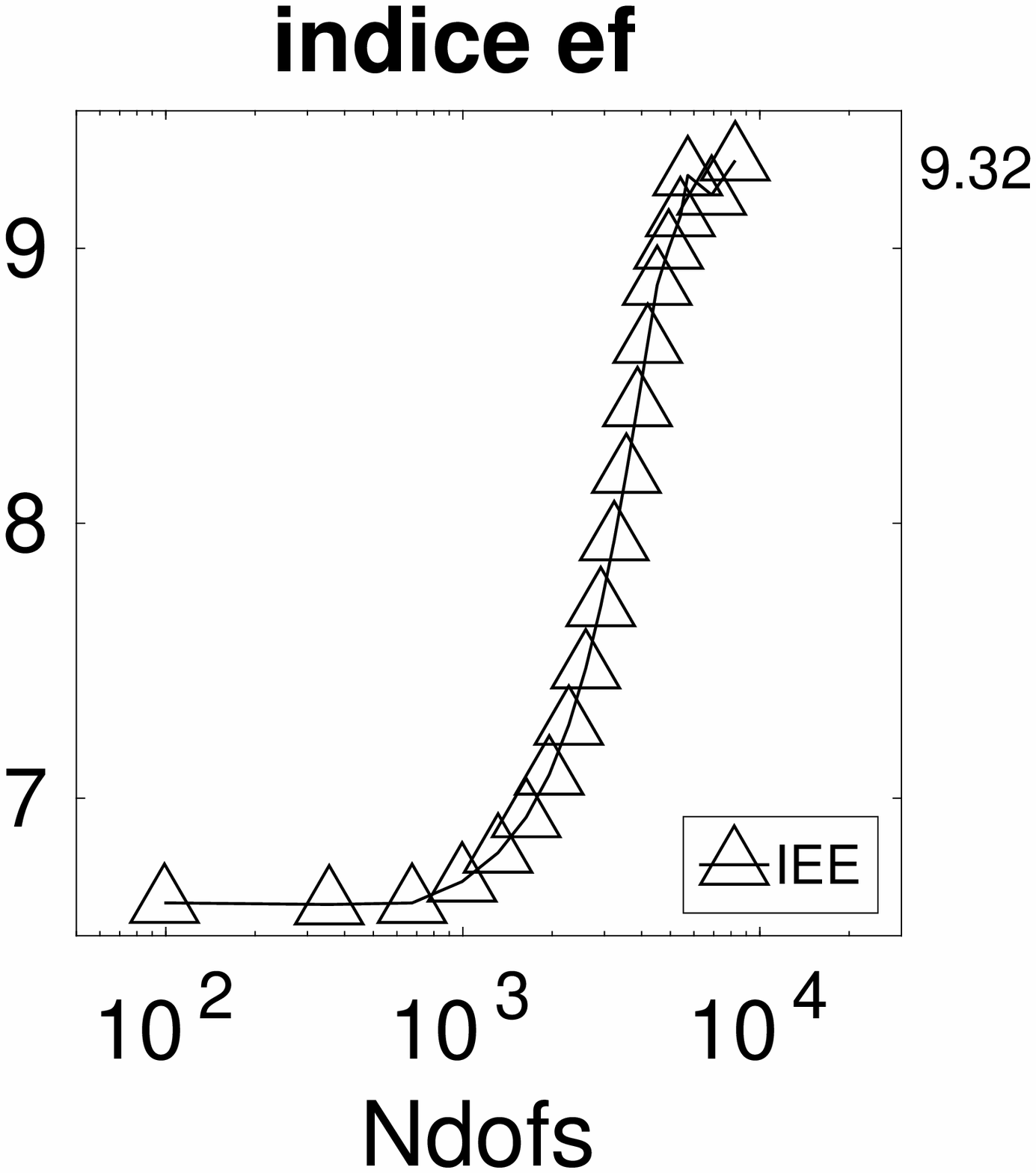} \\
\hspace{-0.8cm}\tiny{(C.2)}\\~\\
\psfrag{indice ef}{\hspace{-1.8cm}\huge{Effectivity index for $p = 1.8$}}
\includegraphics[trim={0 0 0 0},clip,width=3.8cm,height=3.5cm,scale=0.65]{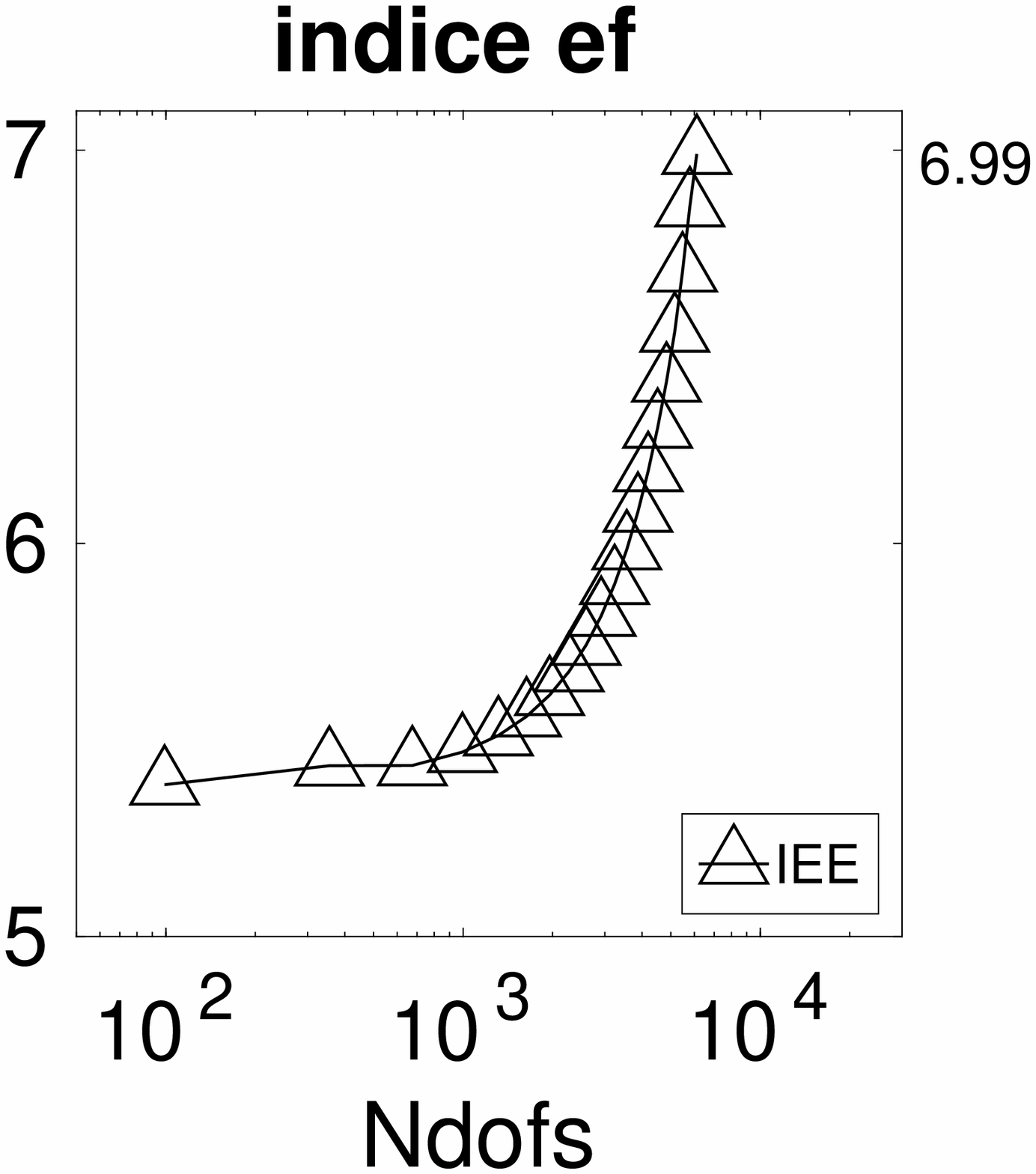} \\
\hspace{-0.8cm}\tiny{(D.2)}
\end{minipage}
\begin{minipage}[c]{0.25\textwidth}\centering
\vspace{-0.5em}
\includegraphics[trim={0 0 0 0},clip,width=2.8cm,height=2.5cm,scale=0.6]{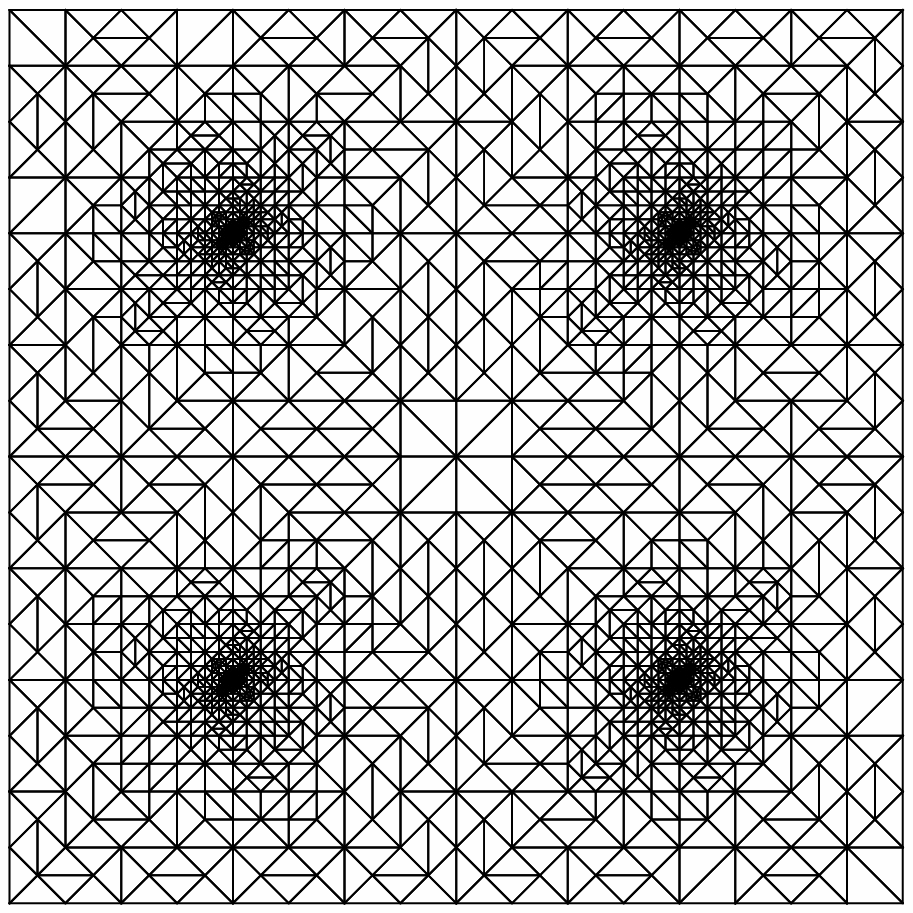}
\\~\\
\quad \tiny{(A.3)}\vspace{3.5em}
\includegraphics[trim={0 0 0 0},clip,width=2.8cm,height=2.5cm,scale=0.6]{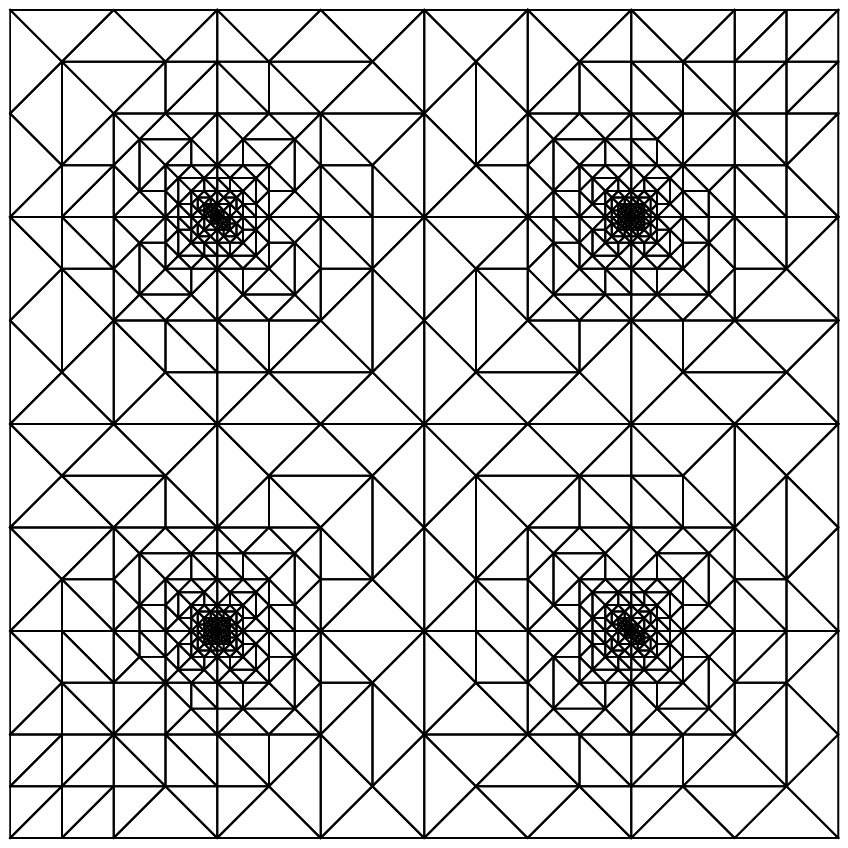} 
\\~\\
\quad \tiny{(B.3)}\vspace{3.5em}
\includegraphics[trim={0 0 0 0},clip,width=2.8cm,height=2.5cm,scale=0.6]{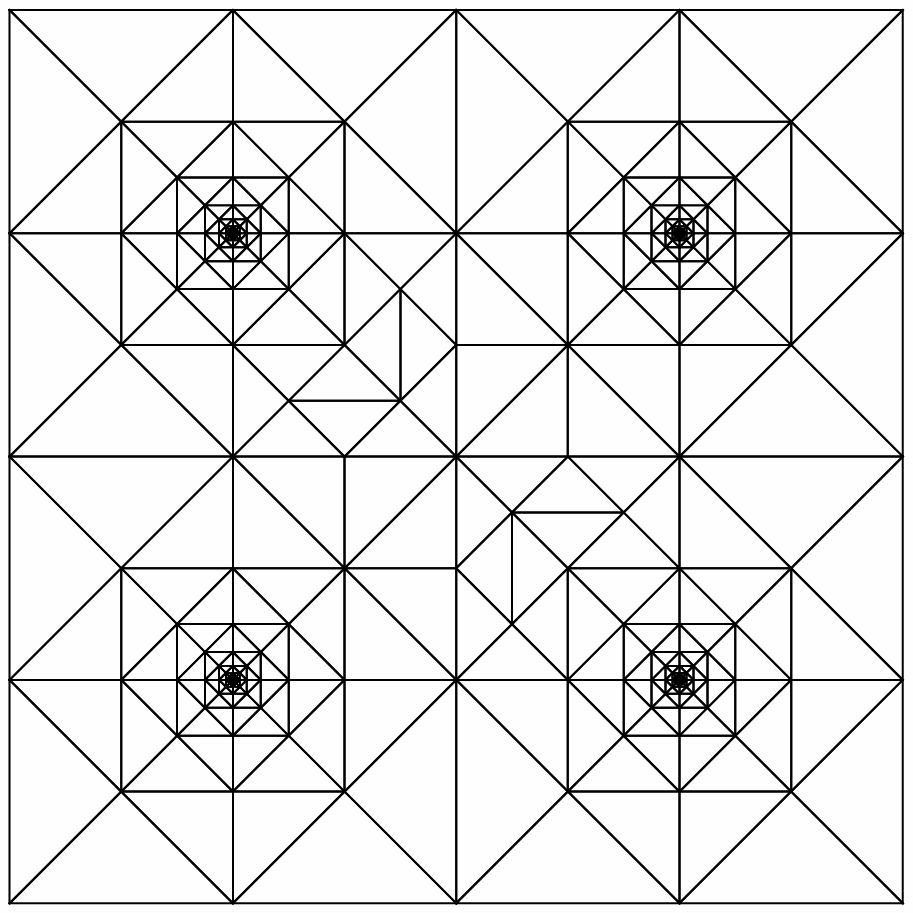}
\\~\\
\quad \tiny{(C.3)}\vspace{3.5em}
\includegraphics[trim={0 0 0 0},clip,width=2.8cm,height=2.5cm,scale=0.6]{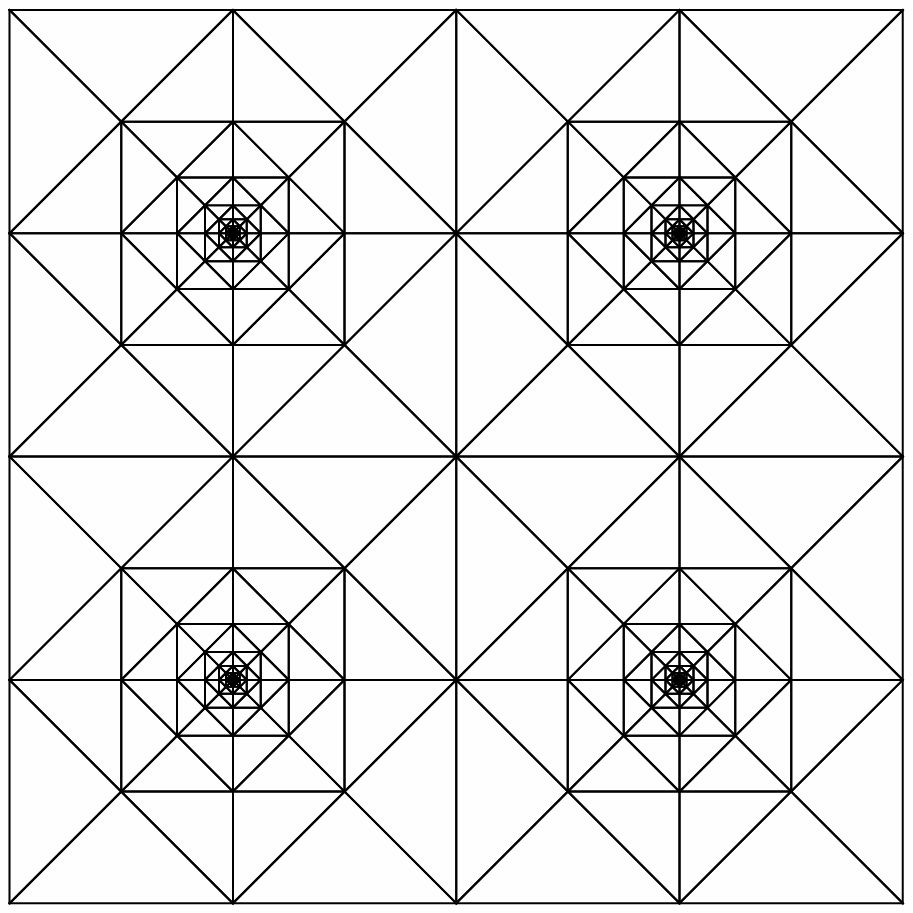}
\\~\\
\quad \tiny{(D.3)}
\end{minipage}
\caption{Example 1: Experimental rates of convergence for the error $\|(\mathbf{e}_{\boldsymbol{u}},e_\pi)\|_{\mathcal{X}}$ and error estimators $\zeta_p$ (A.1)--(D.1); effectivity indices $\mathscr{I}_p$ (A.2)--(D.2) and the 16th adaptively refined mesh (A.3)--(D.3) for $p \in \{1.2,1.4,1.6,1.8\}$.}
\label{fig:ex-1.1}
\end{figure}

\begin{figure}[!ht]
\centering
\psfrag{error}{{\huge $\|(\mathbf{e}_{\boldsymbol{u}},e_{\pi})\|_{\mathcal{X}}$}}
\psfrag{eta}{{\Huge $\zeta_p$}}
\psfrag{O(Ndofs-11)}{\huge$\mathsf{Ndof}^{-1}$}
\psfrag{O(Ndofs-45)}{\huge$\mathsf{Ndof}^{-0.45}$}
\psfrag{Ndofs}{{\huge $\mathsf{Ndof}$}}
\psfrag{IEE unifor}{{\huge $\mathscr{I}$ - unif.}}
\psfrag{IEE adaptiv}{{\huge $\mathscr{I}$ - adap.}}
\psfrag{IEE}{{\huge $\mathscr{I}$}}
\begin{minipage}[c]{0.35\textwidth}
\centering
\psfrag{Example 1 - p105}{\hspace{1.5cm}\huge{Uniform refinement}}
\includegraphics[trim={0 0 0 0},clip,width=3.7cm,height=3.5cm,scale=0.7]{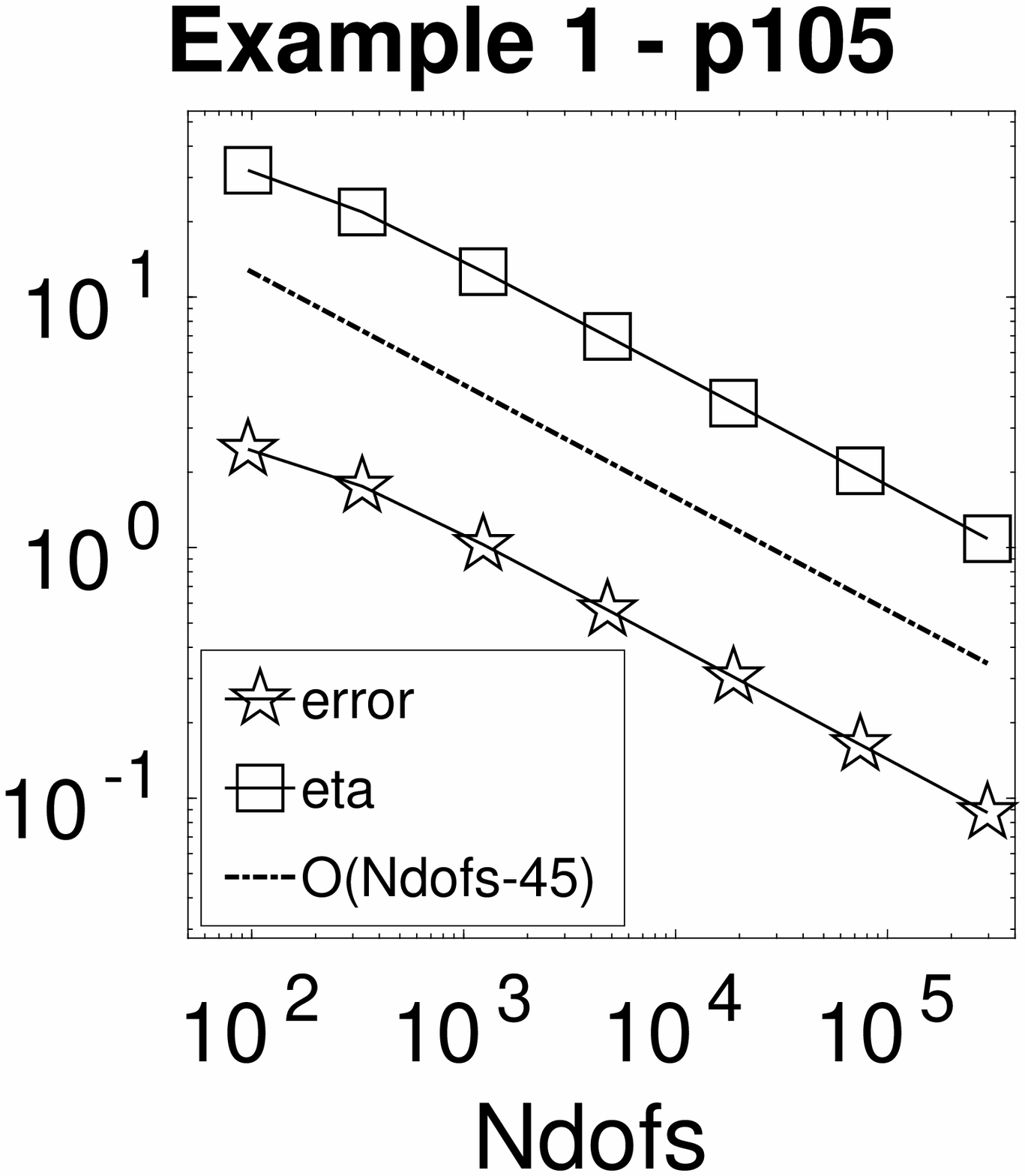} \\
\tiny{(A)}
\end{minipage}
\begin{minipage}[c]{0.35\textwidth}
\centering
\psfrag{Example 1 - p105}{\hspace{1.5cm}\huge{Adaptive refinement}}
\includegraphics[trim={0 0 0 0},clip,width=3.7cm,height=3.5cm,scale=0.7]{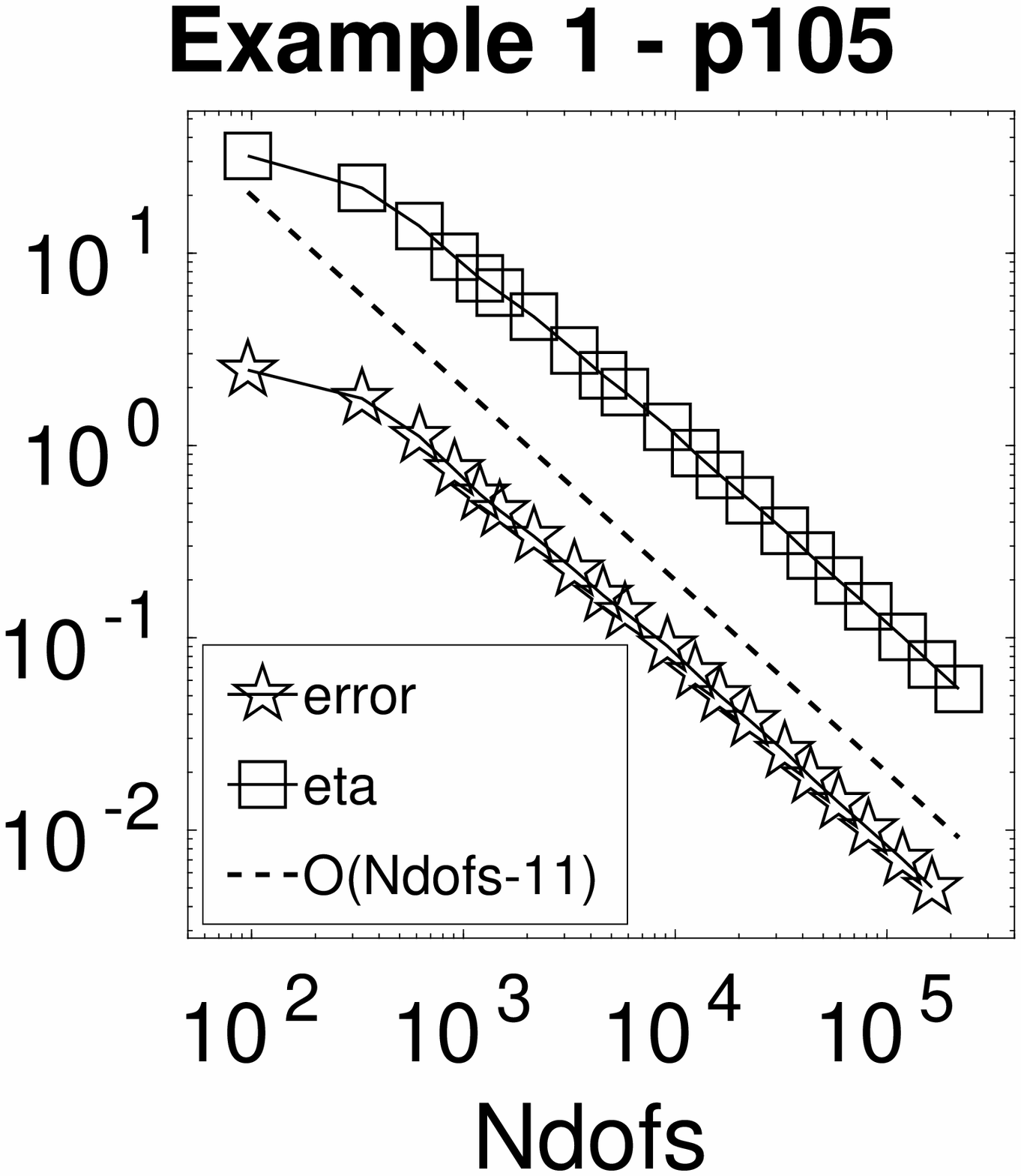} \\
\tiny{(B)}
\end{minipage}
\caption{Example 1: Experimental rates of convergence for the error $\|(\mathbf{e}_{\boldsymbol{u}},e_{\pi})\|_{\mathcal{X}}$ and the error estimator $\zeta_p$ for uniform refinement (A) and adaptive refinement (B).}
\label{fig:ex-1.2}
\end{figure}


\begin{figure}[!ht]
\centering
\psfrag{estimador}{\huge{$\zeta$}}
\psfrag{O(Ndofs-1)}{\huge$\mathsf{Ndof}^{-1}$}
\psfrag{O(Ndofs-11)}{\huge$\mathsf{Ndof}^{-1}$}
\psfrag{Ndofs}{\Huge{$\mathsf{Ndof}$}}
\psfrag{Example 1 - p}{\hspace{1.8cm}\huge{Estimators $\zeta_p$}}
\psfrag{p=1.05}{\huge{$p = 1.05$}}
\psfrag{p=1.2}{\huge{$p = 1.2$}}
\psfrag{p=1.4}{\huge{$p = 1.4$}}
\psfrag{p=1.6}{\huge{$p = 1.6$}}
\psfrag{p=1.8}{\huge{$p = 1.8$}}
\hspace{-1.0cm}
\begin{minipage}[b]{0.38\textwidth}\centering
\includegraphics[trim={0 0 0 0},clip,width=3.5cm,height=3.5cm,scale=0.65]{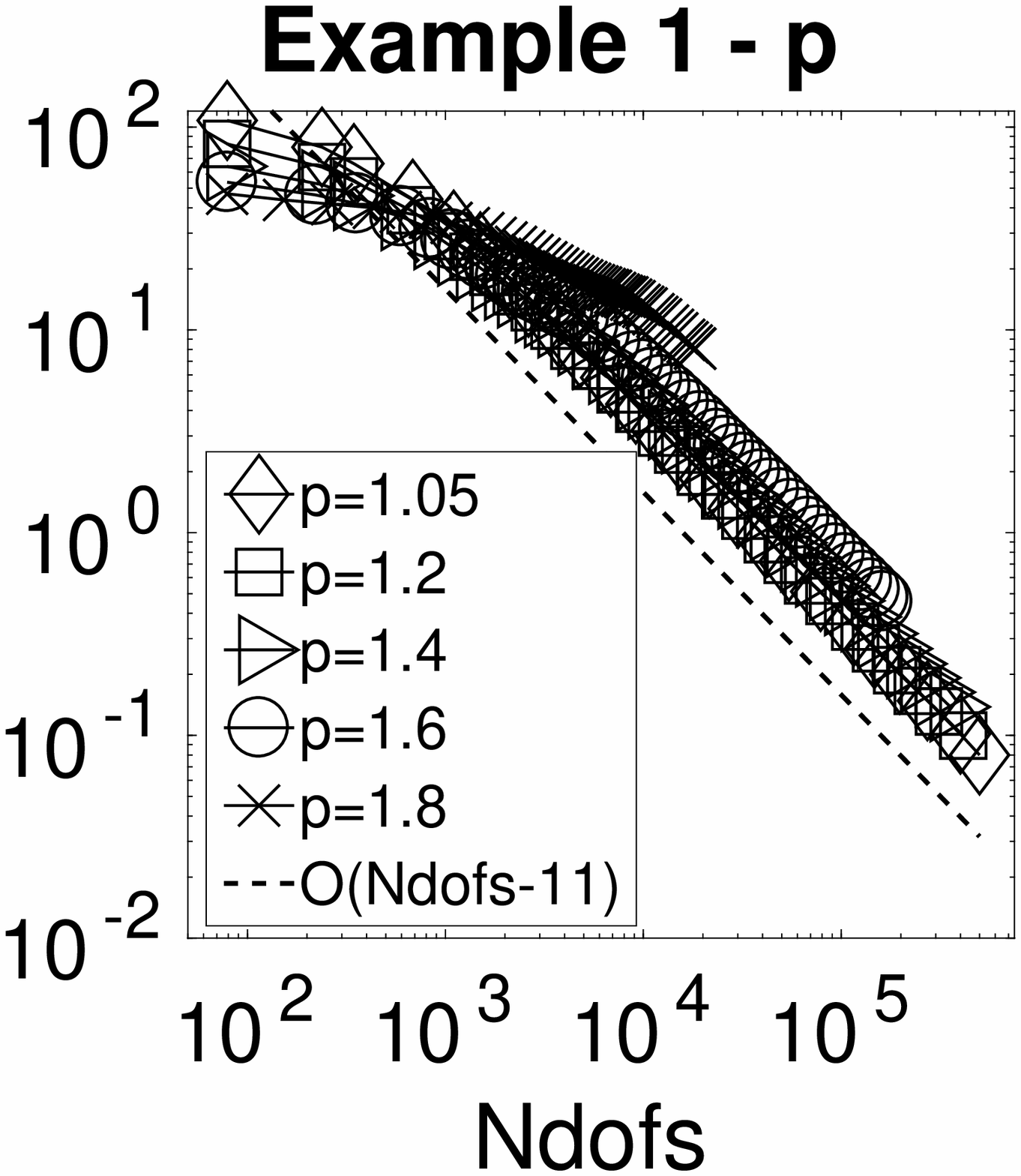} \\
\tiny{(A)}
\hspace{-1cm}
\end{minipage}
\begin{minipage}[b]{0.3\textwidth}\centering
\scriptsize{$|\boldsymbol{u}_{\mathscr{T}}|$}\\
\includegraphics[trim={0 0 0 0},clip,width=3.3cm,height=3.0cm,scale=0.5]{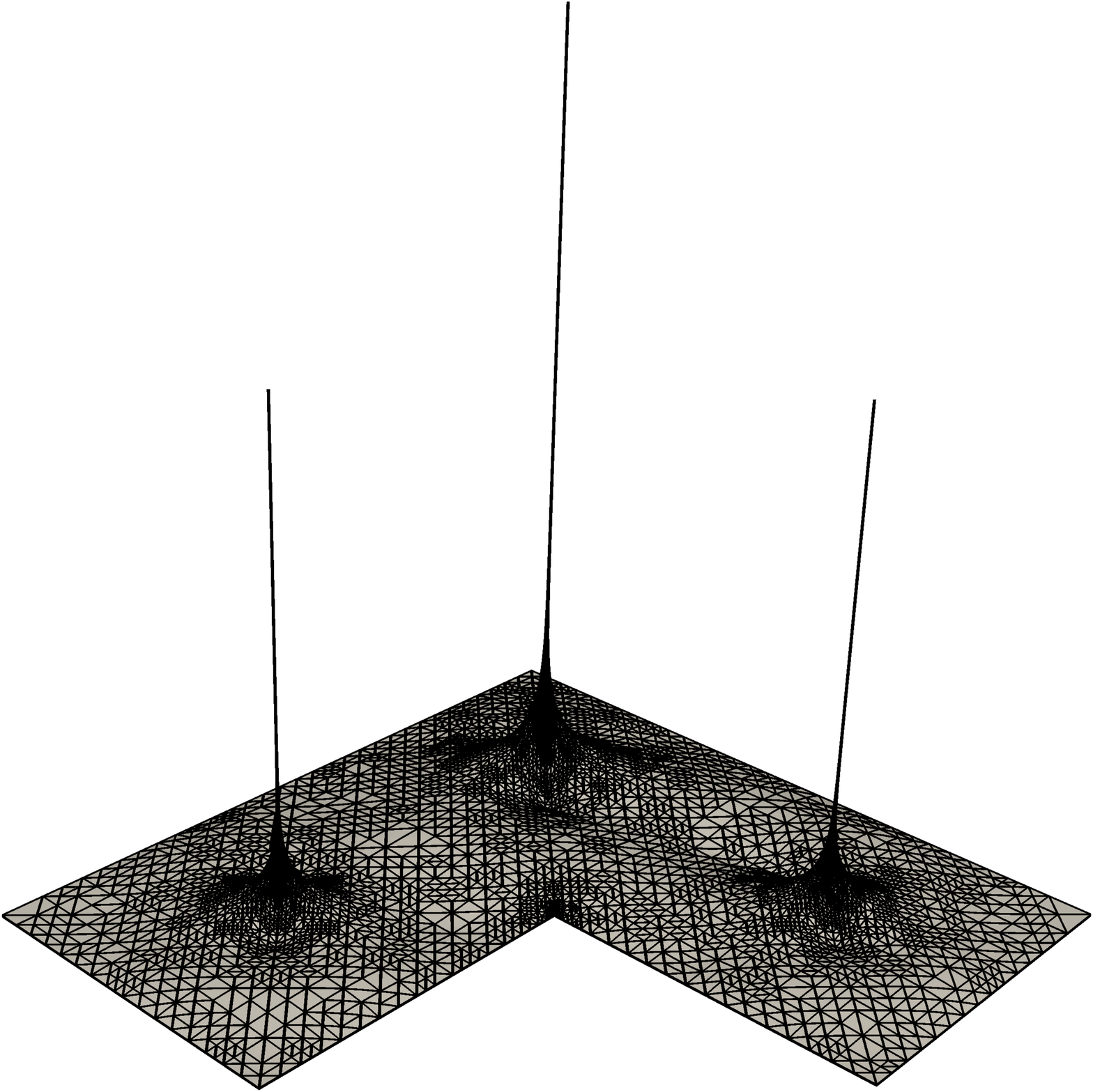} \\
\qquad \tiny{(B)}
\end{minipage}
\begin{minipage}[b]{0.3\textwidth}\centering
\scriptsize{$\pi_{\mathscr{T}}$}\\
\includegraphics[trim={0 0 0 0},clip,width=3.3cm,height=3.0cm,scale=0.5]{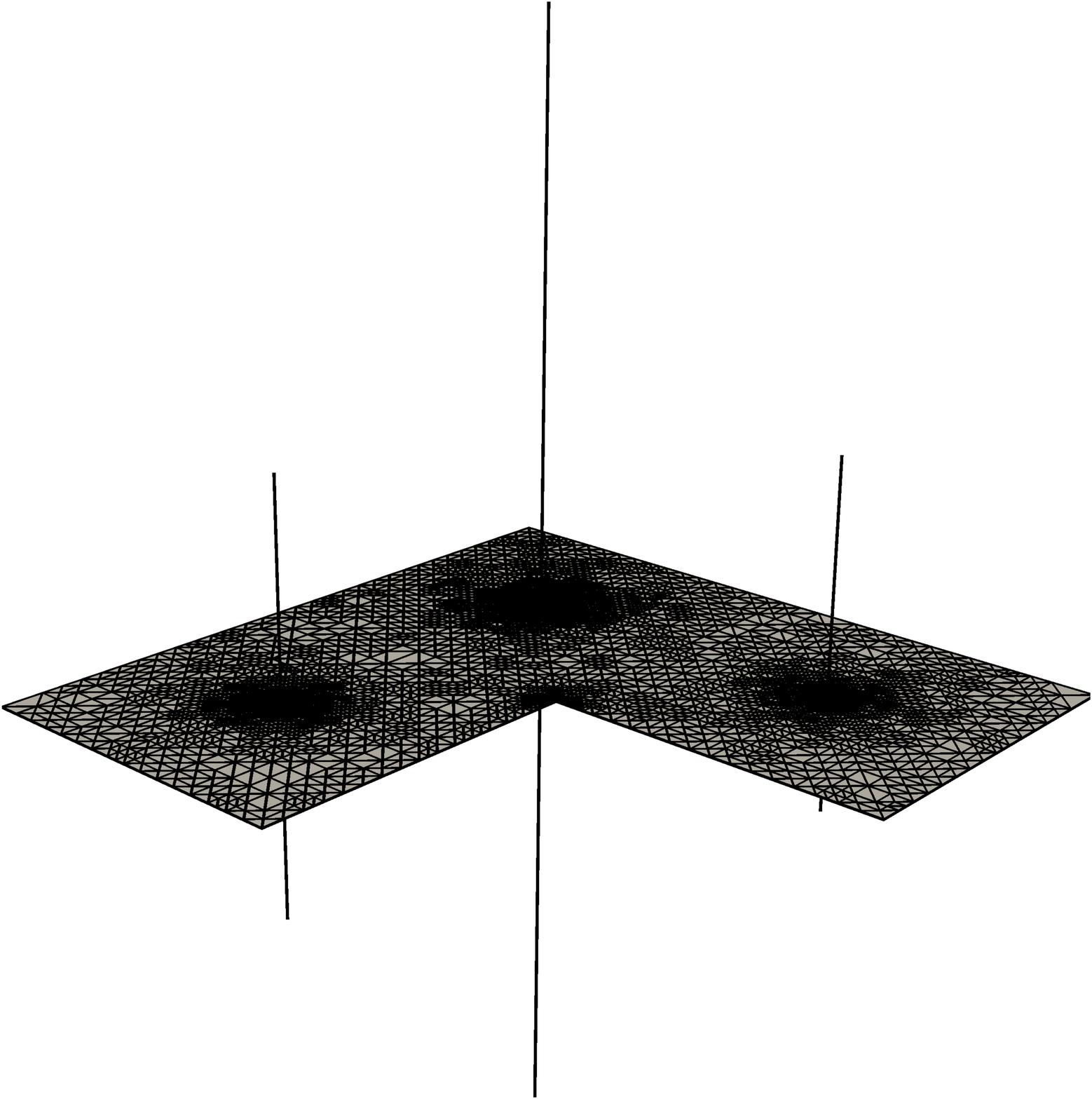} \\
\qquad \tiny{(C)}
\end{minipage}
\\~\\
\begin{minipage}[b]{0.3\textwidth}\centering
\includegraphics[trim={0 0 0 0},clip,width=3.2cm,height=3.5cm,scale=0.5]{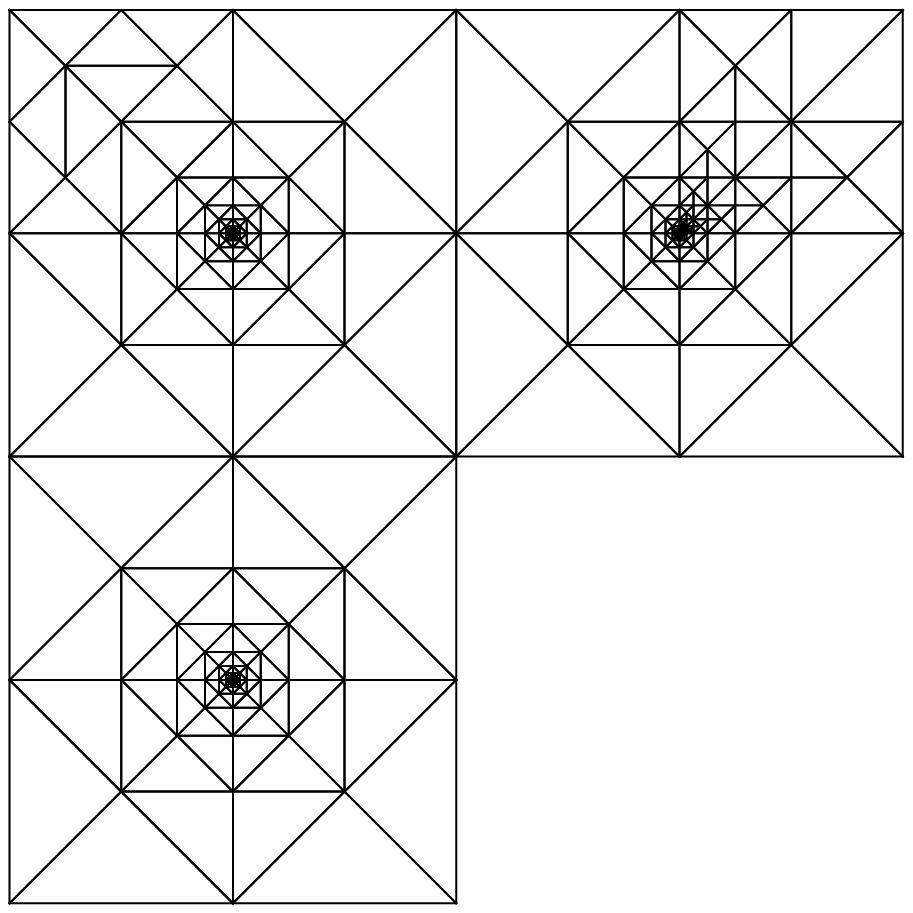}\\
\qquad \tiny{(D)}
\end{minipage}
\begin{minipage}[b]{0.3\textwidth}\centering
\includegraphics[trim={0 0 0 0},clip,width=3.2cm,height=3.5cm,scale=0.5]{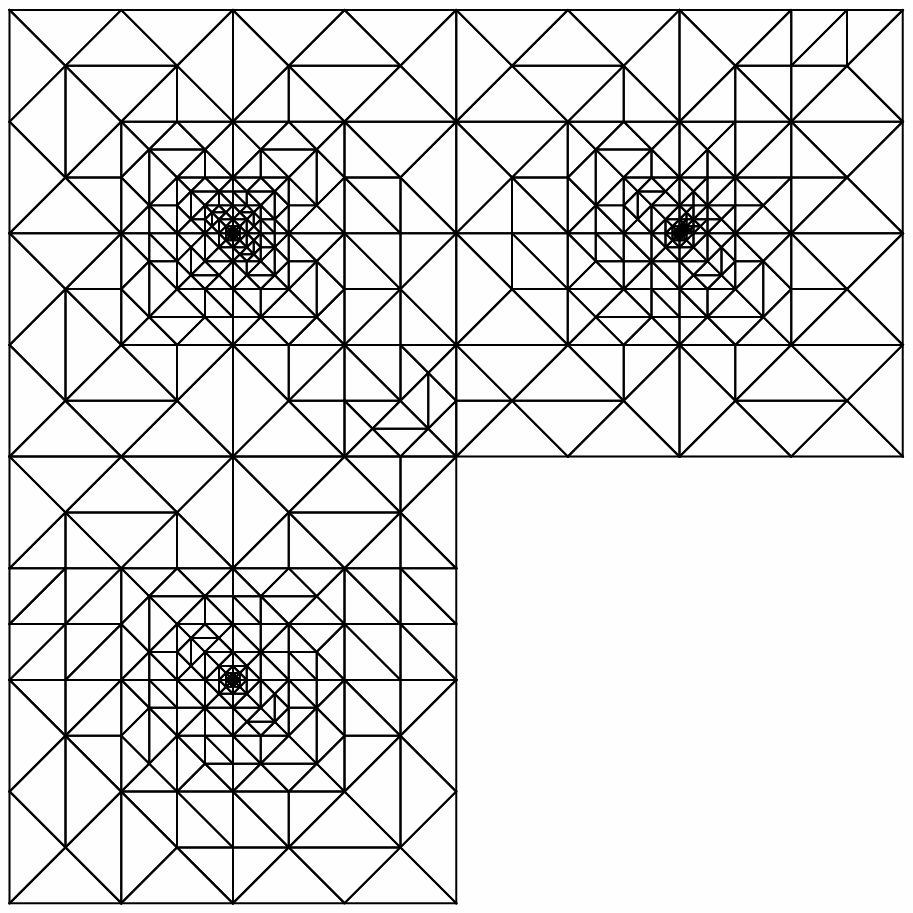}\\
\qquad \tiny{(E)}
\end{minipage}
\begin{minipage}[b]{0.3\textwidth}\centering
\includegraphics[trim={0 0 0 0},clip,width=3.2cm,height=3.5cm,scale=0.5]{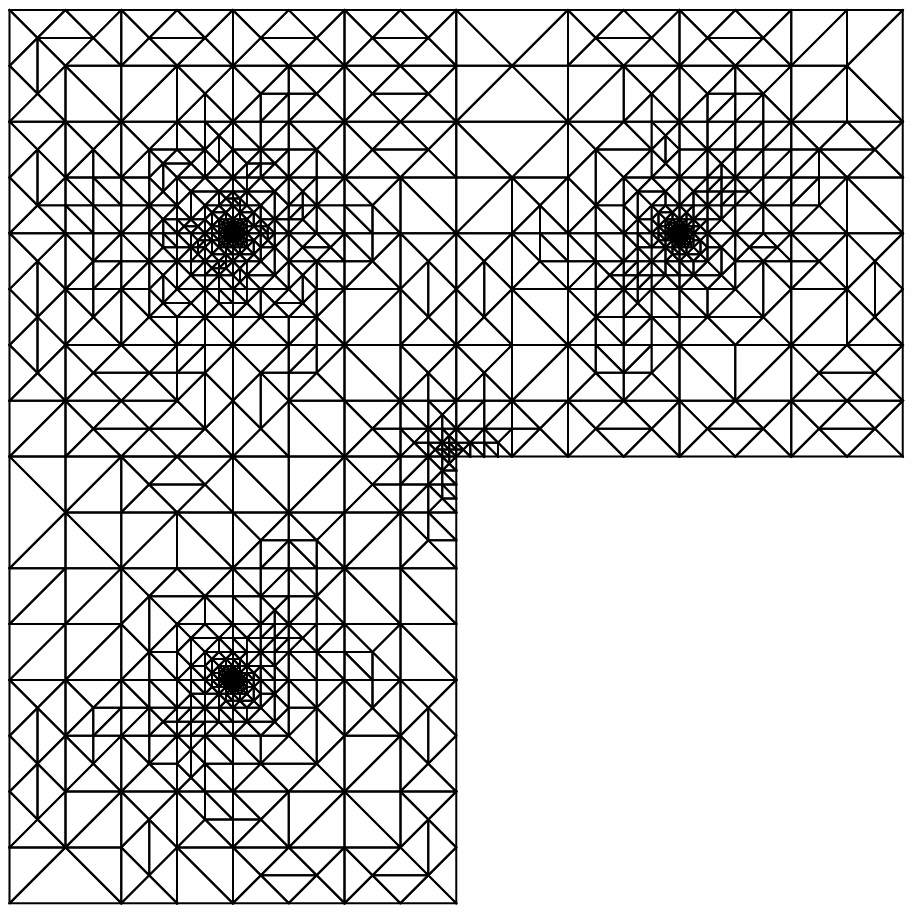}\\
\qquad \tiny{(F)}
\end{minipage}
\caption{Example 2: Experimental rates of convergence for the error estimators $\zeta_p$, for $p \in \{1.05,1.2,1.4,1.6,1.8\}$ (A), the finite element approximation of $|\boldsymbol{u}_{\mathscr{T}}|$ (B) and $\pi_{\mathscr{T}}$ (C) obtained on the 30th adaptively refined mesh and the meshes obtained after 10 (D), 20 (E) and 30 (F) iterations of the adaptive loop ($p = 1.4$).}
\label{fig:ex_2}
\end{figure}

\begin{remark}[Influence of $p$ on $\mathsf{Ndof}$]
In Figure \ref{fig:ex-1.1}, we present experimental rates of convergence for the estimators $\zeta_p$ and the total error $\|(\mathbf{e}_{\boldsymbol{u}},e_\pi)\|_{\mathcal{X}}$ for different values of the integrability index $p$. Notice that, each plot involves a different range of numbers of degrees of freedom. The reported numerical results suggest that this may be due to the fact that the adaptive refinement depends on the value of $p$: as $p$ increases, the refinement is mostly performed on the elements that are close to the Dirac measure points; see Figure \ref{fig:ex-1.1} (A.3)--(D.3). When $p$ gets close to $2$, after a certain number of adaptive iterations, there are elements $T\in \mathscr{T}$, around the singular points, such that $h_T\approx 10^{-16}$. As a consequence, the assembly calculations reach machine precision numbers and thus make impossible more computations within the adaptive procedure. The same behavior is observed in Figure \ref{fig:ex_2} (A), where for each value of $p$ the estimator $\zeta_p$ involves a different range of numbers of degrees of freedom.
\end{remark}
We now present a three dimensional example with inhomogeneous Dirichlet boundary conditions. 
~\\
\textbf{Example 3 (Convex domain).} In this case we consider $\Omega=(0,1)^3$ and
\[\mathcal{D} = \{ (0.25,0.25,0.25), (0.25,0.25,0.75), (0.75,0.75,0.25), (0.75,0.75,0.75) \}. \]

The solution for this example corresponds to the one described in \eqref{def:fund_sol}.

In Figure~\ref{fig:ex-3.1} we report the results obtained for Example 3. We present, for different values of the integrability index $p\in \{1.1, 1.2, 1.3, 1.4\}$, experimental rates of convergence for $\zeta_p$ and $\|(\mathbf{e}_{\boldsymbol{u}},e_{\pi})\|_{\mathcal{X}}$ and adaptively refined meshes. We observe in subfigures (A.1)--(D.1), optimal experimental rates of convergence for the error estimators $\zeta_p$ and the total error $\|(\mathbf{e}_{\boldsymbol{u}},e_{\pi})\|_{\mathcal{X}}$. In subfigures (A.2)--(D.2), we observe the effect of varying the integrability index $p$ on the adaptively refined meshes. In particular, we appreciate that the adaptive refinement is mostly concentrated on the points $t\in\mathcal{D}$ where the Dirac measures are supported.


\begin{figure}[!ht]
\centering
\psfrag{error total}{{\huge $\|(\mathbf{e}_{\boldsymbol{u}},e_{\pi})\|_{\mathcal{X}}$}}
\psfrag{estimador}{{\Huge $\zeta_p$}}
\psfrag{O(Ndofs-1)}{\huge$\mathsf{Ndof}^{-1}$}
\psfrag{O(Ndofs-233)}{\huge$\mathsf{Ndof}^{-2/3}$}
\psfrag{Ndofs}{{\huge $\mathsf{Ndof}$}}
\psfrag{IEE}{{\huge $\mathscr{I}$}}
\begin{minipage}[c]{0.47\textwidth}
\centering
\psfrag{Example 3 - p11}{\hspace{-0.3cm}\huge{$\|(\mathbf{e}_{\boldsymbol{u}},e_{\pi})\|_{\mathcal{X}}$ and $\zeta_p$ for $p = 1.1$}}
\includegraphics[trim={0 0 0 0},clip,width=2.8cm,height=3.0cm,scale=0.68]{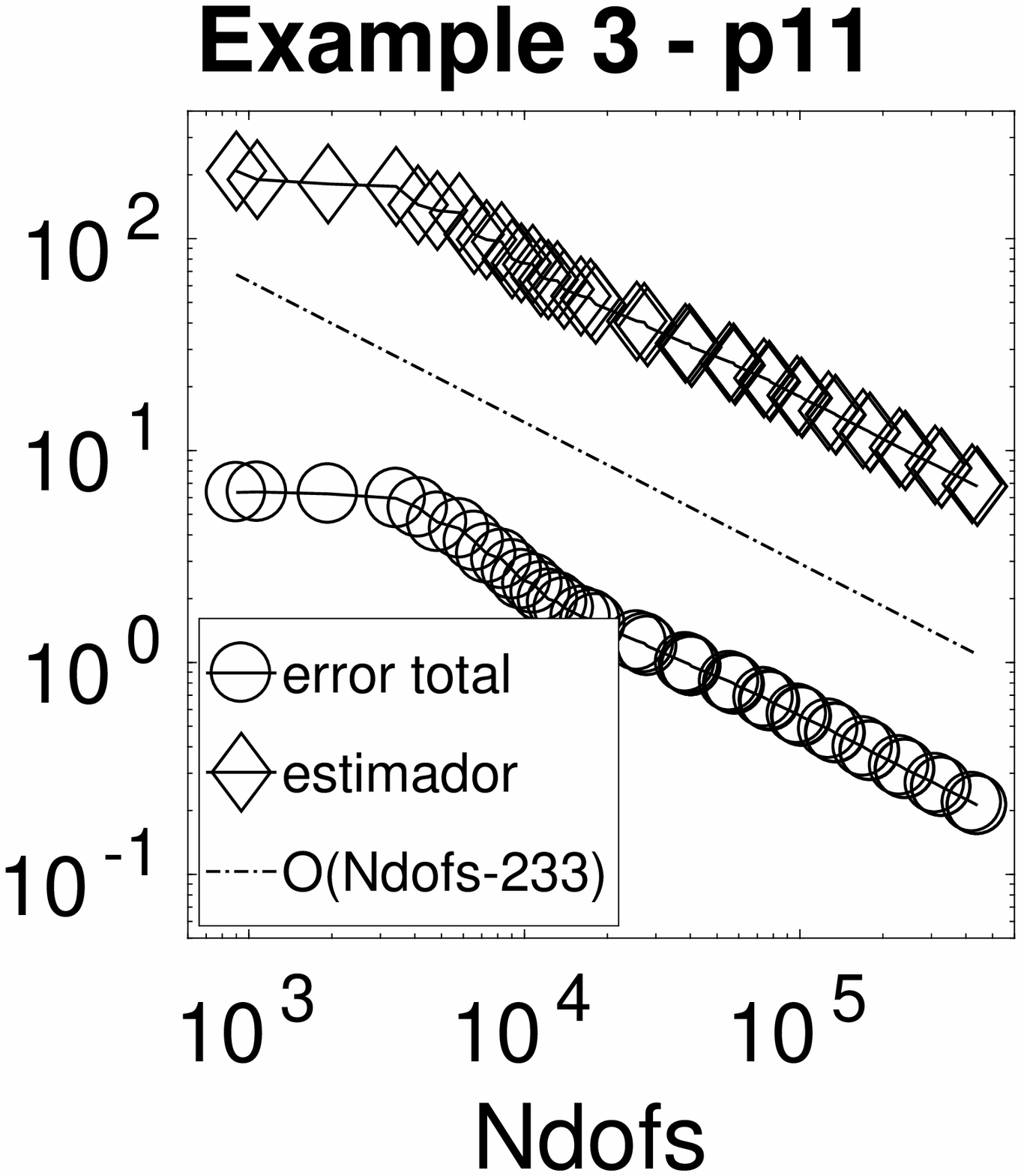} 
\includegraphics[trim={0 0 0 0},clip,width=2.6cm,height=2.5cm,scale=0.6]{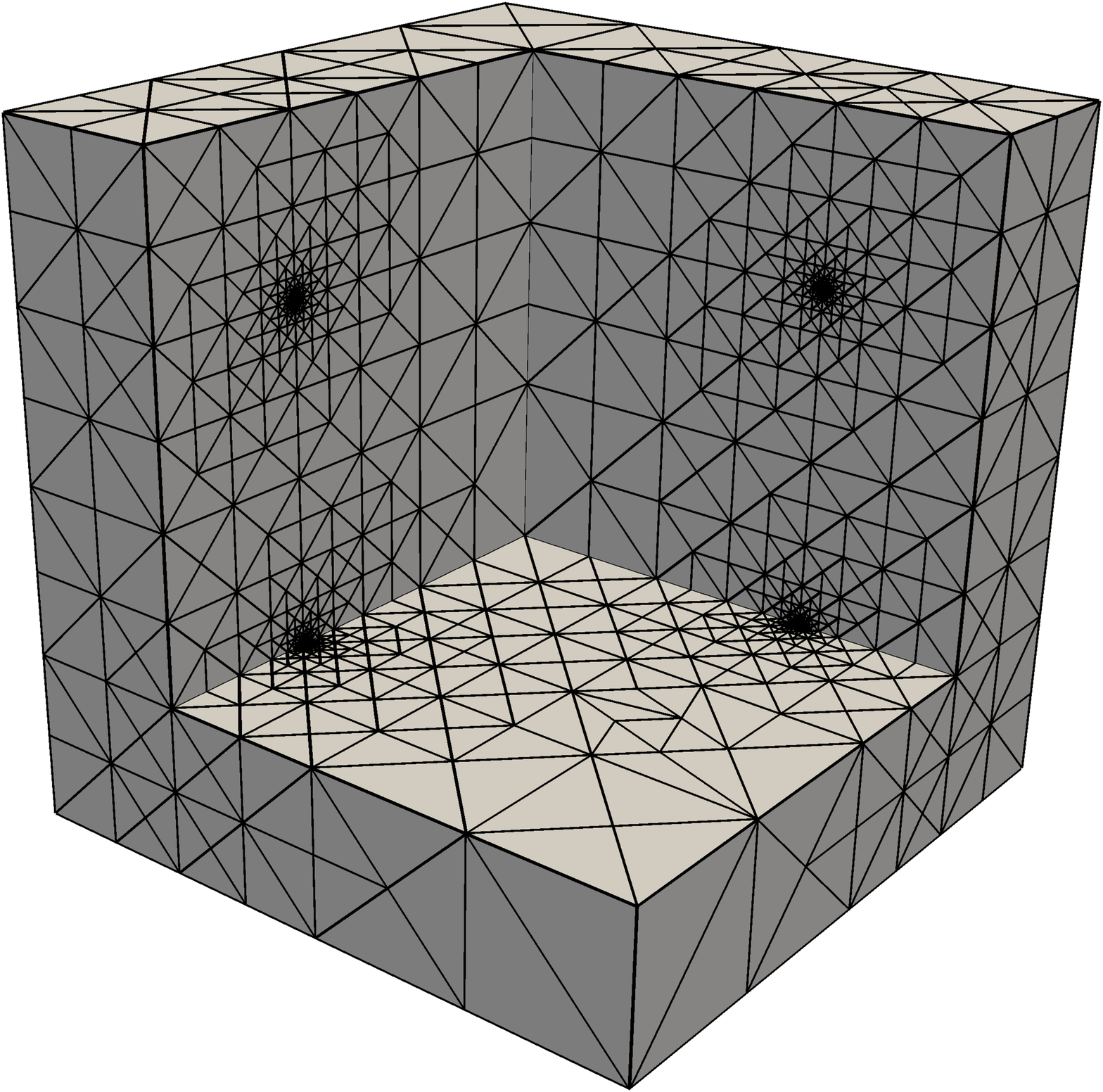}
\\
\tiny{(A.1)} \hspace{2cm} \tiny{(A.2)}\vspace{3.5em}
\end{minipage}
\begin{minipage}[c]{0.47\textwidth}
\centering
\psfrag{Example 3 - p12}{\hspace{-0.3cm}\huge{$\|(\mathbf{e}_{\boldsymbol{u}},e_{\pi})\|_{\mathcal{X}}$ and $\zeta_p$ for $p = 1.2$}}
\includegraphics[trim={0 0 0 0},clip,width=2.8cm,height=3.0cm,scale=0.68]{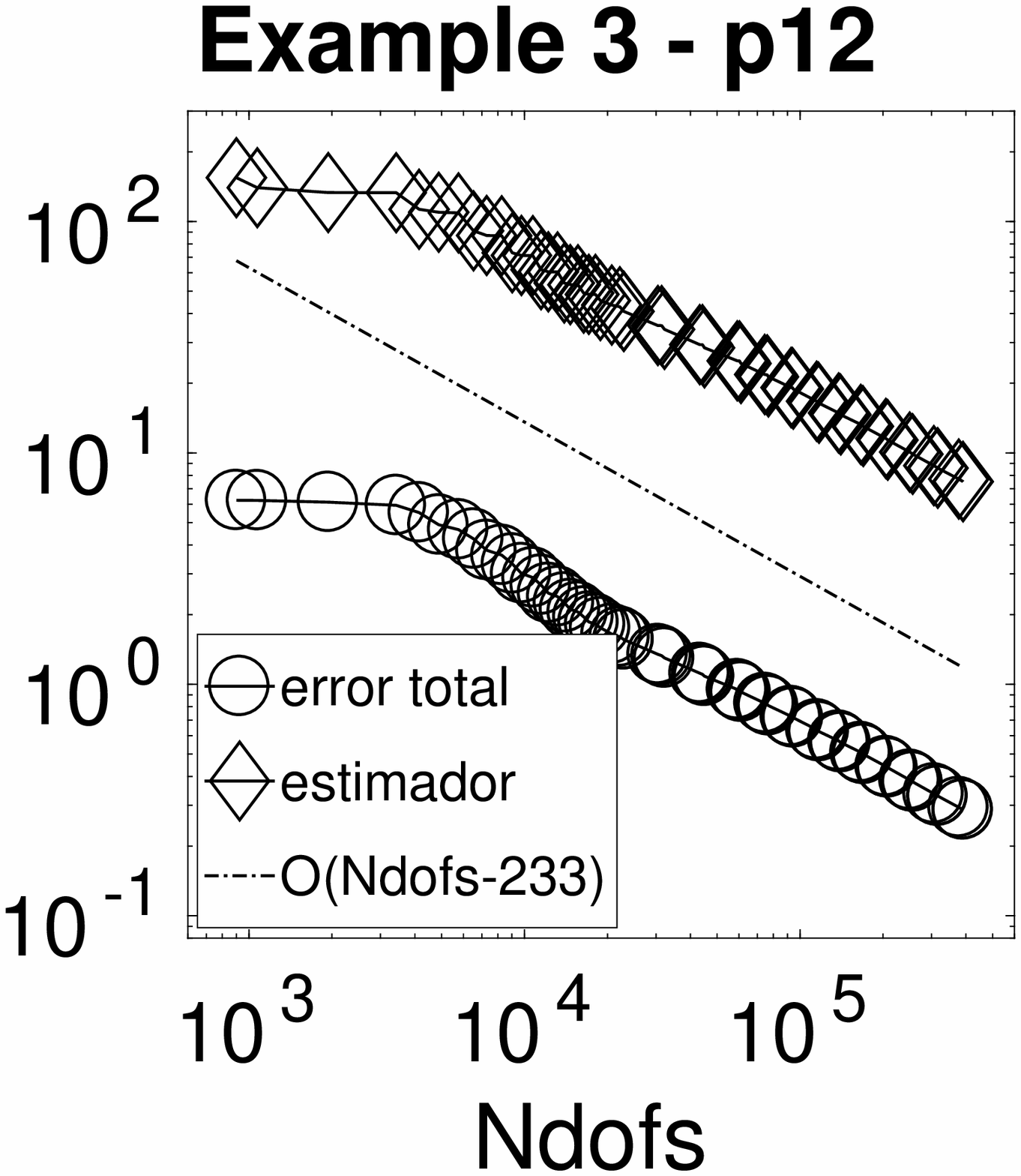} 
\includegraphics[trim={0 0 0 0},clip,width=2.6cm,height=2.5cm,scale=0.6]{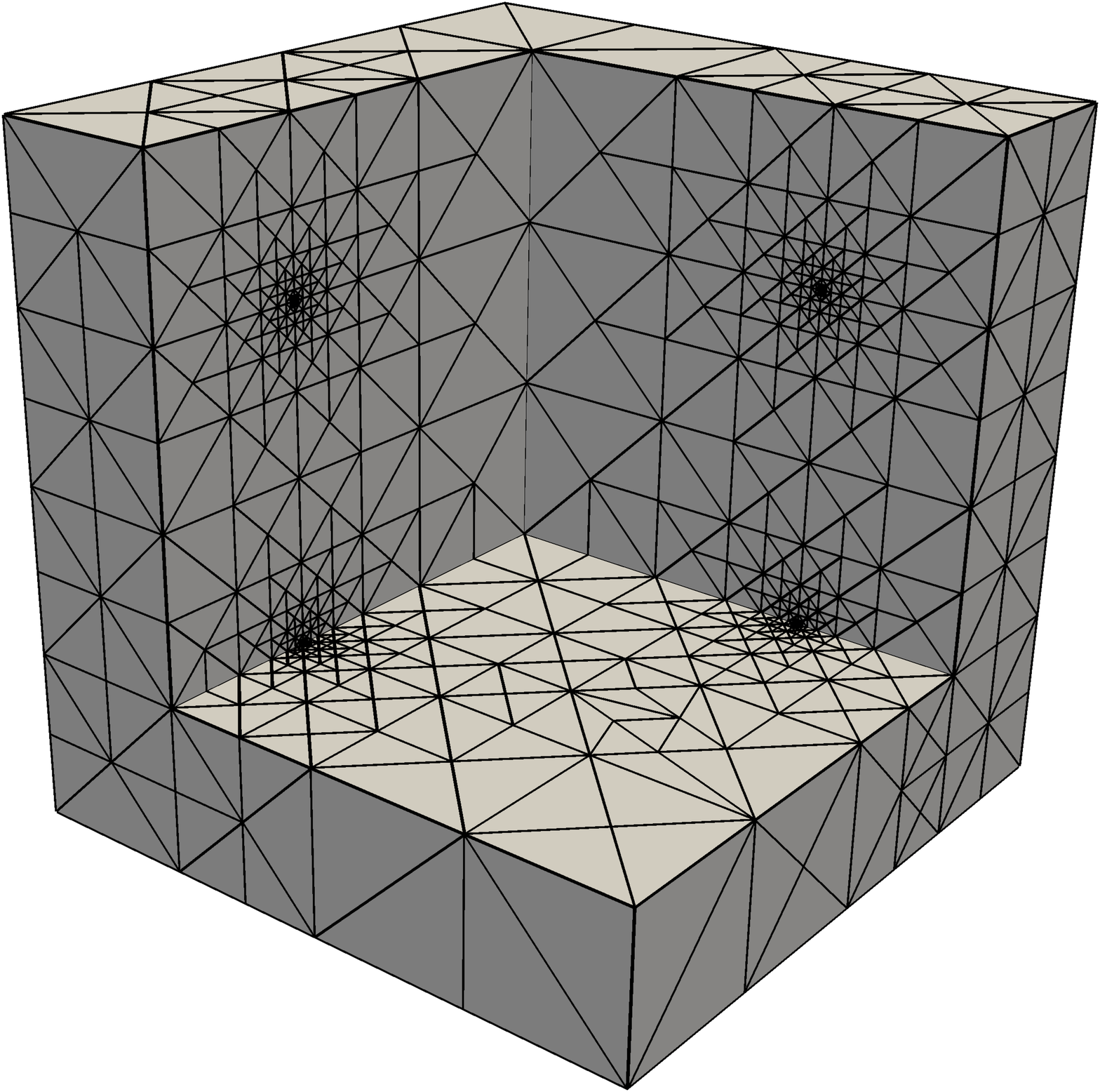}
\\
\tiny{(B.1)}  \hspace{2cm}\tiny{(B.2)}\vspace{3.5em}
\end{minipage}
\begin{minipage}[c]{0.47\textwidth}
\centering
\psfrag{Example 3 - p13}{\hspace{-0.3cm}\huge{$\|(\mathbf{e}_{\boldsymbol{u}},e_{\pi})\|_{\mathcal{X}}$ and $\zeta_p$ for $p = 1.3$}}
\includegraphics[trim={0 0 0 0},clip,width=2.8cm,height=3.0cm,scale=0.68]{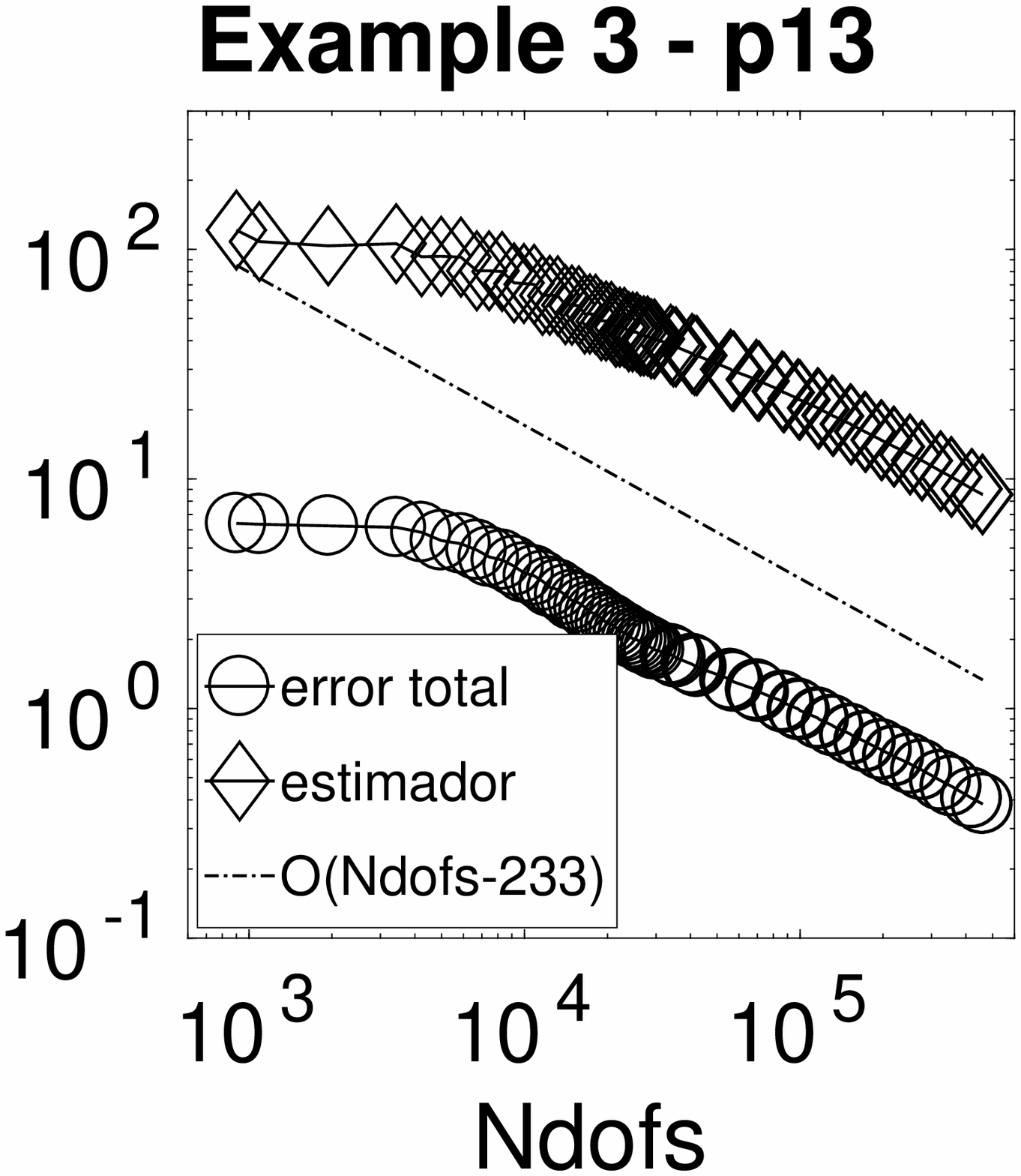} 
\includegraphics[trim={0 0 0 0},clip,width=2.6cm,height=2.5cm,scale=0.6]{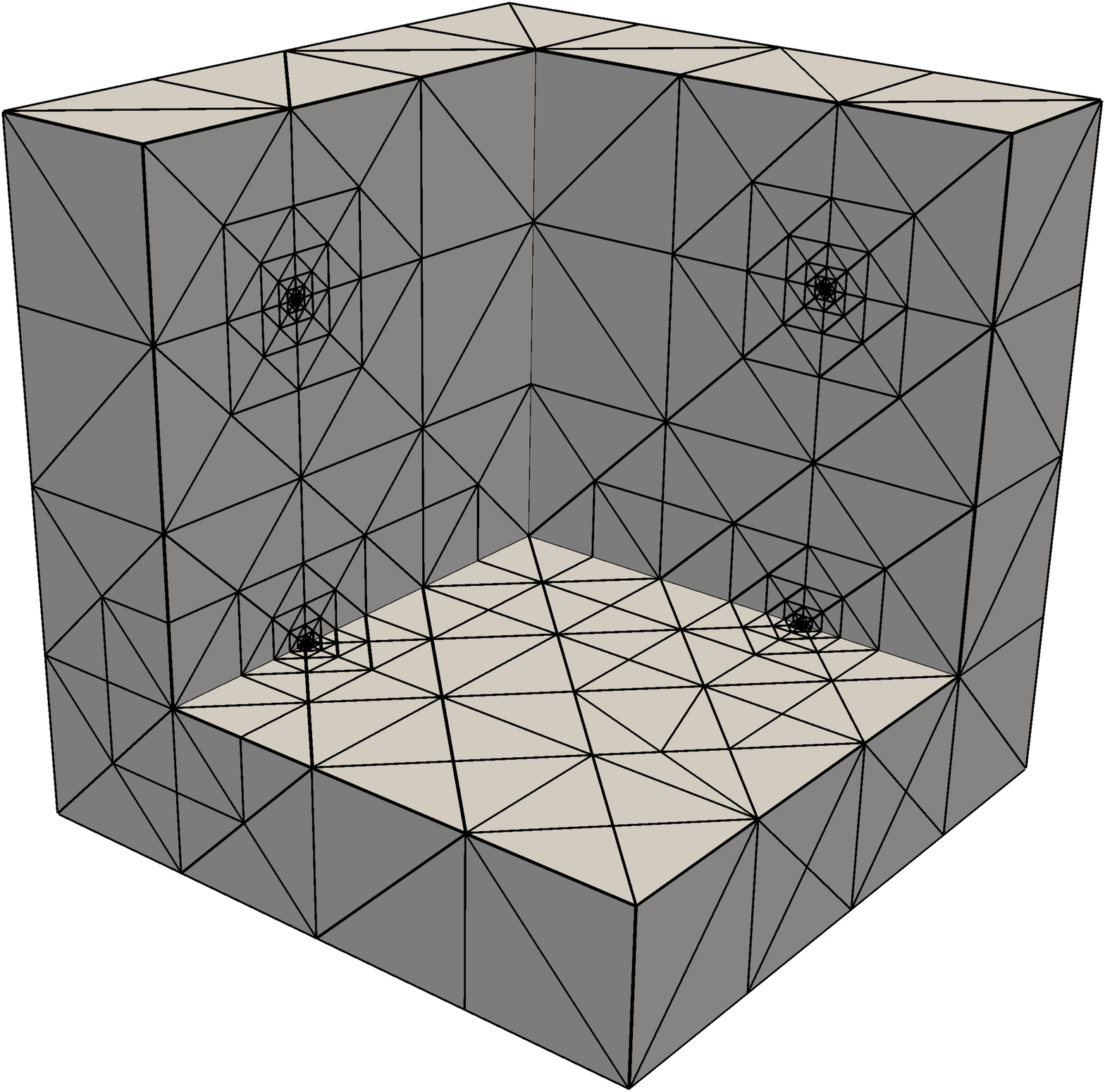}
\\
\tiny{(C.1)}  \hspace{2cm} \tiny{(C.2)}\vspace{3.5em}
\end{minipage}
\begin{minipage}[c]{0.47\textwidth}
\centering
\psfrag{Example 3 - p14}{\hspace{-0.3cm}\huge{$\|(\mathbf{e}_{\boldsymbol{u}},e_{\pi})\|_{\mathcal{X}}$ and $\zeta_p$ for $p = 1.4$}}
\includegraphics[trim={0 0 0 0},clip,width=2.8cm,height=3.0cm,scale=0.68]{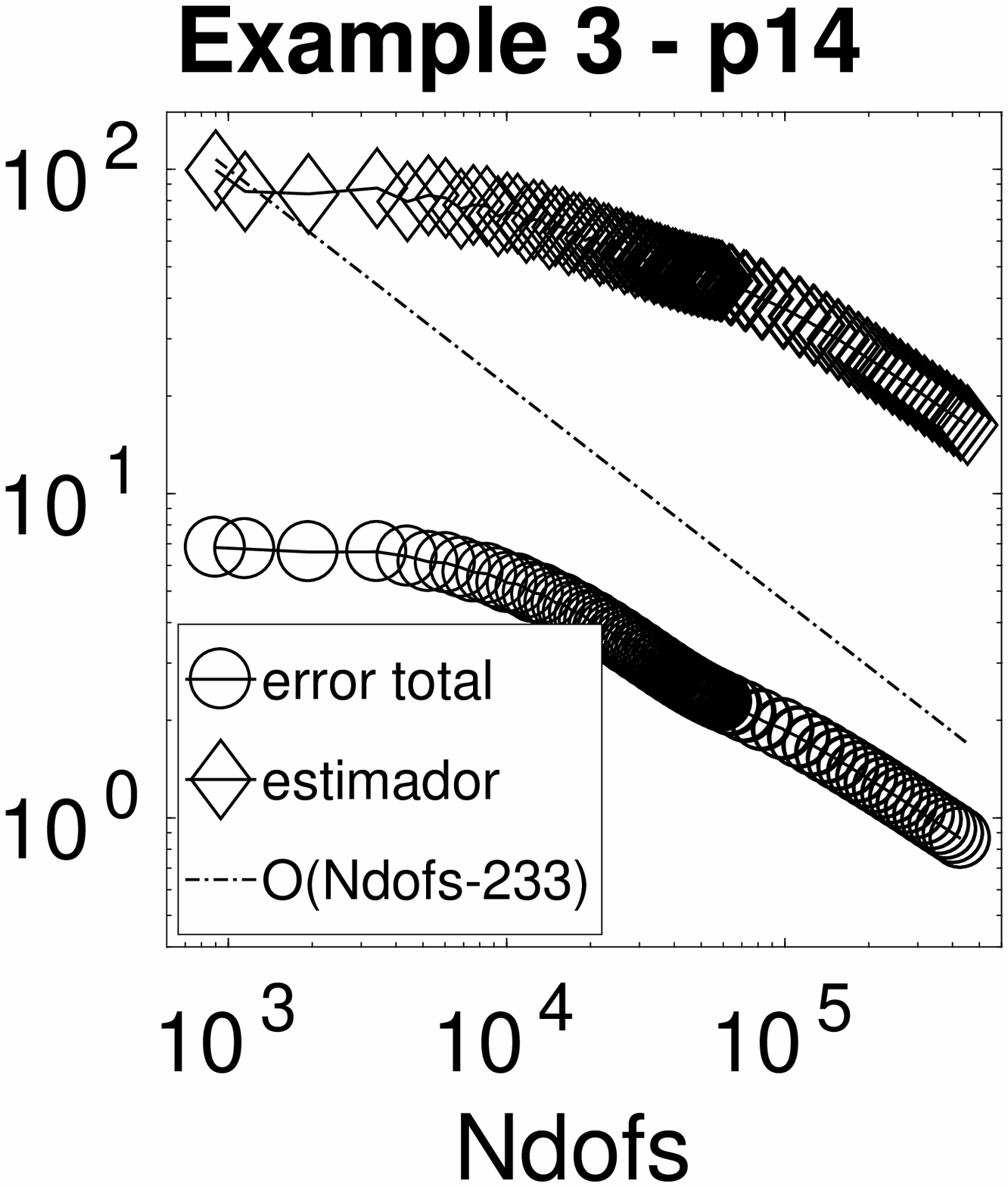} 
\includegraphics[trim={0 0 0 0},clip,width=2.6cm,height=2.5cm,scale=0.6]{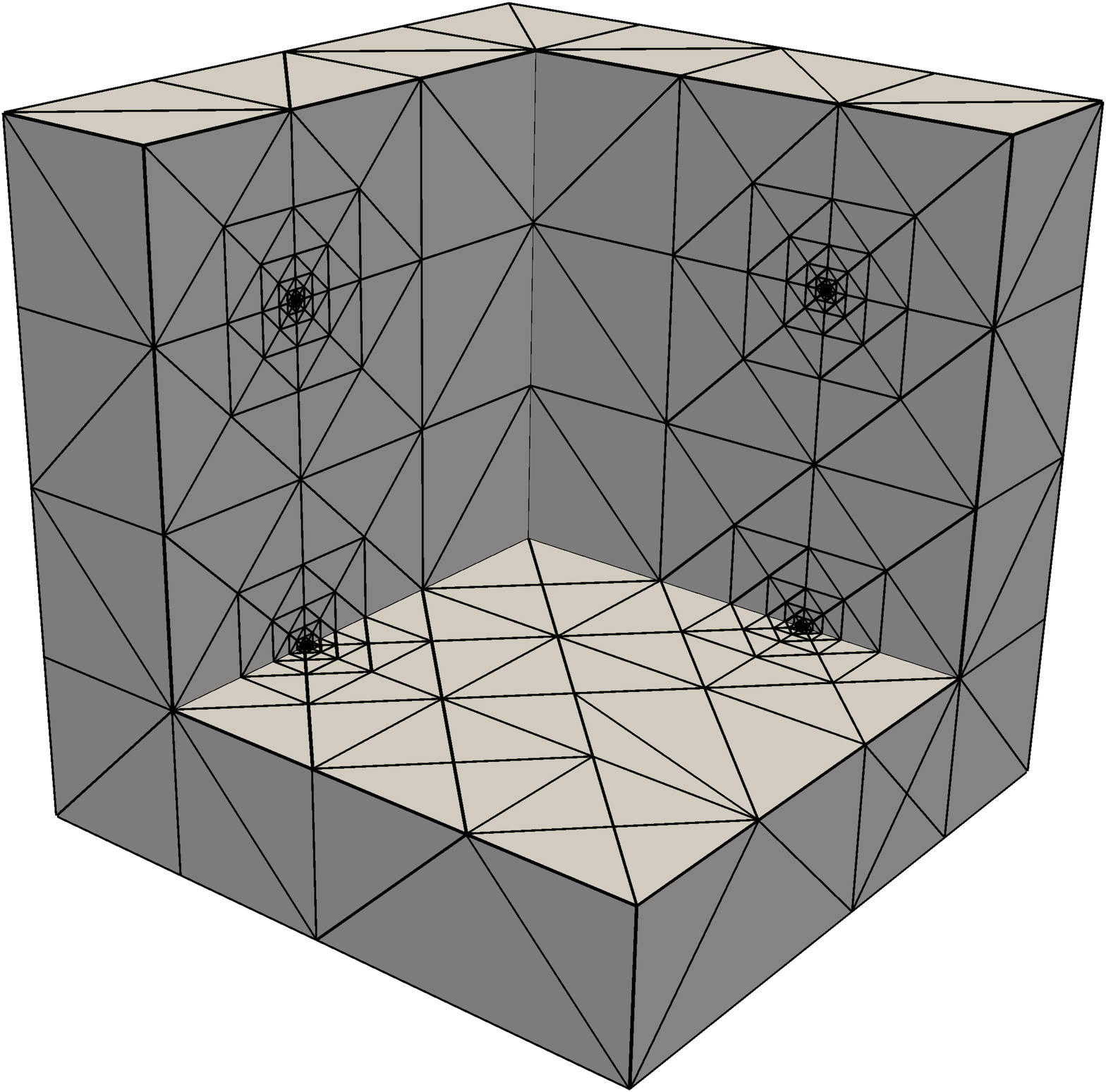}
\\
\tiny{(D.1)}  \hspace{2cm}\tiny{(D.2)}\vspace{3.5em}
\end{minipage}
\caption{Example 3: Experimental rates of convergence for the error $\|(\mathbf{e}_{\boldsymbol{u}},e_\pi)\|_{\mathcal{X}}$ and error estimator $\zeta_p$ (A.1)--(D.1) and the 37th adaptively refined mesh (A.2)--(D.2) for $p \in \{1.1,1.2,1.3,1.4\}$.}
\label{fig:ex-3.1}
\end{figure}

\subsection{A stabilized scheme}\label{sec:numer_stab}
We perform numerical experiments for \emph{low--order stabilized approximation} with the discrete spaces \eqref{eq:V_estab} and \eqref{eq:P_estab}, taking $\ell=0$, and the stabilization parameters $\tau_{\text{div}}=0$, $\tau_T=0$, and $\tau_S=1/12$.

~\\
\textbf{Example 4 (L-shaped domain).} We let $\Omega=(0,1)^2 \setminus [0.5,1)\times (0,0.5]$, $p = 1.4$, $\mathcal{D}=\{(0.75,0.75)\}$, and $\boldsymbol{f}_{(0.75,0.75)} = (1,1)$.

We report in Figure \ref{fig:ex_4} the results obtained for Example 4. We present the finite element approximations of $|{\boldsymbol{u}}_\mathscr{T}|$ and $\pi_\mathscr{T}$, experimental rates of convergence for the error estimators $\zeta_{\text{stab},p}$ and adaptively refined meshes. We observe in subfigure (A), that an optimal experimental rate of convergence for the total error estimator $\zeta_{\text{stab},p}$ is attained. We also observe in the adaptively refined meshes (D)--(E), that the refinement is being concentrated around the re--entrant corner and the source point.


\begin{figure}[!ht]
\centering
\psfrag{estimador}{\huge{$\zeta_{\text{stab},p}$}}
\psfrag{O(Ndofs-12)}{\huge$\mathsf{Ndof}^{-1/2}$}
\psfrag{Ndofs}{\Huge{$\mathsf{Ndof}$}}
\psfrag{Example 1 - p14}{\huge{Estimator $\zeta_{\text{stab},p}$ for $p = 1.4$}}
\hspace{-1.0cm}
\begin{minipage}[b]{0.32\textwidth}\centering
\includegraphics[trim={0 0 0 0},clip,width=3.3cm,height=3.3cm,scale=0.66]{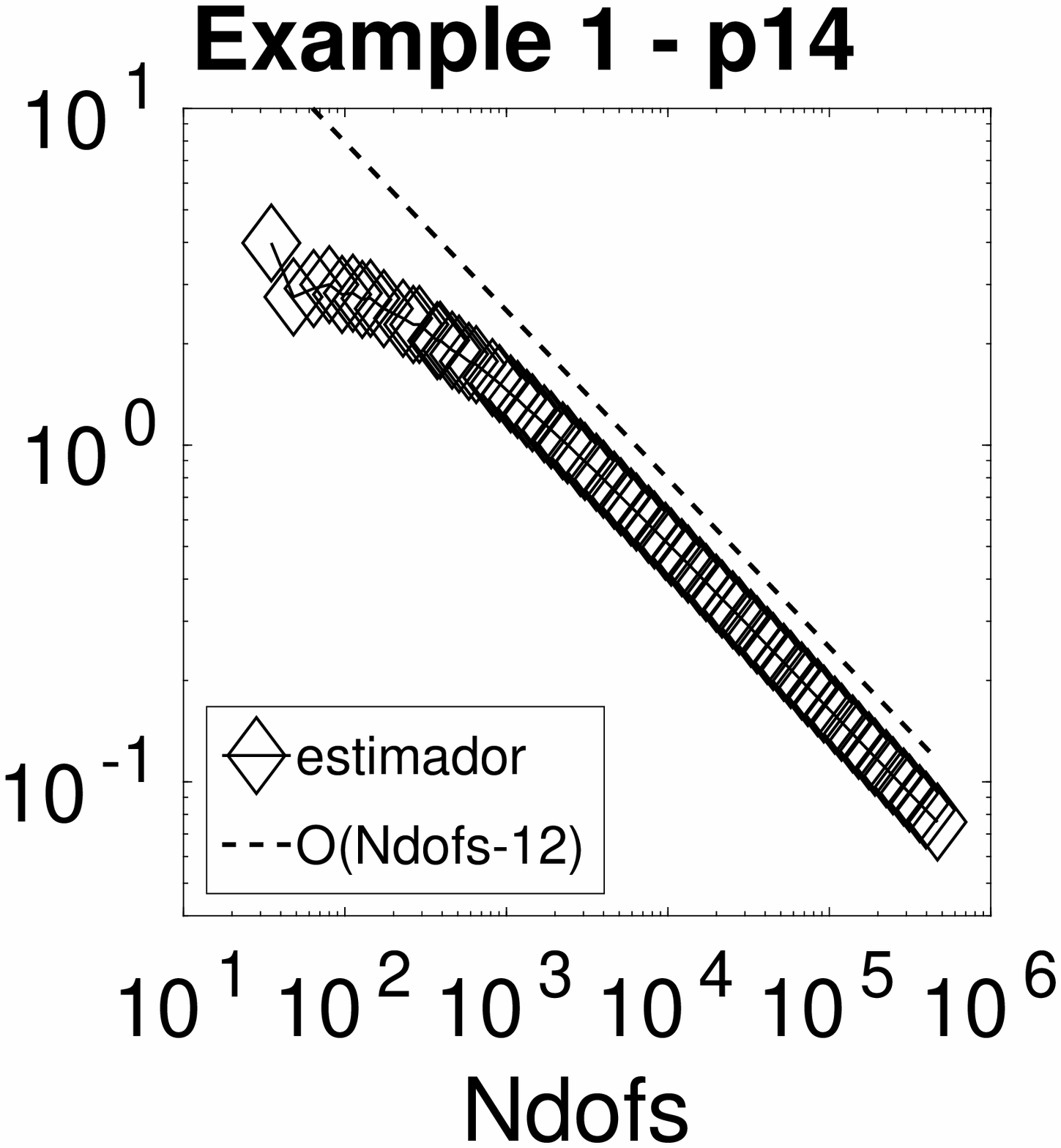} \\
\tiny{(A)}
\end{minipage}
\begin{minipage}[b]{0.27\textwidth}\centering
\scriptsize{$|\boldsymbol{u}_{\mathscr{T}}|$}\\
\includegraphics[trim={0 0 0 0},clip,width=2.8cm,height=2.6cm,scale=0.5]{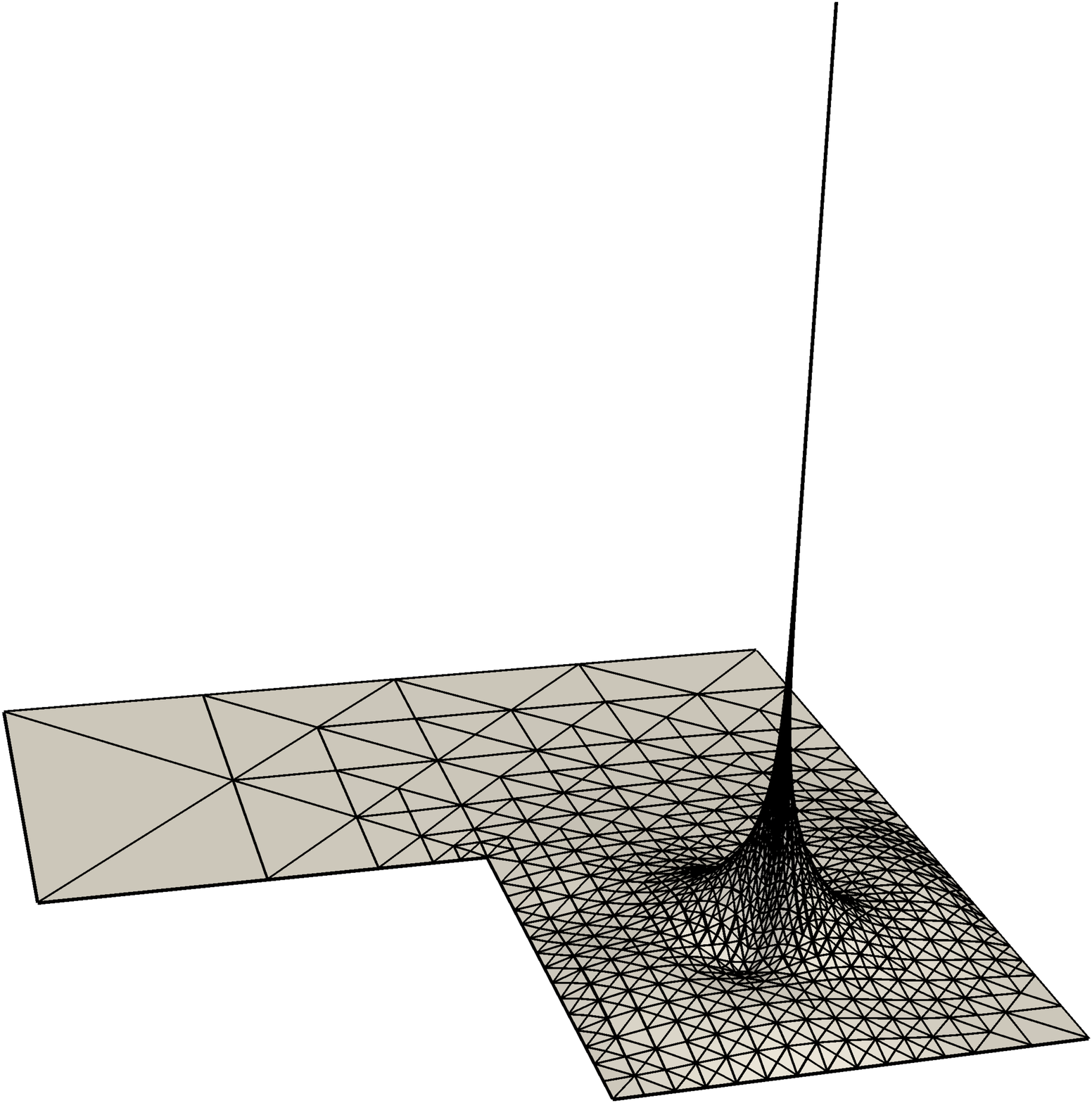} \\
\qquad \tiny{(B)}
\end{minipage}
\begin{minipage}[b]{0.27\textwidth}\centering
\scriptsize{$\pi_{\mathscr{T}}$}\\
\includegraphics[trim={0 0 0 0},clip,width=2.8cm,height=2.6cm,scale=0.5]{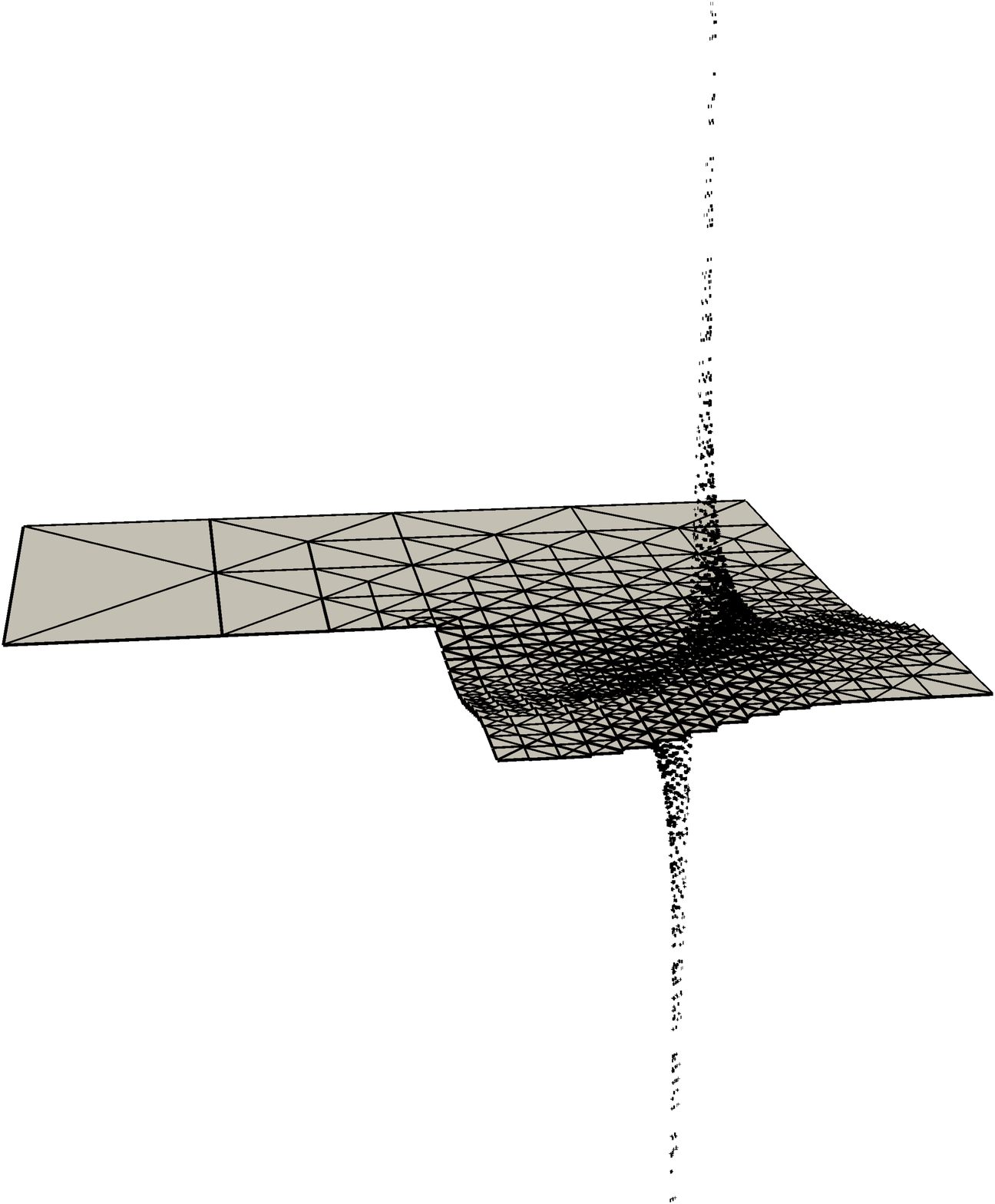} \\
\qquad \tiny{(C)}
\end{minipage}
\\~\\
\begin{minipage}[b]{0.3\textwidth}\centering
\includegraphics[trim={0 0 0 0},clip,width=3.0cm,height=2.8cm,scale=0.5]{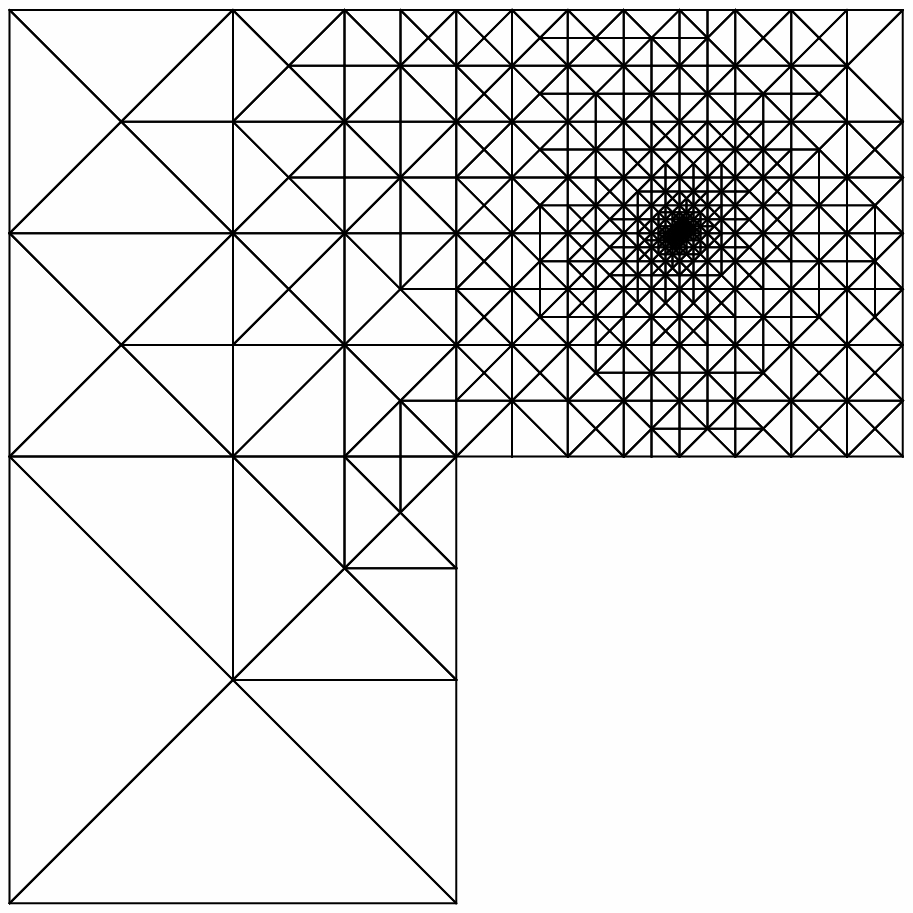} \\
\qquad \tiny{(D)}
\end{minipage}
\begin{minipage}[b]{0.3\textwidth}\centering
\includegraphics[trim={0 0 0 0},clip,width=3.0cm,height=2.8cm,scale=0.5]{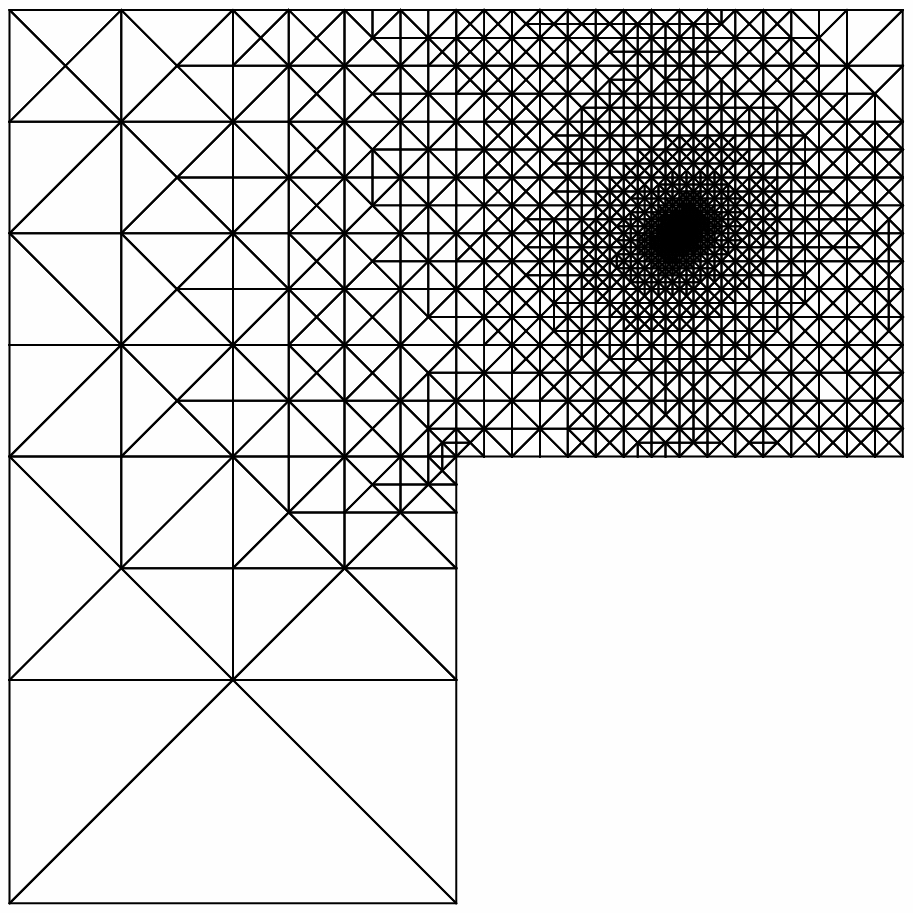} \\
\qquad \tiny{(E)}
\end{minipage}
\caption{Example 4: Experimental rate of convergence for the error estimator $\zeta_{\text{stab},p}$ (A), the finite element approximation of $|\boldsymbol{u}_{\mathscr{T}}|$ (B) and $\pi_{\mathscr{T}}$ (C) obtained on the 35th adaptively refined mesh and the meshes obtained after 30 (D) and 40 (E) iterations of the adaptive loop ($p = 1.4$).}
\label{fig:ex_4}
\end{figure}

~\\
\textbf{Example 5 (Convex domain).} We let $p=1.1$, $\Omega=(0,1)^3$,
\begin{align*}
\mathcal{D} = \left\{ (0.75,0.25,0.5), (0.25,0.75,0.5), (0.75,0.25,0.75), (0.25,0.75,0.75),\right. \\
\left. (0.25,0.25,0.75), (0.25,0.25,0.5), (0.5,0.5,0.5), (0.75,0.75,0.5) \right\},
\end{align*}
and
\begin{align*}
\boldsymbol{f}_{(0.75,0.25,0.5)}&=\boldsymbol{f}_{(0.25,0.75,0.75)}=\boldsymbol{f}_{(0.5,0.5,0.5)}=(1,1,1), \\
\boldsymbol{f}_{(0.25,0.75,0.5)}&=\boldsymbol{f}_{(0.75,0.25,0.75)}=\boldsymbol{f}_{(0.75,0.75,0.5)}=(-5,-5,-5),\\
\boldsymbol{f}_{(0.25,0.25,0.75)}&=(-1,-1,-1), \quad  \boldsymbol{f}_{(0.25,0.25,0.5)}=(5,5,5).
\end{align*}

In Figure \ref{fig:ex-5.1} we report the results obtained for Example 5. We observe in subfigure (A), that an optimal experimental rate of convergence for the total error estimator $\zeta_{\text{stab},p}$ is attained. On the other hand, it is clear in the adaptively refined mesh (B), that the adaptive refinement is mostly concentrated on the points $t\in\mathcal{D}$ where the Dirac measures are supported.

\begin{figure}[!ht]
\centering
\psfrag{error total}{{\huge $\|(\mathbf{e}_{\boldsymbol{u}},e_{\pi})\|_{\mathcal{X}}$}}
\psfrag{estimador}{{\huge $\zeta_{\text{stab},p}$}}
\psfrag{O(Ndofs-1)}{\huge$\mathsf{Ndof}^{-1}$}
\psfrag{O(Ndofs-1222)}{\huge$\mathsf{Ndof}^{-1/3}$}
\psfrag{Ndofs}{{\huge $\mathsf{Ndof}$}}
\psfrag{IEE}{{\huge $\mathscr{I}$}}
\begin{minipage}[c]{0.35\textwidth}
\centering
\psfrag{Example 3 - p11}{\huge{Estimator $\zeta_{\text{stab},p}$ for $p = 1.1$}}
\includegraphics[trim={0 0 0 0},clip,width=3.2cm,height=3.4cm,scale=0.62]{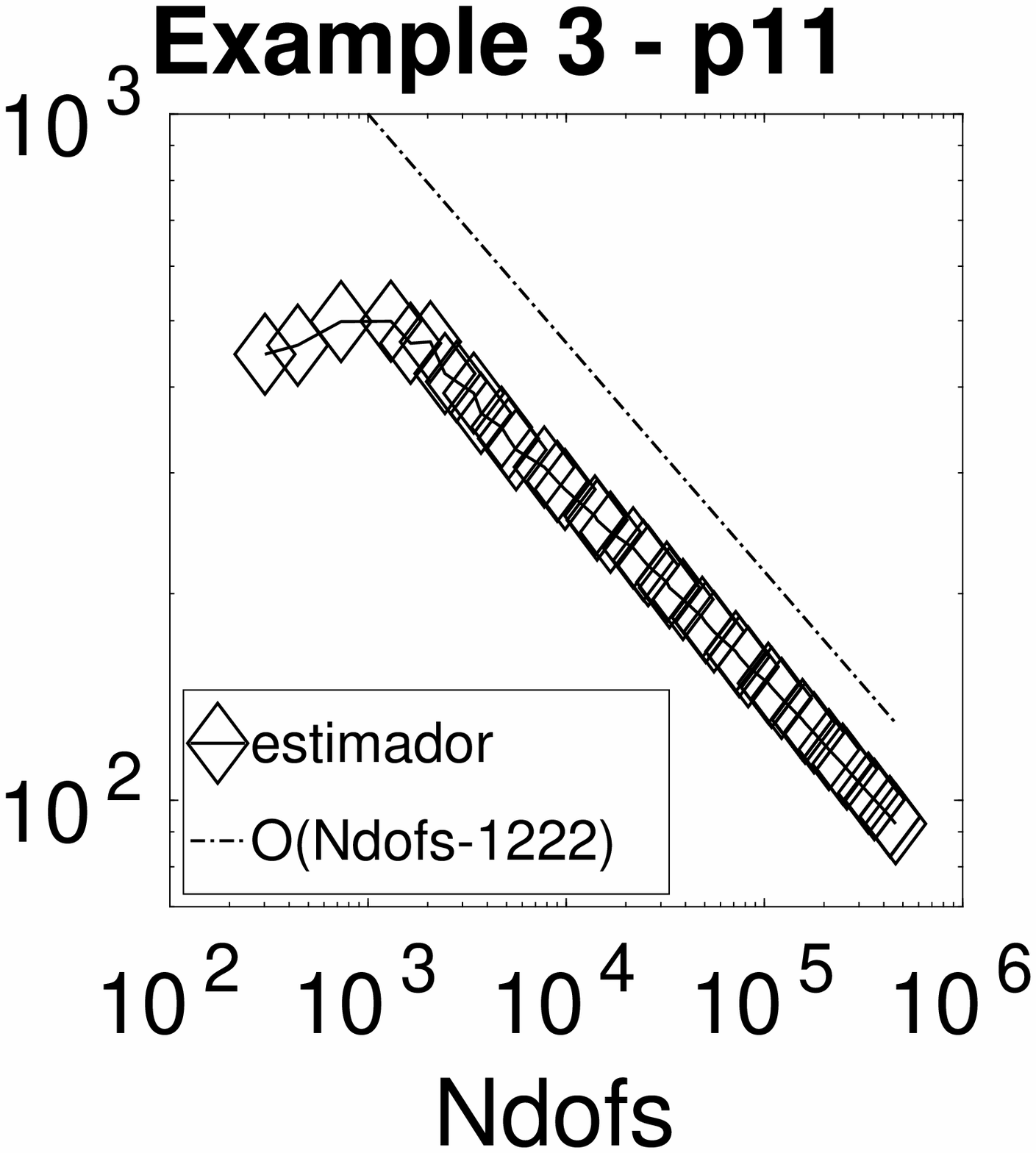} \\
\tiny{(A)}
\end{minipage}
\begin{minipage}[c]{0.35\textwidth}
\centering
\includegraphics[trim={0 0 0 0},clip,width=3.2cm,height=3.4cm,scale=0.57]{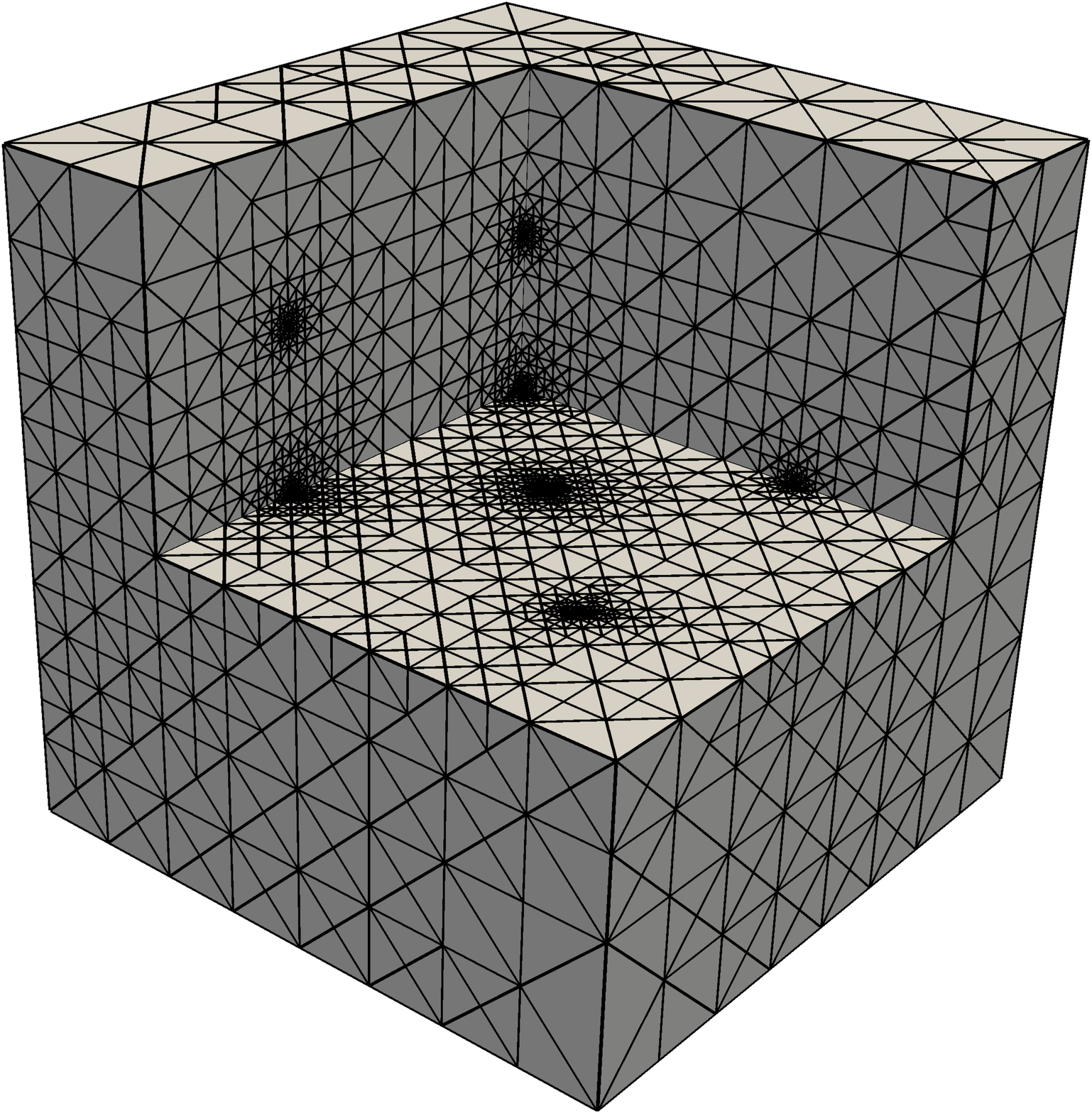} \\
\hspace{-0.8cm}\tiny{(B)}
\end{minipage}
\caption{Example 5: Experimental rates of convergence for the error estimator $\zeta_{\text{stab},p}$ (A) and the 35th adaptively refined mesh (B) ($p = 1.1$).}
\label{fig:ex-5.1}
\end{figure}

\subsection{Conclusions.}\label{sec:conclusions.}

We present the following conclusions.
    
\begin{itemize}
         
\item[$\bullet$] Most of the refinement occurs near to where the Dirac measures are located. This attests to the efficiency of the devised estimators. When the domain involves geometric singularities, refinement is also being performed in regions that are close to them. This shows a competitive performance of the a posteriori error estimators.
 
\item[$\bullet$] The numerical experiments suggest that a small value of $p$ delivers the best results.

\item[$\bullet$] In spite of the very singular nature of the problem \eqref{def:Stokes_singular_rhs}, our proposed estimators are able to deliver optimal experimental rates of convergence within an adaptive loop. 
\end{itemize}



%
%

\footnotesize
\bibliographystyle{siam} 
\bibliography{bib_FLOQ}   

\begin{thebibliography}{10}

\bibitem{MR1885308}
{\sc M.~Ainsworth and J.~T. Oden}, {\em A posteriori error estimation in finite
  element analysis}, Pure and Applied Mathematics (New York),
  Wiley-Interscience [John Wiley \& Sons], New York, 2000.

\bibitem{2018arXiv181002415A}
{\sc A.~Allendes, F.~Fuica, E.~Ot\'{a}rola, and D.~Quero}, {\em An adaptive
  {FEM} for the pointwise tracking optimal control problem of the {S}tokes
  equations}, SIAM J. Sci. Comput., 41 (2019), pp.~A2967--A2998.

\bibitem{MR3892359}
{\sc A.~Allendes, E.~Ot\'{a}rola, and A.~J. Salgado}, {\em A posteriori error
  estimates for the {S}tokes problem with singular sources}, Comput. Methods
  Appl. Mech. Engrg., 345 (2019), pp.~1007--1032.

\bibitem{MR3237857}
{\sc A.~Alonso~Rodr\'{\i}guez, J.~Cama\~{n}o, R.~Rodr\'{\i}guez, and A.~Valli},
  {\em A posteriori error estimates for the problem of electrostatics with a
  dipole source}, Comput. Math. Appl., 68 (2014), pp.~464--485.

\bibitem{MUMPS1}
{\sc P.~Amestoy, I.~Duff, and J.-Y. L'Excellent}, {\em Multifrontal parallel
  distributed symmetric and unsymmetric solvers}, Computer Methods in Applied
  Mechanics and Engineering, 184 (2000), pp.~501 -- 520.

\bibitem{MUMPS2}
{\sc P.~R. Amestoy, I.~S. Duff, J.-Y. L'Excellent, and J.~Koster}, {\em A fully
  asynchronous multifrontal solver using distributed dynamic scheduling}, SIAM
  J. Matrix Anal. Appl., 23 (2001), pp.~15--41 (electronic).

\bibitem{MR2262756}
{\sc R.~Araya, E.~Behrens, and R.~Rodr\'{\i}guez}, {\em A posteriori error
  estimates for elliptic problems with {D}irac delta source terms}, Numer.
  Math., 105 (2006), pp.~193--216.

\bibitem{MR972452}
{\sc C.~Bernardi, C.~Canuto, and Y.~Maday}, {\em Generalized inf-sup conditions
  for {C}hebyshev spectral approximation of the {S}tokes problem}, SIAM J.
  Numer. Anal., 25 (1988), pp.~1237--1271.

\bibitem{MR3854357}
{\sc S.~Bertoluzza, A.~Decoene, L.~Lacouture, and S.~Martin}, {\em Local error
  analysis for the {S}tokes equations with a punctual source term}, Numer.
  Math., 140 (2018), pp.~677--701.

\bibitem{blake_1972}
{\sc J.~Blake}, {\em A model for the micro-structure in ciliated organisms},
  Journal of Fluid Mechanics, 55 (1972), p.~1–23.

\bibitem{MR2490235}
{\sc P.~B. Bochev and M.~D. Gunzburger}, {\em Least-squares finite element
  methods}, vol.~166 of Applied Mathematical Sciences, Springer, New York,
  2009.

\bibitem{MR2373954}
{\sc S.~C. Brenner and L.~R. Scott}, {\em The mathematical theory of finite
  element methods}, vol.~15 of Texts in Applied Mathematics, Springer, New
  York, third~ed., 2008.

\bibitem{MR1386766}
{\sc R.~M. Brown and Z.~Shen}, {\em Estimates for the {S}tokes operator in
  {L}ipschitz domains}, Indiana Univ. Math. J., 44 (1995), pp.~1183--1206.

\bibitem{stephen81}
{\sc S.~Childress}, {\em Mechanics of swimming and flying}, vol.~2 of Cambridge
  Studies in Mathematical Biology, Cambridge University Press, Cambridge-New
  York, 1981.

\bibitem{2019arXiv190500476D}
{\sc R.~G. {Dur\'an}, E.~{Ot\'arola}, and A.~J. {Salgado}}, {\em {Stability of
  the Stokes projection on weighted spaces and applications}}, arXiv e-prints,
  (2019), p.~arXiv:1905.00476.

\bibitem{MR2050138}
{\sc A.~Ern and J.-L. Guermond}, {\em Theory and practice of finite elements},
  vol.~159 of Applied Mathematical Sciences, Springer-Verlag, New York, 2004.

\bibitem{2019arXiv190711096F}
{\sc F.~{Fuica}, E.~{Ot\'arola}, and D.~{Quero}}, {\em {Error estimates for
  optimal control problems of the Stokes system with Dirac measures}}, arXiv
  e-prints,  (2019), p.~arXiv:1907.11096.

\bibitem{FULFORD1986381}
{\sc G.~Fulford and J.~Blake}, {\em Muco-ciliary transport in the lung},
  Journal of Theoretical Biology, 121 (1986), pp.~381 -- 402.

\bibitem{MR2808162}
{\sc G.~P. Galdi}, {\em An introduction to the mathematical theory of the
  {N}avier-{S}tokes equations}, Springer Monographs in Mathematics, Springer,
  New York, second~ed., 2011.
\newblock Steady-state problems.

\bibitem{MR851383}
{\sc V.~Girault and P.-A. Raviart}, {\em Finite element methods for
  {N}avier-{S}tokes equations}, vol.~5 of Springer Series in Computational
  Mathematics, Springer-Verlag, Berlin, 1986.
\newblock Theory and algorithms.

\bibitem{hasimoto59}
{\sc H.~Hasimoto}, {\em On the periodic fundamental solutions of the {S}tokes'
  equations and their application to viscous flow past a cubic array of
  spheres}, J. Fluid Mech., 5 (1959), pp.~317--328.

\bibitem{higdon79}
{\sc J.~J.~L. Higdon}, {\em The generation of feeding currents by flagellar
  motions}, J. Fluid Mech., 94 (1979), pp.~305--330.

\bibitem{MR3561143}
{\sc V.~John}, {\em Finite element methods for incompressible flow problems},
  vol.~51 of Springer Series in Computational Mathematics, Springer, Cham,
  2016.

\bibitem{MR1740398}
{\sc D.~Kay and D.~Silvester}, {\em A posteriori error estimation for
  stabilized mixed approximations of the {S}tokes equations}, SIAM J. Sci.
  Comput., 21 (1999/00), pp.~1321--1336.

\bibitem{LACOUTURE2015187}
{\sc L.~Lacouture}, {\em A numerical method to solve the {S}tokes problem with
  a punctual force in source term}, Comptes Rendus M\'ecanique, 343 (2015),
  pp.~187--191.

\bibitem{Larsson_Svensson}
{\sc S.~Larsson and E.~D. Svensson}, {\em Pointwise a posteriori error
  estimates for the {S}tokes equations in polyhedral domains}, Preprint,
  (2006).

\bibitem{liron78}
{\sc N.~Liron and R.~Shahar}, {\em Stokes flow due to a {S}tokeslet in a pipe},
  J. Fluid Mech., 86 (1978), pp.~727--744.

\bibitem{MR2987056}
{\sc M.~Mitrea and M.~Wright}, {\em Boundary value problems for the {S}tokes
  system in arbitrary {L}ipschitz domains}, Ast\'{e}risque,  (2012),
  pp.~viii+241.

\bibitem{MR2648380}
{\sc R.~H. Nochetto, K.~G. Siebert, and A.~Veeser}, {\em Theory of adaptive
  finite element methods: an introduction}, in Multiscale, nonlinear and
  adaptive approximation, Springer, Berlin, 2009, pp.~409--542.

\bibitem{MR2454024}
{\sc H.-G. Roos, M.~Stynes, and L.~Tobiska}, {\em Robust numerical methods for
  singularly perturbed differential equations}, vol.~24 of Springer Series in
  Computational Mathematics, Springer-Verlag, Berlin, second~ed., 2008.
\newblock Convection-diffusion-reaction and flow problems.

\bibitem{MR993474}
{\sc R.~Verf\"{u}rth}, {\em A posteriori error estimators for the {S}tokes
  equations}, Numer. Math., 55 (1989), pp.~309--325.

\bibitem{MR1650051}
{\sc R.~Verf\"{u}rth}, {\em A posteriori error estimators for
  convection-diffusion equations}, Numer. Math., 80 (1998), pp.~641--663.

\bibitem{MR3059294}
{\sc R.~Verf\"{u}rth}, {\em A posteriori error estimation techniques for finite
  element methods}, Numerical Mathematics and Scientific Computation, Oxford
  University Press, Oxford, 2013.

\end{thebibliography}

\end{document}